\documentclass[nonblindrev]{informs3}
\RequirePackage{tgtermes}
\RequirePackage{newtxtext}
\RequirePackage{bm}
\RequirePackage{endnotes}

\OneAndAHalfSpacedXI

\makeatletter
\let\normalsize\relax
\let\@currsize\normalsize
\makeatother

\usepackage{tikz}
\usepackage{longtable}
\usepackage{tabularx}
\usepackage{booktabs}
\usepackage{etex}
\usepackage{multirow}
\usepackage{bbold}
\usepackage{amsmath}
\usepackage{mathtools}
\usepackage{subfig}
\usepackage{array}
\usepackage{hyperref}
\usepackage{cleveref}
\usepackage{natbib}
\usepackage{courier}

\usepackage{xcolor}
\definecolor{myblue}{RGB}{35, 76, 143} 

\hypersetup{
    colorlinks=true,
    linkcolor=myblue, 
    citecolor=myblue, 
    filecolor=magenta,     
    urlcolor=black          
}

\usepackage{algorithm}
\usepackage{algorithmicx}
\usepackage{capt-of}
\usepackage[noend]{algpseudocode}
\usepackage{amsbsy, amsfonts, amsgen, amsmath, amsopn, amssymb, amstext, amsxtra, bezier, color, enumerate, graphicx, latexsym, verbatim, pictexwd, supertabular, url, dsfont,leftidx, mathrsfs, appendix,setspace}
\usepackage{pbox}
\usepackage{caption}
\usepackage{tablefootnote}
\usepackage{adjustbox}
\usepackage{threeparttable}

\usepackage{setspace}

\newtheorem{assu}{Assumption}[section]

\usepackage[normalem]{ulem}
\newcommand{\ins}[1]{\textcolor{black}{{#1}}} 
\newcommand{\delete}[1]{\textcolor{red}{\sout{#1}}} 
\newcommand{\cmtY}[1]{\textcolor{red}{(Yixing: {#1})}} 
    \newcommand{\cmtA}[1]{\textcolor{purple}{#1}} 

\usepackage{endnotes}
\let\footnote=\endnote

\algblockdefx{For}{EndFor}[1]{\algorithmicfor\ #1}{\algorithmicend\ \algorithmicfor}
\algblockdefx{While}{EndWhile}[1]{\algorithmicwhile\ #1}{\algorithmicend\ \algorithmicwhile}

\usepackage{natbib}
 \bibpunct[, ]{(}{)}{,}{a}{}{,}%
 \def\bibfont{\small}%
\EquationsNumberedThrough    

\TheoremsNumberedThrough     
\ECRepeatTheorems  %

\MANUSCRIPTNO{IJOC-0001-2024.00}

\begin{document}


\RUNAUTHOR{Hua et al.}

\RUNTITLE{Geometrically Convergent Spatial Hypercube Queue}

\TITLE{A Geometrically Convergent Solution to Spatial Hypercube Queueing Models}

\ARTICLEAUTHORS{%
\AUTHOR{Cheng Hua$^\ast$, Jun Luo$^{\ast\ddagger}$, Arthur J. Swersey$^\dagger$, Yixing Wen$^\ast$}
\AFF{$^\ast$Antai College of Economics and Management, Shanghai Jiao Tong University, Shanghai 200030, 
\EMAIL{\{cheng.hua, jluo\_ms, star-winmex\}@sjtu.edu.cn}}
\AFF{$^\ddagger$Data-Driven Management Decision Making Lab, Shanghai Jiao Tong University, Shanghai 200030}
\AFF{$^\dagger$Yale School of Management, Yale University, New Haven, CT 06511, \EMAIL{arthur.swersey@yale.edu}}

} 

\ABSTRACT{%
The hypercube queueing model was initially developed to address spatial queueing problems and has found wide applications in emergency services, such as ambulance and police systems. While the model was originally designed for homogeneous service rates, we extend it to handle heterogeneous service rates by devising an exact solution through a birth-death process and an equivalent reformulation. We demonstrate that our algorithm converges to the exact solution at a geometric rate. Additionally, we developed a parallel algorithm that leverages the convergence property and two structural features of the hypercube model, achieving more than \ins{91\%} parallelization. \ins{Numerical experiments on emergency medical service systems show that our sequential algorithm is over 1,000 times faster than the sparse solver and more than 500 times faster than discrete-event simulation, while maintaining high accuracy. The parallel algorithm further improves efficiency, achieving an approximately eightfold speedup with 12 processing units, with additional gains possible when more computational resources are available. Overall, the proposed algorithms improve computational efficiency and enable the solution of large-scale problems that are otherwise intractable using traditional approaches.} 
}%





\maketitle


\section{Introduction} \label{sec_intro}
Spatial characteristics play a crucial role in emergency service systems, particularly in the deployment of service units and the formulation of resource dispatching policies \citep{maxwell2010approximate, jenkins2021approximate}. 
The spatial hypercube queueing model, introduced by \cite{larson74hypercube} about 50 years ago, was developed to study the spatial distribution of emergency service units. Since then, it has been widely applied to analyze emergency service systems and has inspired numerous extensions. 
Our paper improves the performance of exact solution methods for the hypercube model. In addition, we extend the hypercube to allow for service times that vary by server and propose a solution that guarantees geometric convergence. 

The spatial hypercube model and its extensions have been used by many cities to study the deployment of police cars, ambulances, and other emergency services. For instance, \cite{mcewen1974patrol} used the hypercube to deploy police patrol cars in Rotterdam, Netherlands. \cite{chelst1975quantitative} applied the hypercube to the deployment of police cars in New Haven, CT. \cite{berman1982median} developed approximations for the stochastic $p$-median problem for emergency service systems based on the hypercube model, while \cite{mendoncca2001analysing}, \cite{iannoni2007multiple}, and \cite{iannoni2009optimization} applied the hypercube to locate emergency medical response units on a Brazilian highway. More recently, \cite{hua2022cross} developed both exact and approximate methods to analyze a three-state system where cross-trained personnel respond to both emergency medical calls and fire incidents. Additionally, \cite{beojone2021efficient} proposed an exact model to study the dispatching policy of a system with dedicated servers,  while \cite{zhu2022data} presented a data-driven police patrol zone design framework that combines the approximate method with mixed-integer linear programming to balance the workload of police cars at the Atlanta Police Department. \cite{iannoni2023review} provided a comprehensive review of the application of hypercube models to various real-world problems. Interested readers may refer to \cite{ghobadi2019hypercube} for a comprehensive review of these models, and to \cite{larson2013hypercube} for a brief introduction to the history, development, and applications of the hypercube model.

The spatial hypercube model \citep{larson74hypercube} is distinguished by its ability to model \textit{server-to-customer} systems, where servers travel to customer locations to fulfill service requests, as described in the aforementioned applications. This model accounts for the spatially distributed locations of servers and the randomness associated with their availability due to incoming calls. In contrast, \textit{customer-to-server} systems involve customers traveling to server locations. A key characteristic of server-to-customer systems is that servers are spread out, making it difficult to model the system in an aggregate manner and resulting in a larger state space.

The original hypercube model provides a solution by solving $2^N$ simultaneous equations arising from an $N$-unit system. The system is described by a vector of $N$ elements, where each element is either 0 (free) or 1 (busy), representing the status of each unit. Solving the steady-state balance equations yields the probability distribution of system states, which can be used to derive key performance measures, such as unit workload and region-wide response times \citep{alanis2013markov, rastpour2020modeling} to calls for service in emergency service systems.
The state-of-the-art solution method for solving the homogeneous service rate problem is the alternating hyperplane approach developed by \cite{larson74hypercube}, with the user’s manual provided in \cite{larson1975hypercube}. In this method, each hyperplane is defined as the set of states with the same number of busy servers. The approach alternates between updating probabilities for even- and odd-numbered hyperplanes, using the most recent estimates for each equation. We refer to this method throughout the paper as the original Larson method. 

\ins{To the best of our knowledge, explicit convergence rate guarantees and exact solution methods for systems with heterogeneous service rates are limited in the existing literature. Motivated by this, we develop an exact approach and incorporate parallel computing techniques to better accommodate larger problem instances.} \ins{Throughout the paper, we use the sparse solver \texttt{spsolve} as a benchmark, which provides exact solutions but is computationally demanding.}



Approximation methods have been developed to solve spatial hypercube queueing models. \cite{larson75approx} developed an approximation method to deal with large systems, in which the number of equations is reduced from $2^N$ to $N$. \cite{chelst81multiple} extended the model to allow for the dispatch of multiple units. \cite{jarvis85approximating} extended Larson's approximation method to include multiple customer types and general service time distributions. \cite{budge09technical} extended the approximation method to allow for multiple units at each location, workload variation by station, and call- and station-dependent service times. These approximation techniques incorporate correction factors to adjust for the lack of independence among servers, which are analyzed using the Erlang loss function. A major issue with these approximation methods is their limited convergence and performance guarantees. To assess the accuracy of these approximation techniques, exact solutions are needed. 


    A core contribution of this paper is the development of a new approach for solving the spatial hypercube queueing model with heterogeneous service rates, and proving that it converges to the exact solution at a geometric rate. The proof relies on a technique that bounds the differences between consecutive iterations. The result also extends to homogeneous cases. To the best of our knowledge, this is the first study to establish convergence and characterize its rate for both heterogeneous and homogeneous settings.
    
    We also develop a parallel algorithm based on the convergence result and two key properties: \textit{non-dependency within layer} and \textit{local-dependency between layers}, which we discuss in detail in Section~\ref{sec_parallel}. We propose a new coefficient generation algorithm that is faster and compatible with the parallel algorithm. This parallel algorithm substantially reduces computation time, allowing us to solve problems with more than 25 emergency units, which were previously computationally prohibitive for exact methods at this scale. We demonstrate that our algorithm achieves 91\% parallelization, according to Amdahl's law \citep{amdahl1967validity}.
    
    We conduct numerical experiments using data from two emergency medical service systems: St. Paul, Minnesota, and  Greenville County, South Carolina. Our \ins{sequential} algorithm reduces computation time by more than a thousand times compared to the sparse solver \ins{for problems with 15 or more units}.  \ins{The sparse solver requires over 150 minutes to solve the 15-unit case.} \ins{Moreover, our algorithm outperforms discrete-event simulation in both accuracy and efficiency, running over 500 times faster while achieving higher accuracy. The parallel algorithm further accelerates computation by more than eightfold for problems with more than 20 units using 12 processing units, and this gain would increase further with additional computing resources.} 
    We have open-sourced both our algorithm and a reimplementation of Larson's alternating hyperplane method.


The remainder of this paper is organized as follows. Section \ref{sec_model} outlines the spatial queueing problem and the hypercube model formulation. 
In Section \ref{sec_convergence}, we develop an algorithm and prove its geometric convergence to the exact solution. Section \ref{sec_parallel} details the parallel computing framework for our approach.  Numerical results are presented in Section \ref{sec_numerical}. Finally, Section \ref{sec_conclusion} concludes the paper. Proofs and additional numerical experiments are provided in the E-Companion.

\section{The Spatial Hypercube Queueing Model}
\label{sec_model}
 In this section, we begin by introducing the spatial hypercube queueing model in Section \ref{ssec_mode_basic}, and then present our reformulated model based on a birth-death process and demonstrate its equivalence to the original model in Section \ref{subsec:reformulations}.

\subsection{\ins{Model}} \label{ssec_mode_basic}
Consider an emergency service system operating in a geographical region partitioned into $J$ demand nodes.  There are $N$ emergency units, all of the same type, e.g., police or ambulance units. 
Region-wide calls for service arrive according to a Poisson process with rate $\lambda$, where $f_j$ is the fraction of calls at each node $j$. For each call at node $j$, a unit is dispatched, the one with the highest preference in a fixed preference list associated with that node. We denote $\zeta_{j}(h)$ as the $h$-th preferred unit for node $j$. If all units are busy, the call is queued, and we assume calls are served on a first-come, first-served basis. When the number of queued calls exceeds a buffer capacity $C$, the call is lost. The service time for each call is exponential with mean $1/\nu_i$ for unit $i$, a generalization of the original hypercube model, which assumes that the service rates are the same across all units.

For clarity of exposition, we first focus on the special case of $C=0$, as the core idea is included in this case. 
The more general case, where $C  > 0$, is discussed immediately afterward. We denote the state of the system as an $N$-dimensional vector $B_m=[b_{m,N},\ldots, b_{m,1}]$, which is written in reverse order, similar to the representation of binary numbers. Let $b_{m,i}$ represent the state of unit $i$ for state $B_m$, where $b_{m,i} = 0$ if the unit is available and $b_{m,i} = 1$ if the unit is busy.  We define $B_m(i)=b_{m,i}$ as the status of unit $i$ in state $B_m$. The state space is given by $\mathscr{B} = \{B_m: m=0,1\ldots,2^N-1\}$. Throughout the paper, we denote $\mathbb{1}\{\cdot\}$ as the indicator function. We summarize the notations in Table \ref{tab_notation}.

\begin{table}[tbp]
\renewcommand\arraystretch{1.2}
  \centering
  \caption{Summary of Symbols and Notations.}
  \begin{footnotesize}
    \begin{tabular}{c p{13cm}}
    \hline
        \multicolumn{2}{c}{Model}\\
        \hline
          $N$ & Total number of units.\\
          $J$ & Total number of demand nodes. \\
          $\zeta_{j}$ & Preference list of node $j$. $\zeta_{j} = \left[\zeta_{j}(1),  \zeta_{j}(2), \ldots, \zeta_{j}(N)\right]$ is a permutation of $\{1,2,\ldots,N\}$, where $\zeta_{j}(h)$ is the $h$-th preferred unit for calls at node $j$.\\
          $\eta_j(i)$ & The preference ranking of unit $i$ at node $j$.  $\eta_j(i)=h$ if and only if $\zeta_{j}(h) = i$.\\
          $\lambda$ & Overall arrival rate of calls. \\
          $f_j$ & Fraction of region-wide demand generated from demand node $j$. \\
          $\nu_i$ & Service rate of unit $i$.\\
          $C$ & Queue capacity.\\
          \hline
          \multicolumn{2}{c}{States and Transitions}\\
          \hline
          $B_m$ &  State $B_m=\left[b_{m,N}, \ldots, b_{m,1}\right]$, where $b_{m,i}$ is the status of unit $i$ in state $B_m$.\\
          $\mathscr{B}$ & State space of all states.\\
          $w(B_m)$ & Number of busy units in state $B_m$, $w(B_m) = \sum_i b_{m,i}$.\\
          $d_{ml}^+/d_{ml}^-$ & Upward/Downward Hamming distance from state $B_m$ to state $B_l$.\\
          $\mathscr{C}_n$ & Set of all states that have $n$ busy units, also referred to as layer $n$.\\
          $\mathscr{C}_n^l$ & Set of all states that have $n$ busy units and Hamming distance $1$ to state $l$.\\
          $\lambda_{ml}$ & Transition rate from $B_m$ to $B_l$ by an arrival, given by Equation \eqref{eq:tran_rates}.\\ 
          $\mu_{ml}$ & Transition rate from $B_m$ to $B_l$ by a service completion, given by Equation \eqref{eq:tran_rates}.\\
          $\lambda_{m}$ & Upward transition rate from state $B_m$, $\lambda_{m} = \sum_{l\in \mathscr{C}_{n+1}} \lambda_{ml}$.\\ 
          $\mu_m$ & Downward transition rate from state $B_m$, $\mu_{m} = \sum_{l\in \mathscr{C}_{n-1}} \mu_{ml}.$\\
          $\lambda(n)$ & Upward transition rate of the birth-death process for layer $n$, given by Equation \eqref{eq:lambda}.\\
          $\mu(n)$ & Downward transition rate of the birth-death process for layer $n$, given by Equation \eqref{eq:mu}.\\ 
          \hline
          \multicolumn{2}{c}{State Probabilities}\\
        \hline
        $P\{B_m\}$ & Steady-state probability that the system is in state $B_m$.\\
        $P\{n,B_m\}$ & Joint probability that the system is in state $B_m$ and $n$ units are busy. \\
        $p_{n}\left(B_m\right)$ & Conditional probability of being in state $B_m$ given $n$ units busy, $p_{n}\left(B_m\right) = P\{B_m \big| n \ \text{busy units}\}$.\\
        $p(n)$ & Probability distribution of the birth-death process for layer $n$, given by Equation \eqref{eq:bnd_stationary}.\\
    \hline
    \end{tabular}%
    \end{footnotesize}
  \label{tab_notation}%
\end{table}%

\begin{figure}[b!]
\begin{minipage}{.5\linewidth}
\centering
\subfloat[3 Units]{\label{main:a}\includegraphics[scale=.3]{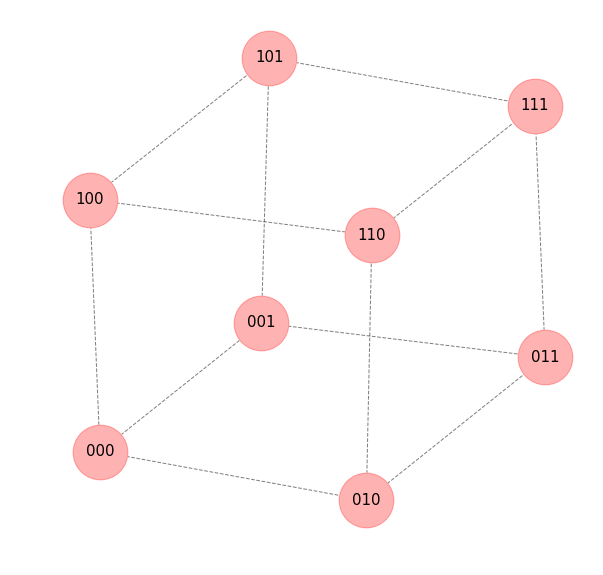}}
\end{minipage}%
\begin{minipage}{.5\linewidth}
\centering
\subfloat[5 Units]{\label{main:b}\includegraphics[scale=.3]{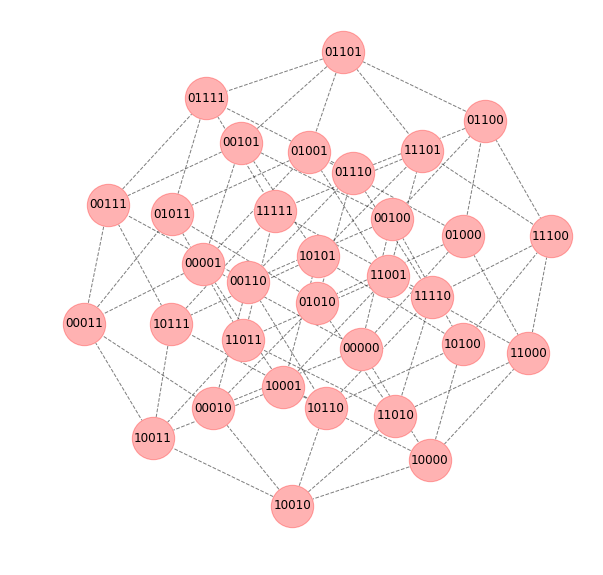}}
\end{minipage}
\caption{The States and Transitions of Spatial Hypercube Queueing Models.}
{\footnotesize 
    \linespread{1.0}\selectfont 
    \noindent\textit{Note.} An $N$-unit system can be modeled as an $N$-dimensional hypercube, where the vertices represent the states $B_m$, and the edges represent the possible transitions (both upward and downward) between states. For example, the state 110 in (a) represents two busy units and one free unit, and it can transition to three other states due to call arrivals or service completions. 
    \par}
\label{fig_hypercube}
\vspace{-0.3cm}
\end{figure}

A transition occurs when a call arrives at a particular node or a unit returns to service. For an $N$-unit system, the vertex of an $N$-dimensional hypercube represents each state.  The diagram in Figure \ref{fig_hypercube} shows the states and transitions of the hypercube for systems with three and five units. For an $N$-unit system, the steady-state balance equations specify that for each of the $2^N$ states, the rate of transitions out of each state must equal the rate of transitions into that state. To guarantee a probability distribution requires that the probabilities sum to one. Adding this condition and eliminating one balance equation as redundant yields the following system of equations, which we solve to find the steady-state distribution of system states. For $m=0,1\ldots, 2^N-1$, 
\begin{eqnarray}
&& P\left\{B_{m}\right\}\Big(\overbrace{\lambda_{m}}^{\text{arrival}}+\overbrace{\sum_{i=1}^N b_{m,i} \nu_i}^{\text{service completion}} \Big)= \overbrace{\sum_{l : d_{lm}^+=1} P\left\{B_{l}\right\} \lambda_{lm}}^{\text{upward transition}}+\overbrace{\sum_{l: d_{lm}^-=1} P\left\{B_{l}\right\}\mu_{lm}}^{\text{downward transition}}, \label{eq:balance_eq_heter}\\ 
&& \sum_{m=0}^{2^N-1} P\{B_m\} = 1, \label{eq:sumto1_intial}
\end{eqnarray}
where $d_{lm}^+$ and $d_{lm}^-$ are upward and downward Hamming distances, respectively, as defined in \cite{larson74hypercube}. Briefly, the upward (downward) Hamming distance is the minimum number of substitutions of 0 to 1 (1 to 0) required to transition from one state to another.
The left-hand side of \eqref{eq:balance_eq_heter} represents the rate of transitions out of state $B_m$, either by an arrival at rate $\lambda_m$ or by a service completion at rate $\sum_{i=1}^N b_{m,i} \nu_i$. Specifically, $\lambda_m=0$ if all units in state $B_m$ are busy; otherwise, we have $\lambda_m=\lambda$. The right-hand side represents all the states that transition into state $B_m$ with rate $\lambda_{lm}$ by an arrival, or transition out of state $B_l$ with rate $\mu_{lm}$ following a service completion. Specifically, we have
\begin{equation}
\begin{aligned}
    \lambda_{lm} &= \sum_{j=1}^J \lambda f_j\prod_{r=1}^{\eta_{j}(i_{lm})} \mathbb{1} \{B_l(\zeta_{j}(r)) \not= 0\}, \\
    \mu_{lm} &= \nu_{i_{lm}},
\end{aligned}
\label{eq:tran_rates}
\end{equation}
where $i_{lm}$ denotes the index of the unit that differs between states $B_l$ and $B_m$ when the Hamming distance between these two states is 1. The term $\eta_j\left(i_{lm}\right)$ represents the preference ranking of unit $i_{lm}$ at node $j$. Specifically, $\eta_j(i) = h$ if and only if $\zeta_j(h) = i$, meaning that $\eta_j(i)=\zeta_j^{-1}(i)$. 
\color{black}
Briefly, $\lambda_{lm}$ is the sum of arrival rates across all demand nodes $j$ that have unit $i_{lm}$ as the available unit with the highest preference. 
Solving the linear system provides the probability distribution of system states, which is then used to compute system performance metrics. 


\subsection{Model Reformulation}
\label{subsec:reformulations}

In this section, we reformulate the spatial hypercube queueing model, which serves as the foundation for the algorithm introduced in the following section. As an overview, we aggregate states with the same number of busy servers into a super state, thereby transforming the model into a birth-death structure. This reformulation allows us to express the stationary distribution in terms of conditional probabilities using the birth-death process.

We begin with defining $\mathscr{C}_n=\{m: w(B_m)=n, m=0,1\ldots,2^N-1\}$ as the set of all states with $n$ busy units, where $w(B_m)=\sum_{i=1}^N b_{m,i}$ is the number of units busy in state $B_m$. We also define $\mathscr{C}_n^l=\{m: m \in \mathscr{C}_n, d_{ml}^-=1\}$ as the set of states that transition down to state $B_l$ with a service completion. For this formulation, the state space is given by the set of vectors $\mathscr{C} = \{(n, B_m): m\in \mathscr{C}_n, n=0,1\ldots, N\}$. This formulation results in a two-dimensional Markov chain, as shown in the upper part of Figure \ref{fig_BnD}. 

The balance equations under this formulation, with the additional requirement that the state probabilities sum to one, are given, for $n=0,1\ldots,N$ and for $m\in \mathscr{C}_n$, by
\begin{eqnarray}
    &&P\left\{n,B_{m}\right\}\Big(\overbrace{\lambda_{m}}^{\text{arrival}}+ \overbrace{ \sum_{i=1}^N b_{m,i} \nu_i}^{\text{service completion}}\Big)= \overbrace{\sum_{l \in \mathscr{C}_{n-1}} P\left\{n-1,B_{l}\right\} \lambda_{l m}}^{\text{upward transition}}+\overbrace{\sum_{l \in \mathscr{C}^m_{n+1}} P\left\{n+1,B_{l}\right\} \mu_{lm}}^{\text{downward transition}}. \label{eq:transformer_bal_eq_heter}\\
    &&\sum_{n=0}^{N} \sum_{m \in \mathscr{C}_n} P\left\{n, B_m\right\} =1. \label{eq:sumto1}
\end{eqnarray}
Note that, by definition in \eqref{eq:tran_rates}, $\lambda_{lm} = \mu_{lm} = 0$ whenever the Hamming distance between two states $B_l$ and $B_m$ exceeds one. Therefore, Equation \eqref{eq:transformer_bal_eq_heter} is equivalent to Equation \eqref{eq:balance_eq_heter}. Similar to the hypercube formulation, both outward and inward flows arise in two ways: receiving a service call and completing a service. This two-dimensional formulation yields the same probability distribution of system states as the hypercube formulation. In this formulation, $P\left\{n, B_m\right\} = P\left\{B_m\right\}$ when $n = w\left(B_m\right)$; otherwise, $P\left\{n, B_m\right\} = 0$.

The spatial hypercube queueing model is a server-to-customer system, where units are dispatched based on a preference list. When calls are randomly assigned to any available unit with equal probability, the hypercube model simplifies into a finite birth-death process. To exploit the simplicity of the birth-death process, we combine its solution with hypercube transitions to capture the probability distribution of system states, while considering the server preferences in the transition probabilities. We begin by grouping states with the same number of busy servers into a single super state. These super states together form a birth-death process, as illustrated in Figure \ref{fig_BnD}. 

\begin{figure}[t]
\begin{center}
\includegraphics[width=9.0cm]{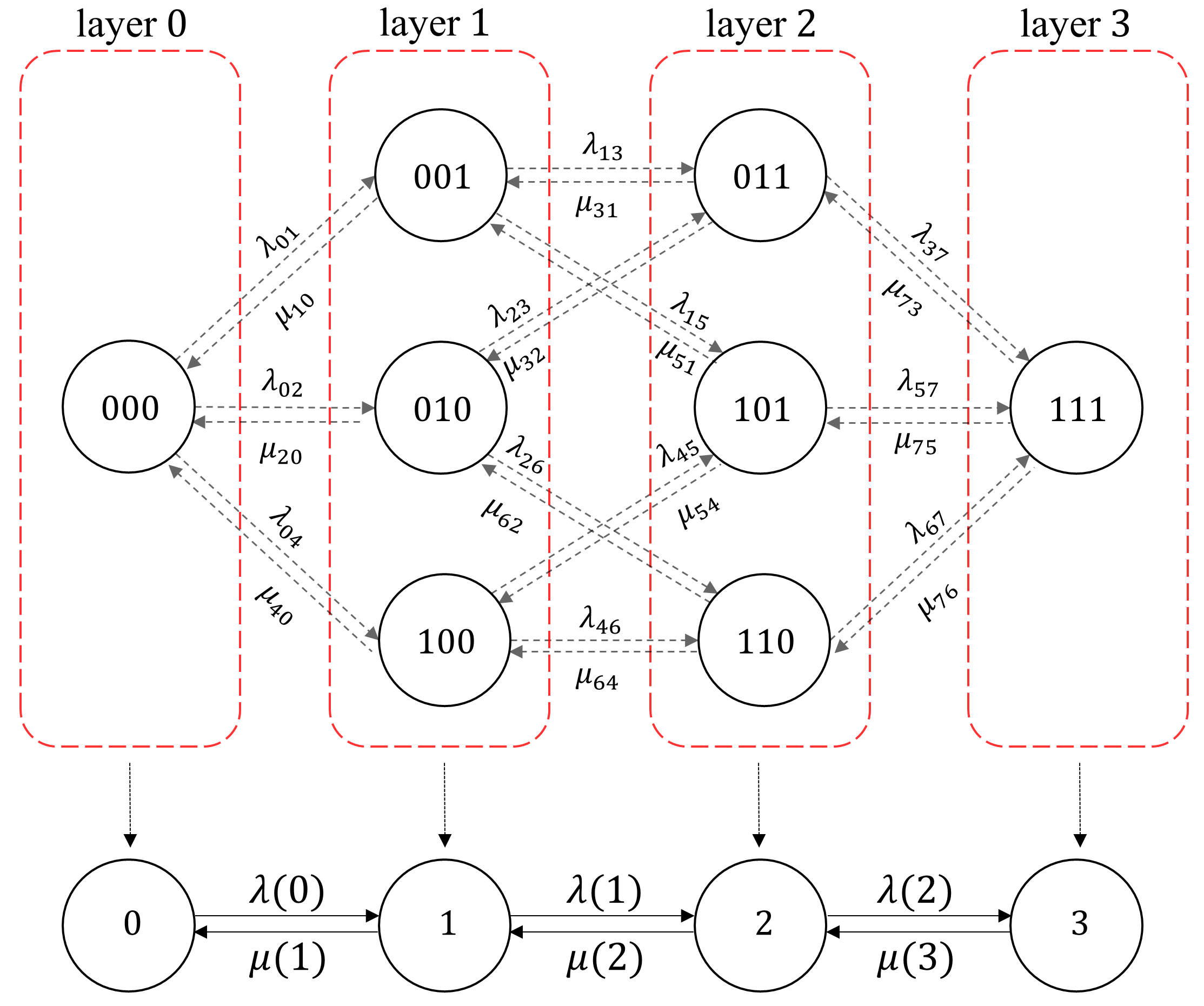}
\end{center}
\caption{Reformulation to a Birth-Death Process. }
 {\footnotesize 
    \linespread{1.0}\selectfont 
    \noindent\textit{Note.} We aggregate states with the same number of busy units into layers. These layers form a birth-death process. The transition rates $\lambda(n)$ and $\mu(n)$ in the birth-death process are linear combinations of the upward and downward transitions from the hypercube model.
    \par}
\label{fig_BnD}
\vspace{-0.3cm}
\end{figure}

Define $p_{n}\left(B_m\right)=P\left\{B_m |  n \text{ busy units}\right\}$ as the conditional probability that the system is in state $B_m$ given $n$ busy units, where $\sum_{m \in \mathscr{C}_n} p_{n}\left(B_m\right) = 1$ for all $n = 0,1, \ldots, N$. The key to developing this formulation is to recognize that the arrival rates and service rates $\lambda(n)$ and $\mu(n)$ can be derived from the transitions of the hypercube model and the conditional probabilities $p_{n}\left(B_m\right)$. 
We have for $n=0, \ldots, N$,
\begin{eqnarray}
    \lambda\left(n\right)&=&\sum_{m \in \mathscr{C}_n} p_{n}\left(B_m\right) \lambda_{m},  \label{eq:lambda}\\
    \mu(n) &=& \sum_{m \in \mathscr{C}_n} p_{n}\left(B_m\right) \mu_{m}, \label{eq:mu}
\end{eqnarray}
where $\lambda_{m} = \sum_{l\in \mathscr{C}_{n+1}} \lambda_{ml}$ and $\mu_{m} = \sum_{l\in \mathscr{C}_{n-1}} \mu_{ml} = \sum_{i=1}^N b_{m,i}\nu_i$ are the total rates leaving state $B_m$ by arrivals and by service completions, respectively. With the above newly introduced variables, we present a reformulation using conditional probabilities in the following proposition. We present the detailed proof in \S \ref{app_thm_equal_2_3} of the E-Companion. 

\begin{proposition}[Reformulation]
\label{lemma_bnd}
The conditional probability formulation is given, for \ins{$n=1, 2, \ldots, N-1$}, and $m \in \mathscr{C}_n$, by
\begin{eqnarray}
    &&p_n(B_m)=\sum_{l \in \mathscr{C}_{n-1}} p_{n-1}(B_l) \frac{\mu(n)}{\lambda(n-1)}\frac{\lambda_{l m}}{\lambda_{m}+\mu_{m}} + \sum_{l \in \mathscr{C}^m_{n+1}} p_{n+1}(B_l) \frac{\lambda(n)}{\mu(n+1)} \frac{\mu_{lm}}{\lambda_{m}+\mu_{m}},\label{eq:p_n(B_m)}\\
    &&\sum_{m \in \mathscr{C}_n} p_n(B_m) = 1, \label{eq:sumtol_p_n(B_m)}
\end{eqnarray}
and $p_{0}(B_0) = p_{N}(B_{2^N-1}) = 1$.
The probability distribution of system states is given, for $m=0,1,\ldots, 2^N-1$, by $P\{B_m\} = p(n)\cdot p_{n}(B_m)$,
where $n = w(B_m)$ and $p(n)$ follows
\begin{eqnarray}
p\left(0\right) = \frac{1}{1 + \sum_{n=1}^{N} \prod_{s=1}^{n} \frac{\lambda(s-1)}{\mu(s)}}, \quad p\left(n\right)&=& \frac{\prod_{s=1}^{n} \frac{\lambda(s-1)}{\mu(s)}}{1 + \sum_{n=1}^{N} \prod_{s=1}^{n} \frac{\lambda(s-1)}{\mu(s)}} \ \mbox{for} \ n = 1, \ldots, N.  \label{eq:bnd_stationary}
\end{eqnarray}
This solution yields the same probability distribution of system states as the original model. 
\end{proposition}

We now extend this reformulation to systems with a finite buffer capacity of $C$. Let $\tilde{P}\{c\}$ denote the probability of having $c$ calls waiting in the system when all $N$ units are busy. Therefore, when there are $c$ calls waiting, the total number of calls in the system is $N+c$. For the queued system, the balance equation \eqref{eq:balance_eq_heter} remains valid, except for the state where all units are busy, i.e., the state $\{1,1,\ldots,1\}$. Specifically, for this state, we now have: 
\begin{equation*}
    P\left\{1,1,\ldots,1\right\}\big(\overbrace{\lambda}^{\text{arrival}}+\overbrace{\sum_{i=1}^N \nu_i}^{\text{service completion}}\big)= \overbrace{\sum_{i=1}^N P\{1,1,\ldots,\underbrace{0}_{i^{{\rm th}}},\ldots,1\} \lambda}^{\text{upward transition}}+\overbrace{\tilde{P}\left\{1\right\}\sum_{i=1}^N \nu_i}^{\text{downward transition}}.
\end{equation*}
For this state, it transitions to state $\tilde{P}\{1\}$, where one call is waiting, with a rate of $\lambda$ due to a service call arrival. 
Furthermore, the transition rate from $\tilde{P}\{1\}$ to this state is $\sum_{i=1}^N \nu_i$. Additionally, the states $\tilde{P}\{c\}$ that represent waiting calls follow the standard birth-death transition equation. For $c=1,\ldots,C-1$, we have
\begin{equation} \label{eq:finite_buffer_queue}
   \tilde{P}\{c\}(\lambda+\sum_{i=1}^N  \nu_i)= \tilde{P}\{c-1\}\lambda + \tilde{P}\{c+1\} \sum_{i=1}^N  \nu_i,
\end{equation}
and for $c=C$, we have $\tilde{P}\{C\}\sum_{i=1}^N  \nu_i= \tilde{P}\{C-1\}\lambda$. Note that, when $c=1$, the state $\tilde{P}\{c-1\}=\tilde{P}\{0\}:=P\{1,1,\ldots,1\}$. 
Lastly, the normalization Equation \eqref{eq:sumto1_intial} becomes
\begin{equation*}
    \sum_{m=0}^{2^N-1} P\{B_m\} + \sum_{c=1}^{C} \tilde{P}\{c\}= 1,
\end{equation*}
which ensures that the probabilities of all states in the system sum to one. For the transition rates of the birth-death process, we have
\begin{eqnarray*}
    \lambda\left(n\right)&=& \begin{cases}\sum_{m \in \mathscr{C}_n} p_{n}\left(B_m\right) \lambda_{m} & \text { for } n=0,1, \ldots, N, \\ \lambda & \text { for } n=N+1, N+2, \ldots, N+C-1,\\
    0 & \text { for } n=N+C;
    \end{cases} \\
    \mu(n) &=&  \begin{cases}\sum_{m \in \mathscr{C}_n} p_{n}\left(B_m\right) \mu_{m} & \text { for } n=0,1, \ldots, N, \\ \sum_{i=1}^N  \nu_i & \text { for } n=N+1, N+2, \ldots, N+C,
    \end{cases}
\end{eqnarray*}
where $\lambda_{m}$ and $\mu_{m}$ are analogous to those in the loss system.

\begin{figure}[t]
\begin{center}
\includegraphics[width=15.5cm]{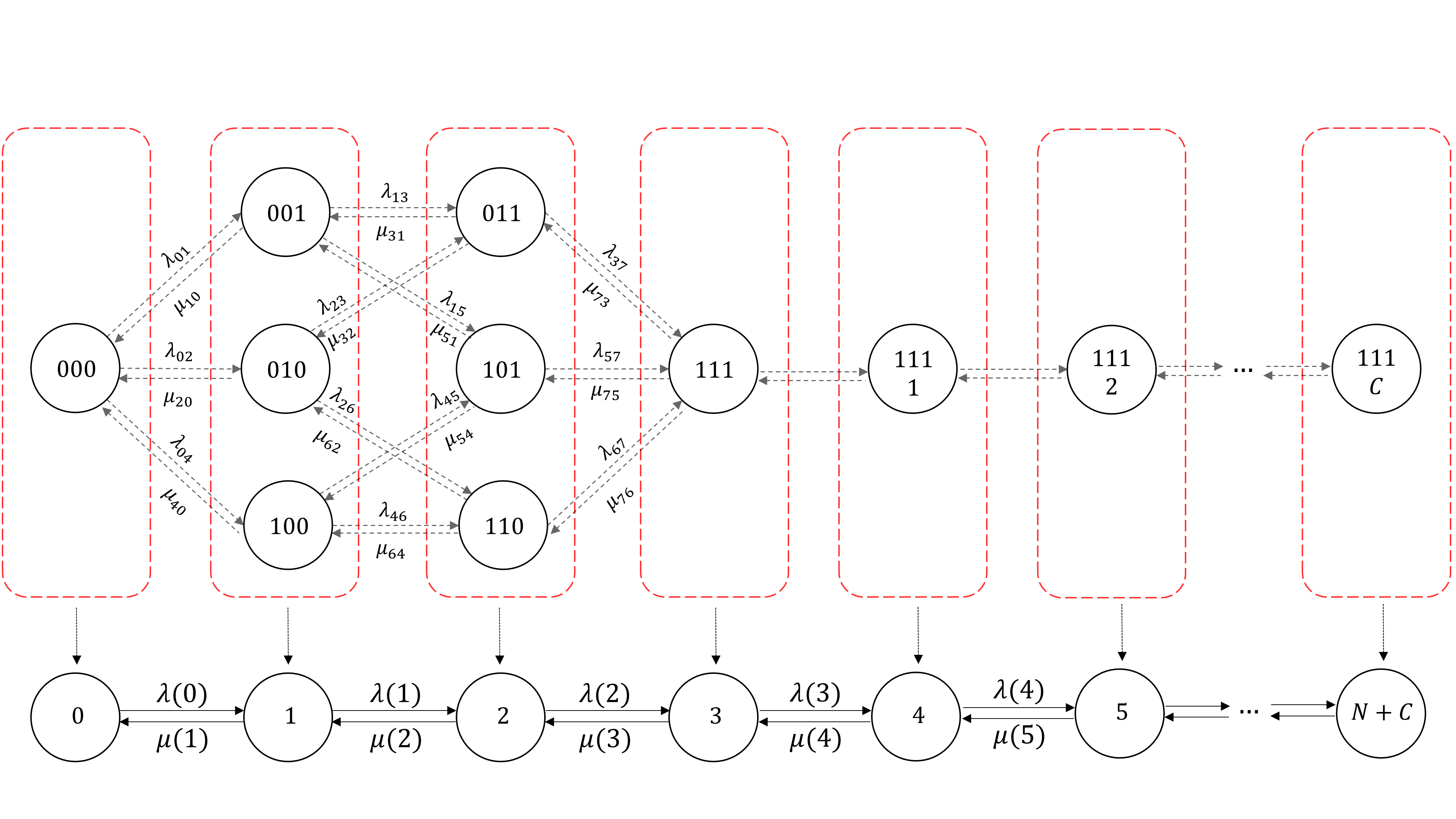}
\caption{Reformulation to a Birth-Death Process for a Finite Buffer System.}
\label{fig_BnD_queued}
\end{center}
 {\footnotesize 
    \linespread{1.0}\selectfont 
    \noindent\textit{Note.} Similar to the no-queue system, we aggregate states with the same number of busy units into layers. When all units are busy, customers wait until the buffer capacity $C$ is reached. 
    \par}
\end{figure}

The illustration for a three-unit system is shown in Figure \ref{fig_BnD_queued}. Using the system balance equations defined earlier for the finite buffer queued system, we now derive the following reformulation.

\begin{proposition}[Reformulation, Finite Buffer]
\label{lemma:finite_buffer_solution}
The conditional probability formulation for the finite buffer system is given, for $n=1,2,\ldots, N-1$, and $m \in \mathscr{C}_n$, by
\begin{eqnarray*}
    &&p_n(B_m)=\sum_{l \in \mathscr{C}_{n-1}} p_{n-1}(B_l) \frac{\mu(n)}{\lambda(n-1)}\frac{\lambda_{l m}}{\lambda_{m}+\mu_{m}} + \sum_{l \in \mathscr{C}^m_{n+1}} p_{n+1}(B_l) \frac{\lambda(n)}{\mu(n+1)} \frac{\mu_{lm}}{\lambda_{m}+\mu_{m}},\\
    &&\sum_{m \in \mathscr{C}_n} p_n(B_m) = 1,
\end{eqnarray*}
and $p_{0}(B_0) = p_{N}(B_{2^N-1}) = 1$.
The probability distribution of system states $P\{B_m\}$ and $\tilde{P}\{c\}$ is given, for $m=0,1,\ldots, 2^N-1$ and $c = 1, 2, \ldots, C$, by
\begin{eqnarray*}
    && P\{B_m\} = p(w(B_m))\cdot p_{w(B_m)}(B_m),\\
    && \tilde{P}\{c\} = p(N+c),
\end{eqnarray*}
\color{black}
where $p(n) = G(n)p(0)$ is the steady-state solution to the birth-death process, where
\begin{eqnarray*}
    G(n) &=& \begin{cases}\frac{\lambda(0)\ldots\lambda(n-1)}{\mu(1)\ldots\mu(n)} & \text { for } n=1,2, \ldots, N, \\ \frac{\lambda(0)\ldots\lambda(n-1)}{\mu(1)\ldots\mu(N-1)\mu(N)^{n-N+1}} & \text { for } n=N+1, N+2, \ldots, N+C,\end{cases}
\end{eqnarray*}
and 
\begin{equation*}
p(0)=\left[1 + \sum_{n=1}^{N+C} G(n)\right]^{-1}.
\end{equation*}
\end{proposition}

The logic of the proof closely follows the approach used in proving Proposition \ref{lemma_bnd}, where the zero-capacity system was addressed. Specifically, we express the joint probability as $P\{n,B_m\} = p(n) \cdot p_n(B_m)$ in Equation \eqref{eq:transformer_bal_eq_heter} and cancel equivalent terms on both sides to derive the results for $n \leq N$. For $n > N$, we incorporate the balance equations from Equation \eqref{eq:finite_buffer_queue}, noting that these, together with the $n \leq N$ part, form the birth-death process.

\begin{remark}
    Proposition~\ref{lemma:finite_buffer_solution} also applies to a system with an infinite waiting capacity. In a scenario with an infinite capacity, a stable system requires that $\frac{\lambda}{\mu(N)}<1$. The formulation for the infinite capacity system is consistent except for the following distinction:
    \begin{equation*}
    p(0)=\left[1 + \sum_{n=1}^{N-1} G(n) + \frac{\lambda^N}{\mu(1)\ldots\mu(N)} \frac{1}{1-(\lambda / \mu(N))}\right]^{-1}.
    \end{equation*}
\end{remark}
\color{black}
\color{black}

\section{Algorithm and Convergence Analysis}
\label{sec_convergence}
In this section, we derive an iterative algorithm based on the conditional probability formulation \eqref{eq:p_n(B_m)} to determine the probability distribution of system states for the zero-queue system. Section \ref{ssec_algo_algo} provides a detailed description of the algorithm, while Sections \ref{ssec_algo_conv} and \ref{ssec_seqComplexity} analyze its convergence and computational complexity, respectively.

\subsection{\ins{CPU} Algorithm} \label{ssec_algo_algo}
The algorithm builds on the reformulated expression in \eqref{eq:p_n(B_m)}. A key challenge in developing this algorithm lies in the interdependence between the rate $\mu(n)$ in \eqref{eq:mu} and the steady-state probability $P\{B_m\}$. It is therefore essential to ensure that the values within each layer converge before updating the system states.

Let $p_n^{k,r}(B_m)$ denote the value of $p_n(B_m)$ at the $r$-th update within the $k$-th iteration, and let $p_n^{k}(B_m)$ represent its converged value at the $k$-th iteration. The iterative formulation for computing the conditional probabilities $p_n(B_m)$, for $m \in \mathscr{C}_n$ and $n = 0, 1, \dots, N$, is defined as follows, 
\begin{equation} 
    p_n^{k,r}(B_m)=\sum_{l \in \mathscr{C}_{n-1}} p_{n-1}^{k}(B_l) \frac{\mu^{k,r-1}(n)}{\lambda^{k}(n-1)}\frac{\lambda_{l m}}{\lambda_{m}+\mu_{m}} + \sum_{l \in \mathscr{C}^m_{n+1}} p_{n+1}^{k-1}(B_l) \frac{\lambda^{k-1}(n)}{\mu^{k-1}(n+1)}\frac{\mu_{lm}}{\lambda_{m}+\mu_{m}},  \label{eq:updateHeter} 
\end{equation}
and
\begin{eqnarray} 
    \lambda^{k}(n) &=& \sum_{m \in \mathscr{C}_n} p_{n}^{k}\left(B_m\right) \lambda_{m},  \label{eq:lambdaHeter} \\ 
    \mu^{k,r}(n) &=& \sum_{m \in \mathscr{C}_n} p_{n}^{k,r}\left(B_m\right) \mu_{m}, \label{eq:muHeter} 
\end{eqnarray}
where $\mu^{k,r}(n)$ denotes the value of $\mu(n)$ after the $r$-th update during the $k$-th iteration, with $\mu^{k}(n)$ representing the corresponding converged value after all updates in the $k$-th iteration. Similarly, $\lambda^{k}(n)$ represents the value of $\lambda(n)$ at the $k$-th iteration. 
At the beginning of each iteration, we initialize $p_n^{k,0}(B_m) = p_n^{k-1}(B_m)$ for the first update of the $k$-th iteration. 

In the iterative formulation, the term $\mu^{k,r-1}(n)$ refers to the value of $\mu(n)$ at the previous update $r-1$ within the same iteration $k$, reflecting the dependency on the previous state of the system. This term appears in the equation because the rate $\mu(n)$ at layer $n$ relies on the values of $p_n(B_m)$ from the previous update. In contrast, the rest of the formula does not include $r$-dependent terms. 

We refer to our algorithm as the Conditional Probability Update (CPU) algorithm, as detailed  in Algorithm~\ref{algo-BnDHeter}. In step 2, we generate the transition rates and initialize $p_n^1(B_m) = {1}/{{N \choose n}}$. The iteration process begins in step 3 and continues until the maximum absolute error between iterations is smaller than a predefined threshold $\epsilon > 0$, which can be set arbitrarily small. During each iteration $k$, since $\mu(n)$ depends on the evolving steady-state probability $P\{B_m\}$, an inner loop is used to compute $p_{n}^{k}(B_m)$. Once the conditional probabilities $p^k_n(B_m)$ are obtained at the final step, the solution to the spatial queueing problem is achieved by applying the results from Proposition~\ref{lemma_bnd}. 
\ins{We note that this algorithm is designed for the zero-queue case. When a finite buffer is included, the queueing component follows a standard birth-death process with heterogeneous service rates, which can be determined directly using the same iterative framework.} 

Another notable consideration is that generating transition rates can be computationally intensive. In Algorithm \ref{algo-BnDHeter}, we rely on the TOUR algorithm introduced by \cite{larson74hypercube} to efficiently compute and store these rates. The TOUR algorithm traverses the hypercube one state at a time, updating coefficients along the way. For a detailed description of this procedure, we refer interested readers to \cite{larson74hypercube}.

\begin{algorithm}[t] 
    \caption{Conditional Probability Update (CPU)}
    \label{algo-BnDHeter}
    {\fontsize{10}{16}\selectfont
    \begin{algorithmic}[1]
        \State  {\textbf{Input:}} Arrival rate $\lambda$, service rates $\nu_i$, demand fractions $f_j$, preference lists $\zeta_j(h)$, and tolerances $\epsilon$ and $\delta$.
        \State  \textbf{Initialize:} Generate the transition coefficients $\lambda_{lm}$ and $\mu_{lm}$ for all $l, m \in \mathscr{C}_n$ according to the TOUR algorithm \citep{larson74hypercube}. Set $k=1$, $r=1$, $p_n^0(B_m)=0$, $p_n^{1}(B_m)=p_n^{1, 1}(B_m)=\frac{1}{{N \choose n}}$, for all $m \in \mathscr{C}_{n}$ and $n = 0,\ldots,N$. 
        \While {$\max_{n,m} (|p_n^{k}(B_m)-p_n^{k-1}(B_m)|)>\epsilon$}
        \State $k = k + 1$.
        \State $p_n^{k,0}(B_m)=p_n^{k-1}(B_m)$.
        \For {$n = 1 \ \text{to} \ N-1$}
        \State $r = 1$.
        \While {$\max_{m} (|p_n^{k,r}(B_m)-p_n^{k,r-1}(B_m)|)>\delta$}
        \State $r = r + 1$.
        \State Update $p_n^{k,r}(B_m)$ by Equation \eqref{eq:updateHeter}.
        \State Update $\mu^{k,r}(n)$ by  Equation \eqref{eq:muHeter}.
        \EndWhile
        \State Update $\lambda^{k}(n)$ by Equation \eqref{eq:lambdaHeter}. Let $\mu^{k}(n) = \mu^{k,r}(n)$.
        \EndFor
        \EndWhile
        \State Calculate $p(n)$ by Equation \eqref{eq:bnd_stationary}.
        \State \textbf{Output:} {$P\{B_m\}=p(w(B_m))\cdot p_{w(B_m)}^{k}(B_m)$}, for $m=0, \ldots, 2^N-1$.
        \end{algorithmic}
    }
\end{algorithm}

\subsection{Convergence Analysis} \label{ssec_algo_conv}
Next, we present the theoretical results of the CPU algorithm, focusing on its convergence and the speed of convergence. We start by establishing two key properties: \textit{within-layer convergence} and the \textit{normalization property}. Within-layer convergence ensures a stable value for $p_n^{k}(B_m)$ for each layer $n$. The normalization property, on the other hand, ensures that the probabilities naturally sum to one after each iteration of Algorithm \ref{algo-BnDHeter}, eliminating the need for manual normalization to satisfy \eqref{eq:sumtol_p_n(B_m)}. This property is crucial for proving the convergence of the algorithm. Before formally stating the properties, we first introduce the following assumption.
\begin{assu}\label{assu-phi}
      For all layers $n$, let $\overline{\mu}_n = \max\left\{\mu_m, m \in \mathscr{C}_{n}\right\}$ and $\underline{\mu}_n = \min\left\{\mu_m, m \in \mathscr{C}_{n}\right\}$ denote the maximum and minimum downward transition rates in layer $n$, respectively. For any system with $N \geq 2$, we assume that $\Phi_n < 1$ for any $1 \leq n\leq N-1$, where
      \begin{eqnarray*}
      \Phi_1 =  \frac{\lambda}{\lambda + \underline{\mu}_1}\gamma_2, \quad 
     \Phi_n =  \frac{\lambda}{\lambda + \underline{\mu}_n} \left(\sum_{q=2}^{n} \gamma_q \prod_{j=q}^{n} \frac{ \overline{\mu}_j}{\lambda + \underline{\mu}_{j-1}} + \gamma_{n+1} \right),  \ (n \ge 2), \quad \text{and} \quad  \gamma_n  = \frac{ \overline{\mu}_{n}}{\underline{\mu}_{n}}.
\end{eqnarray*}
\end{assu}

\ins{We note that Assumption~\ref{assu-phi} is a technical assumption. Its restrictiveness depends on the degree of heterogeneity in service rates among units. When service rates are homogeneous, the assumption holds automatically. In practice, service rates across servers do vary but usually within a limited range, so the assumption is expected to hold for most real systems. Our numerical experiments show that the algorithm converges rapidly across a wide set of scenarios, and we have not encountered any case where convergence fails.}

\color{black}


\begin{lemma}
\label{lemma:conv&sum}
    The following two properties hold for all iterations $k$:
    \begin{enumerate}
        \item[1)] \textup{(\textbf{Convergence within Layer})}. The conditional probabilities $p_{n}^{k,r}(B_m)$ converge for all states $B_m$ as $r$ increases, i.e., $\forall n=0,1,\ldots, N$ and $\forall m \in \mathscr{C}_{n}$,
        \begin{equation*}
            \lim_{r\rightarrow \infty}\left|p_n^{k,r+1}(B_m)-p_n^{k,r}(B_m)\right| = 0.
        \end{equation*} 
        \item[2)] \textup{(\textbf{Normalization Property})}. The conditional probabilities $p_n^{k,r}(B_m)$ sum to one for all layers $n$ as $r$ increases, i.e., $\forall n=0,1,\ldots, N$,
        \begin{equation*}
            \lim_{r\rightarrow \infty} \sum_{m \in \mathscr{C}_n} p_n^{k,r}(B_m) = 1.
        \end{equation*}
    \end{enumerate}
\end{lemma}

\textbf{Proof} (Sketch): The proof of convergence proceeds by induction. In the inductive step, we first apply Equations \eqref{eq:updateHeter}-\eqref{eq:muHeter} and the triangle inequality to establish the bound of the absolute difference value $ \left| \Delta^{k,r}_{n,m}\right| = \left| p_n^{k,r}(B_m) - p_n^{k,r-1}(B_m) \right|$. Next, we express the total absolute difference for each layer in a recursive form $\sum_{m \in \mathscr{C}_n} \left|\Delta_{n,m}^{k,r} \right| =  \varrho_n \left(\sum_{m \in \mathscr{C}_n} \left|\Delta_{n,m}^{k,r-1} \right| \right)$ with some $\varrho_n < 1$. This shows that the total absolute difference of any layer converges, which in turn indicates in iteration $k$ that the probabilities also converge. The normalization property is then established by summing the probabilities within a layer on both sides of Equation \eqref{eq:updateHeter} and taking the limit on both sides. We present the detailed proof in \S \ref{ssec:proof_lemma_hetero} of the E-Companion. \Halmos

The convergence property ensures that the probabilities in each layer converge at every iteration and the normalization property ensures the probability distributions of system states sum to one at each iteration $k$ without normalization, i.e., $\sum_{m=0}^{2^N-1} P^k\{B_m\} = 1$, for all iterations. The normalization property is essential for the following convergence property. 

\begin{theorem}[Geometric Convergence]
    \label{thm_convergence_heter}
    Under Assumption \ref{assu-phi}, the CPU algorithm converges at a geometric rate, i.e., for all $n=0,1,\ldots, N$ and $m=0, \ldots, 2^N-1$,
    \begin{equation*}
        \lim_{k\rightarrow \infty} \left|p_n^{k+1}(B_m)-p_n^{k}(B_m)\right| \leq \lim_{k\rightarrow \infty} \Phi_N^k \theta = 0
    \end{equation*}
    for some $0<\theta<\infty$. In addition, the algorithm converges to the exact hypercube solution.
\end{theorem}

\textbf{Proof} (Sketch):
The proof for this theorem follows the following steps:
\begin{enumerate}[(i)]
    \item We bound the absolute difference value $\left|p_n^{k+1}(B_m)-p_n^{k}(B_m)\right|$ using Equation \eqref{eq:updateHeter}, Lemma \ref{lemma:conv&sum} and the \textit{Triangle Inequality}. Then, we bound the sum of the absolute differences for each layer $M_{k+1,n}=\sum_{l \in \mathscr{C}_n}\left|p_n^{k+1}(B_l)-p_n^{k}(B_l)\right|$, and express $M_{k+1, n}$ in recursion form $M_{k+1, n} \leq f\left(M_{k+1, n-1}, M_{k, n+1}\right)$ for some function $f$.
    \item By defining $M_k = \sup_n M_{k,n}$, i.e., the largest difference across all layers, we show by induction that $M_{k+1,n} \leq \Phi_n M_{k}$ for any layer $n$ with $\Phi_n$ as defined in Assumption \ref{assu-phi}, through the recursive function $f$. We also prove that $\Phi_n$ increases with $n$ and the largest $\Phi_N < 1$. We thus have $M_{k+1} \leq \Phi_N M_k \leq \Phi^2_N M_{k-1}\leq \Phi_N ^k M_1$, indicating that $M_k$ decreases at a geometric rate. 
    \item Since $\left|p_n^{k+1}(B_m)-p_n^{k}(B_m)\right| \leq M_{k+1, n} \leq M_{k+1}$, we conclude that $\lim _{k \rightarrow \infty}\left|p_n^{k+1}\left(B_l\right)-p_n^k\left(B_l\right)\right|<\lim _{k \rightarrow \infty} M_{k+1}<\lim _{k \rightarrow \infty} \Phi_N^k M_1=0$.
    \item  To show it converges to the exact solution, we argue that for an irreducible and positive recurrent continuous Markov process, the stationary distribution is unique. \Halmos
\end{enumerate}

A detailed proof is provided in Section \ref{ssec:proof_conv_heter} of the E-Companion. The above analysis addresses cases with heterogeneous service rates. It is important to note that the homogeneous case is a special instance. For systems with a homogeneous service rate, i.e., $\nu_i = \nu$ for all units $i$, 
Assumption \ref{assu-phi} is automatically satisfied because $\overline{\mu}_n = \underline{\mu}_n = n\nu$ for every layer $n$.
Specifically, we have: 
\begin{eqnarray*}
    \Phi_N = \frac{\lambda}{\lambda + N\nu} \left(\sum_{q=2}^{N} \prod_{j=q}^{N} \frac{ j\nu}{\lambda + (j-1)\nu} + 1 \right) 
    = \frac{1}{\prod_{j=1}^{N} (\lambda + j\nu)} \left(\prod_{j=1}^{N} (\lambda + j\nu) - \prod_{j=1}^{N} j\nu \right) < 1, 
\end{eqnarray*}
Therefore, both properties in Lemma \ref{lemma:conv&sum} hold, and thus Theorem \ref{thm_convergence_heter} applies.

\begin{corollary}[Homogeneous Setting]
    For homogeneous service rate systems (i.e., $\nu_i = \nu$ for all $i \in \{1,2,\ldots,N\}$), the following properties hold:
    \begin{enumerate}
        \item[1)] \textup{(\textbf{Normalization Property})}. The conditional probabilities $p_n^k(B_m)$ sum to one at each iteration $k$ for all layers $n$, i.e., $\forall k = 0,1,\ldots, \forall n=0,1,\ldots, N$,
            \begin{equation*}
                \sum_{m \in \mathscr{C}_n} p_n^k(B_m) = 1.
            \end{equation*}
        \item[2)] \textup{(\textbf{Geometric Convergence})}. The CPU algorithm converges at a geometric rate, i.e., for all $n=0,1,\ldots, N$ and $m=0, 1,\ldots, 2^N-1$, there exists $0<\varphi<1$ and $0<\delta<\infty$ such that
            \begin{equation*}
                \lim_{k\rightarrow \infty}\left|p_n^{k+1}(B_m)-p_n^{k}(B_m)\right|  \leq \lim_{k\rightarrow \infty} \varphi^k \delta = 0.
            \end{equation*}
    \end{enumerate}
\end{corollary}

\subsection{Complexity Analysis} \label{ssec_seqComplexity}
In this section, we analyze the space and time computational complexities of Algorithm~\ref{algo-BnDHeter}, where the \textit{space computational complexity} (SCC) measures the amount of memory required by the algorithm, while the \textit{time computational complexity} (TCC) represents the number of computational steps needed for completion. 

We first examine the SCC. In Algorithm~\ref{algo-BnDHeter}, three types of data need to be stored, i.e., the transition coefficients $\lambda_{lm}$ and $\mu_{lm}$, the probability distribution $p_{n}^{k,r}(B_m)$, and the birth-death process variables $\lambda^k(n)$, $\mu^{k,r}(n)$, and $p(n)$. There are $2^N$ system states, and each state has $N$ one-step reachable states that differ by a Hamming distance of one. Accordingly, the SCCs of the transition coefficients and the probability distribution are $O(2^N N)$ and $O(2^N)$, respectively. The birth-death process variables are scalar quantities for each layer $n$, resulting in an SCC of $O(N)$. Therefore, the overall SCC of Algorithm~\ref{algo-BnDHeter} is $O(2^N N)$.

Next, we analyze the TCC. In the initialization step (step 2) of Algorithm \ref{algo-BnDHeter}, generating transition coefficients involves aggregating contributions from all $J$ demand nodes across $2^N$ states, each with $N$ reachable neighbors. This yields a TCC of $O(2^N N J)$. Initializing the probability distribution requires $O(2^N)$ operations. In the iteration steps, we focus on updating probabilities and transition parameters (steps 10, 11, and 13). In step 10, each state $m \in \mathscr{C}_n$ has $N$ reachable states in layers $n-1$ and $n+1$. Based on Equation \eqref{eq:updateHeter}, updating $p_{n}^{k,r}(B_m)$ requires $O(N)$ operations. Since layer $n$ contains $\binom{N}{n}$ states, step 10 has a total TCC of $O\big(\binom{N}{n} N\big)$. In steps 11 and 13, by Equations \eqref{eq:muHeter} and \eqref{eq:lambdaHeter}, each step contributes $O\big(\binom{N}{n}\big)$. Summing across all layers gives a total TCC of $\sum_{n=1}^{N-1} O\big(\binom{N}{n} N\big) = O(2^N N)$ for the for-loop in step 6. Additionally, the assignment, addition, and comparison operations on $p_{n}^{k,r}(B_m)$ in steps 3, 5, and 8 each incur a TCC of $O(2^N)$. In the post-processing steps, step 16 requires $O(N^2)$ operations following Equation \eqref{eq:bnd_stationary}, while step 17 adds another $O(2^N)$. Consequently, the overall TCC of Algorithm \ref{algo-BnDHeter} is $O(2^N N J)$.


From the above analysis, the generation and storage of transition coefficients emerge as the primary computational bottleneck. In our experiments, we found that for systems with $N > 25$, the number of transition coefficients becomes too large to be stored on a personal computer. Moreover, as shown in Tables \ref{tab_homo_detail} and \ref{tab_greenville_homo_detail} in Section \ref{sec_extended_numerical}, coefficient generation accounts for more than 80\% of the total computation time. Therefore, we discuss potential strategies to mitigate these issues in Section \ref{sec_coef}.

 \color{black}

\section{Parallel Computing Solution Algorithm}
\label{sec_parallel} 

In this section, we develop a parallel computing solution algorithm to further improve its computational efficiency. Section \ref{ssec_parallel_algo} presents the overall framework of the parallel CPU algorithm. Section \ref{sec_coef} introduces the coefficient generation procedure, which serves as a key component of the parallel framework. Section \ref{ssec_parallel_complexity} analyzes the computational complexity of the proposed parallel algorithm.

\subsection{Parallel CPU Algorithm} \label{ssec_parallel_algo}
The parallel algorithm leverages the convergence property along with the following two key properties:
\begin{enumerate}
    \item[(1)] \textit{Non-dependency within layer:} The computation of the conditional probabilities $p^k_n(B_m)$ in each layer~$n$ does \textit{not} depend on other states in the same layer;
    \item[(2)] \textit{Local-dependency between layers:} The computation of $p^k_n(B_m)$ at the $k$-th iteration of layer $n$ \textit{only} depends on the values of layer $n-1$ at the $k$-th iteration and the values of layer $n+1$ at the $(k-1)$-th iteration. 
\end{enumerate}

The first property specifies the non-dependency between states in the same layer, enabling the algorithm at each layer, within each iteration, to be computed in a \textit{divide-and-conquer} manner \citep{cormen2022introduction}. The second property means that we only need to compute and store the necessary transition rates of the layers before and after the current layer, which significantly reduces the computational and storage requirements.

\begin{figure}[t!]
    \begin{center}
    \includegraphics[width=17cm]{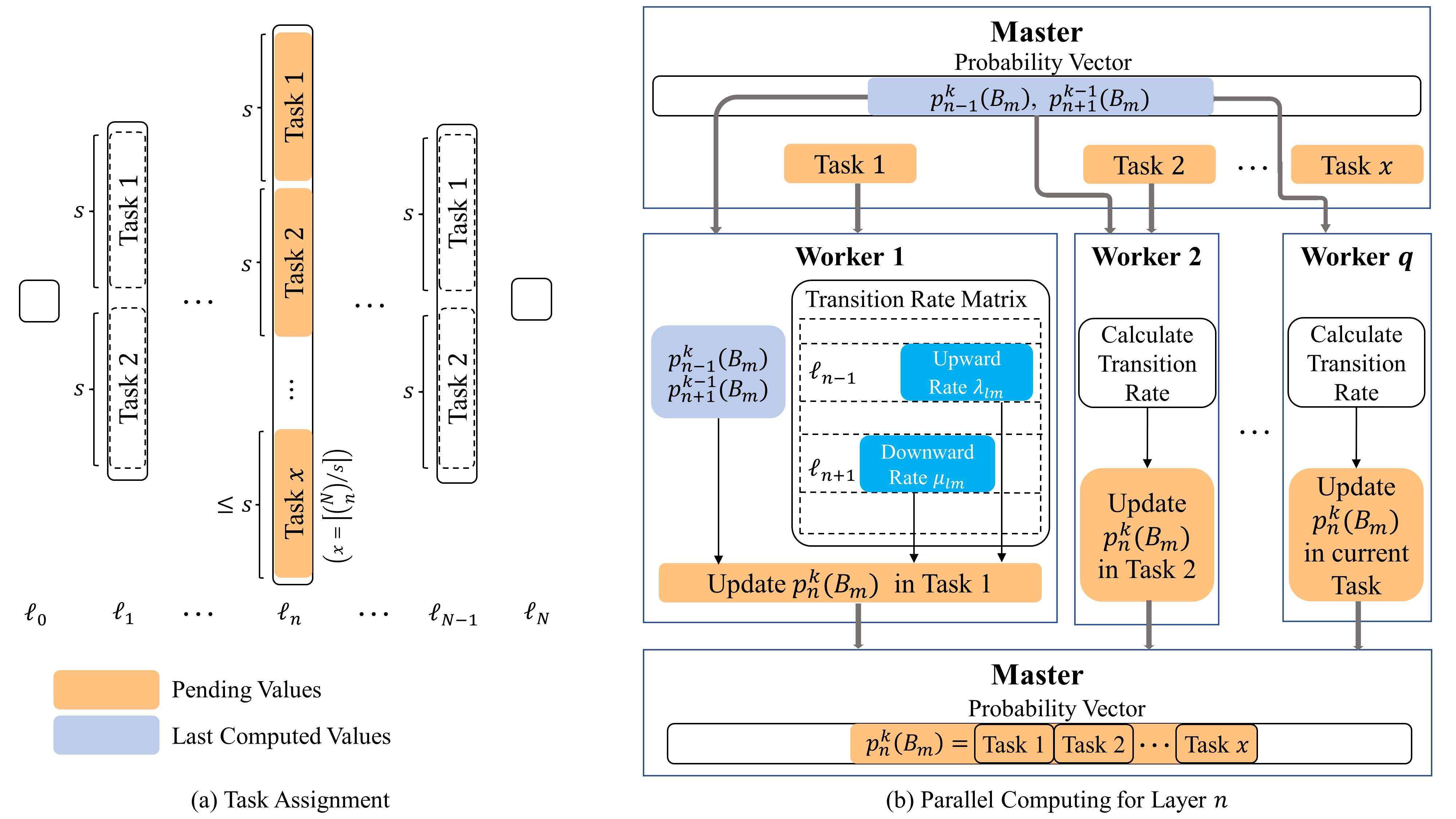}
    \end{center}
    \caption{Task Assignment and Parallel Computing.}
    {\footnotesize 
    \linespread{1.0}\selectfont 
    \noindent\textit{Note.} In (a), we partition the states in each layer into different tasks, each with a size equal to or less than $s$. In (b), we show how the tasks are computed. The master first assigns the tasks and necessary information to the workers. Each worker then generates the transition rates and updates the probabilities within the assigned tasks, sending the results back to the master. The master successively assigns pending tasks to idle workers and collects the outcomes. Once all tasks within a layer are completed, the master moves to the next layer.
    \par}
    \label{fig:parellel}
\end{figure}

Our parallel algorithm is based on the \textit{Master/Worker structure}, as illustrated in Figure \ref{fig:parellel}. The master processor oversees the process, assigning tasks to workers, gathering results, and deciding when to proceed to the next iteration. The workers receive tasks, calculate probabilities, and return results to the master. We consider a system to be solved using $q$ workers. In Figure \ref{fig:parellel}(a), we depict the partitioning of states within each layer. The term \textit{task} refers to updating a subset of conditional probabilities $p^k_n(B_m)$ within a layer. To balance the workload among workers, we partition states in each layer $n$ into $x$ batched tasks, where $x = \left\lceil \binom{N}{n} / s\right\rceil$ and $s$ denotes the batch size. Figure \ref{fig:parellel}(b) shows the coordination between the master and $q$ workers. In each iteration, the master dynamically assigns tasks to available workers. When tasks outnumber available workers or vice versa, the surplus tasks wait, or extra workers remain idle. Upon receiving a task, each worker computes the upward and downward transition rates ($\lambda_{lm}$ and $\mu_{lm}$) for the assigned states using the parallel coefficient generation algorithm (i.e., Algorithm \ref{alg:transition_generation}), which is described in detail in Section \ref{sec_coef}. We then use this information to update $p_n^{k,r}(B_m)$ using Equation \eqref{eq:updateHeter}. 
After updating all states in the task, the worker communicates the outcomes to the master and becomes available again. Once the master collects all updated probabilities within the current layer, the algorithm progresses to the next layer. We present the detailed algorithm in Algorithm \ref{algo-parallel}.


\begin{algorithm}[t] 
  \caption{Parallel Conditional Probability Update (Parallel CPU)}
  \label{algo-parallel}
  {\fontsize{10}{16}\selectfont
  \begin{algorithmic}[1]
  \State {\textbf{Input:}} Arrival rate $\lambda$, service rates $\nu_i$, demand fractions $f_j$,  preference lists $\zeta_{j}(h)$, number of workers $q$, batch size $s$, and tolerance $\epsilon$.
  \State  \textbf{Initialize:} Set step $k=1$, $p_n^0(B_m)=0, p_n^1(B_m)={1}/{{N \choose n}}$ for all $m \in \mathscr{C}_{n}$ and $n = 0,\ldots,N$. 
  \While{$\max_{n,m} (|p_n^{k}(B_m)-p_n^{k-1}(B_m)|)>\epsilon$}
  \State $k = k + 1$.
  \For{$n = 1$ to $N-1$}
    \State \textbf{(Master)} Partition all states into batches of size $s$, resulting in $x$ batches in total. Each batch is then assigned to an available worker until all states are processed. 
    \State \textbf{(Worker)} For each worker, generate transition rates $\lambda_{lm}$ and $\mu_{lm}$ using Algorithm \ref{alg:transition_generation} for states relevant to the current task. Update the probabilities $p_n^{k, r}(B_m)$ \textcolor{black}{and $\mu^{k,r}(n)$} using Equations ~\eqref{eq:updateHeter} and \eqref{eq:muHeter} until convergence. Then, send the updated results to the Master.
    \State \textbf{(Master)} Collect all results in the current layer $n$. 
  \EndFor
  \EndWhile
  \State Obtain $\lambda(n)$ and $\mu(n)$ using Equations ~\eqref{eq:lambda} and ~\eqref{eq:mu}, for $n=0,\ldots, N$. 
  \State Calculate $p(n)$ using Equation ~\eqref{eq:bnd_stationary}.
  \State \textbf{Output:} {$P\{B_m\}=p(w(B_m))\cdot p_{w(B_m)}^k(B_m)$}, for $m=0, \ldots, 2^N-1$.
  \end{algorithmic}
  }
\end{algorithm}

To assess the performance improvement achieved by the parallel solution, we estimate the proportion of the algorithm executable within a parallel environment using Amdahl's law \citep{amdahl1967validity}, which has been used to study parallel ranking and selection problems \citep{luo2015fully}.
\begin{lemma}
 \label{lemma-Amdahl}
  \textup{(\citealt{amdahl1967validity})}. Let $P$ be the proportion of the program that can be parallelized, and $\beta$ be the computational time required without parallelization. If the computation time $T$ in the parallel framework with $q$ workers fits the model
  \begin{equation*} 
  T = \beta \left(1-P + \frac{P}{q}\right) + \epsilon  
  = \underbrace{(\beta - \beta P)}_{:=c_0} + \underbrace{\beta P}_{:=c_1} \cdot \frac{1}{q} + \tilde{\epsilon}, \label{eqn-amdahl}
\end{equation*}
  where $\tilde{\epsilon}$ is the stochastic error term, then the degree of parallelism is $P = \frac{1}{1+ {c_0}/{c_1}}$.
\end{lemma}

Amdahl's law states that the speedup of parallelism $S = \frac{\beta}{T}$ can be expressed as $S = \frac{1}{1-P + {P}/{q}}$. 
In Section~\ref{sec_numerical}, we show that our proposed parallel framework achieves a high degree of parallelization.

\subsection{Coefficient Generation} \label{sec_coef}
The generation of transition coefficients $\lambda_{lm}$ and $\mu_{lm}$ are key operations performed by each worker in step 7 of Algorithm \ref{algo-parallel}. However, the TOUR algorithm used in Algorithm \ref{algo-BnDHeter} is unsuitable for parallel implementation, as it requires traversing the entire hypercube and precomputing all coefficients in advance. To address this limitation, we propose a new algorithm that enables fast, on-demand coefficient generation and storage for any given state, making it adaptable for parallel computing. This approach is particularly beneficial for large-scale systems where pre-storing all coefficients is infeasible, as it allows each worker to compute the required coefficients independently.

Our approach relies on traversing the preference list from most to least preferred for each demand node~$j$, stopping once an available unit is found. For states with fewer busy units, this search is expected to be very fast, while it may take longer for states with more busy units. In our solution approach, we group states based on the number of busy units. Additionally, we store non-zero coefficients using a coordinate list (COO) format for sparse matrices, which contains a list of  (row, column, value) tuples. This format can be readily applied for efficient matrix manipulation and is significantly faster than point updates or standard matrix manipulation. The computation can be executed independently by each worker in conjunction with our parallel computing framework.

To introduce our approach, we define \ins{$\mathscr{M}$ as the set of states included in a given task and} $\bar{b}_{m,i}$ as the logical negation of $b_{m,i}$. We then define the following operation: 
\begin{equation*}
  \begin{aligned}
    B_m \cap \bar{b}_{m,i} = \{b_{m,N}, \ldots, \bar{b}_{m,i}, \ldots, b_{m,1}\}
  \end{aligned},
\end{equation*}
where $B_m = \{b_{m,N}, \ldots, b_{m,1}\}$.
Moreover, for each state $B_m$, its numerical index $y(B_m)$ is defined as:
\begin{equation*}
    y(B_m) = b_{m,N} \cdot 2^{N-1} + \cdots + b_{m,2} \cdot 2^1 + b_{m,1}.
\end{equation*}

\begin{algorithm}[t]
  \caption{Coefficient Generation for Parallel CPU} \label{alg:transition_generation}
  {\fontsize{10}{16}\selectfont
  \begin{algorithmic}[1]
  \State {\textbf{Input:}} Arrival rate $\lambda$, service rates $\nu_i$, fraction of region-wide demand $f_j$, preference list $\zeta_j(h)$, state set $\mathscr{M}$.
  \State {\textbf{Initialize:}} 
  Initialize the upward and downward transition rate sets as empty sets: $\mathscr{L} = \mathscr{S} = \emptyset$.
  \For {$m \in \mathscr{M}$}
    \State Determine the set of busy units $\mathscr{U}(B_m)$ implied by the state $B_m$.
    \For {$l$ not in $\mathscr{U}(B_m)$}
    \State Add $\left(y(B_m \cap \bar{b}_{m,l}), y(B_m), \nu_{l}\right)$ to $\mathscr{S}$.
    \EndFor
    \For {$j = 1$ to $J$}
      \State Record all elements from $\zeta_{j}(h)$ up to the first element absent in $\mathscr{U}(B_m)$ and denote this set as $\zeta'_{j}$.
      \For {$l$ in $\zeta'_{j}$}
        \If{$\left(y(B_m \cap \bar{b}_{m,l}), y(B_m)\right)$ not in $\mathscr{L}$} 
            \State Add $\left(y(B_m \cap \bar{b}_{m,l}), y(B_m), \lambda_{lm} = f_j \lambda\right)$ to $\mathscr{L}$.
        \Else
            \State Update $\left(y(B_m \cap \bar{b}_{m,l}), y(B_m), \lambda_{lm}\right)$ with $\left(y(B_m \cap \bar{b}_{m,l}), y(B_m),  \lambda_{lm} + f_j\lambda\right)$.
        \EndIf
      \EndFor
    \EndFor
  \EndFor
  \State \textbf{Output:} Transition rates: $\mathscr{L}$ (upward) and $\mathscr{S}$ (downward).
  \end{algorithmic}
  }
\end{algorithm}

Our coefficient generation procedure is presented in Algorithm \ref{alg:transition_generation}. At each state, the algorithm computes the transition rates from all other states to the current one. We use $\mathscr{U}(B_m) = \{ i : b_{m,i} = 1, b_{m,i} \in B_m \}$ to denote the set of busy units in state $B_m$. The algorithm stores each transition rate as a three-element tuple (row, column, value) in the sets $\mathscr{L}$ and $\mathscr{S}$, which correspond to upward and downward transitions, respectively. The row and column identify the matrix position of the rate value. We note that a key advantage of our algorithm is its ability to store these tuples regardless of their order, which can be created incrementally. This suggests that our algorithm is more memory-efficient. When the number of demand nodes is large, memory efficiency becomes the primary factor determining the speed of generating and storing transition rates. 

\subsection{Complexity Analysis} \label{ssec_parallel_complexity}
In this section, we analyze the space and time computational complexities (i.e., SCC and TCC) of Algorithm \ref{algo-parallel}. Compared with Algorithm \ref{algo-BnDHeter}, the parallel CPU algorithm achieves reductions in both SCC and TCC. 

We begin with the SCC. Three categories of data need to be stored in Algorithm \ref{algo-parallel}: the probability distribution $p_n^{k}(B_m)$, the layer-level variables $\lambda(n)$ and $\mu(n)$, and the transition coefficients $\lambda_{lm}$ and $\mu_{lm}$. As shown in Section \ref{ssec_seqComplexity}, the SCCs for the first two components are $O(2^N)$ and $O(N)$, respectively. For the transition coefficients, Algorithm \ref{algo-parallel} eliminates the need to precompute and store all coefficients by generating them on the fly during probability updates. Since at most $qs$ states are processed concurrently across all processors, and each state has $N$ one-step reachable neighbors, the SCC of the transition coefficients is $O(N)$. Therefore, the overall SCC of Algorithm \ref{algo-parallel} remains $O(2^N)$.


We next analyze the TCC. The time complexity of a parallel algorithm can be described using the \textit{work–span model}, where the \textit{work} denotes the total number of primitive operations, and the \textit{span} represents the longest chain of sequentially dependent operations \citep{blelloch1996programming}. For Algorithm \ref{algo-parallel}, the total work equals the TCC of Algorithm \ref{algo-BnDHeter}. Regarding the span, the while-loop in step 3, the for-loop in step 5, and the post-processing steps 11-13 are executed sequentially. As analyzed in Section \ref{ssec_seqComplexity}, the TCCs of steps 3, 11, and 13 are each $O(2^N)$, while step 12 requires $O(N^2)$. The for-loop in step 5 involves $\Theta(N)$ iterations, within which step 7 requires each worker to perform $O(NJ)$ operations to generate transition coefficients, $O(N)$ operations to update $p^{k,r}_n(B_m)$ and $\mu^{k,r}(n)$, and $O(N)$ operations to check convergence. Consequently, the total TCC of the for-loop is $O(N^2J)$. Therefore, the overall span of Algorithm \ref{algo-parallel} is $O(2^N)$.


Notably, Algorithm \ref{algo-parallel} achieves improvements in both storage and computational efficiency compared with Algorithm \ref{algo-BnDHeter}. In terms of storage, the SCC decreases from $O(2^N N)$ to $O(2^N)$. In terms of computation time, assuming each primitive operation takes one time unit, Brent’s theorem \citep{padua2011encyclopedia} gives the total execution time of Algorithm \ref{algo-parallel} as $O\left({2^NNJ}/{q} + 2^N\right)$, where $q$ is the number of workers. We recognize that the exponential term cannot be eliminated, as the solution itself has size $2^N$. However, the parallel algorithm avoids storing all quantities simultaneously by processing them on demand, and exploits parallel execution across independent tasks. As a result, it substantially reduces memory requirements and computation time, making the solution of larger problems possible. 



\color{black}

\section{Numerical Results}
\label{sec_numerical}
In our numerical experiments, we first describe the datasets used to evaluate the proposed algorithms in Section~\ref{sec_num_dataset}. The efficiency and accuracy of the CPU algorithm are examined in Section~\ref{sec_num_time&accu}, followed by a comparison with discrete-event simulation in Section~\ref{sec-num_ComparisonDES}. Section~\ref{sec_distributed} evaluates the performance of the Parallel CPU algorithm, and Section~\ref{ssec_general_rate} investigates the robustness of the CPU algorithm under varying service time distributions. The code for this study is available at \url{https://github.com/StarYX96/SpatialQueue}. 


\subsection{Datasets} \label{sec_num_dataset}
We evaluate the performance of the proposed algorithms using data from two emergency service systems. The first dataset, collected from St. Paul, Minnesota, contains 30,911 incidents handled by nine units across 71 demand nodes in 2014. The second dataset, from Greenville County, South Carolina, originally reported by \cite{burwell1993modeling}, includes 10,233 emergency calls across 99 demand nodes in 1980, served by six units. The St. Paul dataset is used as an example to illustrate the parameter estimation procedures for both the CPU and parallel CPU algorithms. 

The St. Paul dataset was collected in collaboration with the St. Paul Fire Department, documenting 30,911 medical incidents that occurred throughout the year. The dataset includes detailed records of medical call dates, times, and locations as recorded by the fire department. In St. Paul, nine fire-medic units are responsible for responding to both fire and medical emergencies. To simplify our analysis, we focus solely on medical incidents. \textcolor{black}{We} vary the number of units to assess the robustness and accuracy of our algorithm.


St. Paul is divided into 71 census tracts, and we designate the center of each tract as a demand node in our analysis. To calculate travel distances from each station to these demand nodes, we used the Google Maps API, which provides the shortest route. The Department data indicated that the average turnout time was 1.75 minutes. To estimate the travel time $T(d)$ from a given distance $d$, we applied a robust estimation method as proposed by \cite{budge2010empirical}. The travel time model is expressed as follows: 
$$
\begin{aligned}
T(d) & =m(d)\cdot e^{c(d) \tau}, \\
m(d) & = \begin{cases}2 \sqrt{d / a}, & d \leq 2 d_c, \\
v_c / a+d / v_c, & d>2 d_c,\end{cases} \\
c(d) & =\frac{\sqrt{b_0\left(b_2+1\right)+b_1\left(b_2+1\right) m(d)+b_2 m(d)^2}}{m(d)}.
\end{aligned}
$$
In this model, $m(d)$ and $c(d)$ represent the median and coefficient of variation of the travel time distribution, respectively, while $\tau$ follows a centered $t$ distribution. To estimate the model parameters, we used a sample of 100 emergency responses and applied maximum likelihood estimation (MLE), achieving an $R^2$ of 0.827. The estimated parameter values are presented in Table \ref{tab_parameter}. Since $T(d)$ follows a log-normal distribution, the mean travel time $\bar{T}$ is calculated as: 
$$\bar{T}=m(d)\cdot e^{c(d)^2 / 2}.$$ 

The service time refers to the period from when a unit receives a call to when it becomes available again, including travel time, on-scene time, and, when applicable, time spent transporting the patient to and at the hospital. The dataset records the exact duration of each incident, measured from call receipt to service completion. We use the average of these durations to compute the service rate. The mean travel time $\bar{T}$ is not used in this calculation but is applied when evaluating response time and mean response time, which are key performance measures in emergency service systems.

\begin{table}[!t]
\caption{Estimated Values for Various Parameters}
\small 
  \centering
  \renewcommand\arraystretch{1.2}
    \begin{tabularx}{\textwidth}{cc*4{>{\centering\arraybackslash}X}}
    \hline
       $a$ (\text{acceleration}) &   $v_c$ (\text{cruising speed})    & $d_c$&   $b_0$ &   $b_1$    & $b_2$\\
    \hline
    17.33 mi/hr/min & 28.54 mi/hr & 0.39 mi & 1.551 & 0.8738 &  -0.0709\\
    \hline
    \end{tabularx}%
  \label{tab_parameter}%
  \vspace{-0.1cm}
\end{table}%

\ins{In summary, the average call rate for the St. Paul, Minnesota, setting was estimated at 3.53 calls per hour, with an average service time of 34.15 minutes.} We obtain the Greenville County, South Carolina data from \cite{burwell1993modeling} and convert it to the same format as the summary statistics used for St. Paul. The results show an average call rate of 1.5 calls per hour and an average service time of 38.0 minutes.

\subsection{Computation Time and Accuracy}
\label{sec_num_time&accu}
We examine the total run time, including coefficient generation and computation time. Larson's (1974) original alternating hyperplane method encounters difficulties with heterogeneous cases due to normalization issues, which prevent convergence. Therefore, we compare Algorithm \ref{algo-BnDHeter} with a method that solves the problem using a sparse solver, \ins{\texttt{spsolve}}.
\ins{This is an efficient Python package designed for problems with sparse matrix structures. 
For both approaches, the transition rates are generated using the TOUR algorithm, and this generation time is excluded from the run time comparison. All experiments were performed on a machine with a 3.40 GHz AMD Ryzen 16-Core processor and 64 GB of RAM.}
We report results for the two systems in Figure \ref{fig_heter_combined}, as the sparse solver solution times are computationally prohibitive for systems with a greater number of units. 

We perform this comparison across varying numbers of units and system utilizations $(\rho = \frac{\lambda}{\sum_i \nu_i})$, where $\nu_i$ is the service rate of each unit, calculated as the reciprocal of the average service time from the data, with a uniform random adjustment between \ins{$-0.2$} and $0.2$ added to account for unit heterogeneity, as individual unit identities were not recorded in the dataset. We achieve different values of $\rho$ by keeping the arrival rates and number of demand nodes constant while adjusting the service rates to simulate various scenarios. 
To compare the performance of different methods, we define the maximum percentage relative error (MPRE) of a specific performance measurement $\mathcal{M}$ as:
\begin{equation}\label{eqn-MPRE}
    \text{MPRE} = \max _m\left|\frac{\mathcal{M}_m-\hat{\mathcal{M}}_m}{\mathcal{M}_m}\right| \times 100 \%,
\end{equation}
where $\mathcal{M}_m$ and $\hat{\mathcal{M}}_m$ are performance measurements calculated by different methods, and the index $m$ refers to the components of the performance vector. 
Figure \ref{fig_heter_combined} shows both the percentage of total time saved by our algorithm and the MPRE of the steady-state probability distributions within the emergency systems in St. Paul, MN, and Greenville County, SC, respectively. 
We focus on systems with more than 11 units, as the computation times for smaller systems are less than one second.

From Figure \ref{fig_heter_combined}, we observe that the advantage of our algorithm increases as the number of units increases and significantly outperforms the sparse solver. Our algorithm is more than 1000 times faster than the sparse solver for larger problems, with the maximum relative error remaining below 0.5\%. \ins{We note that the proposed method yields exact solutions. The reported percentage errors arise from the convergence threshold in the iterative process and can be further reduced by choosing a smaller threshold.}
We provide the exact computation time of our method and the sparse solver in Table \ref{tab_EC_heter_detail} in the E-Companion. As an example, for $N=15$, the sparse solver takes about 2.4 hours while our solution takes only 6.2 seconds.  

\begin{figure}
    \begin{center}
    \includegraphics[width=1\linewidth]{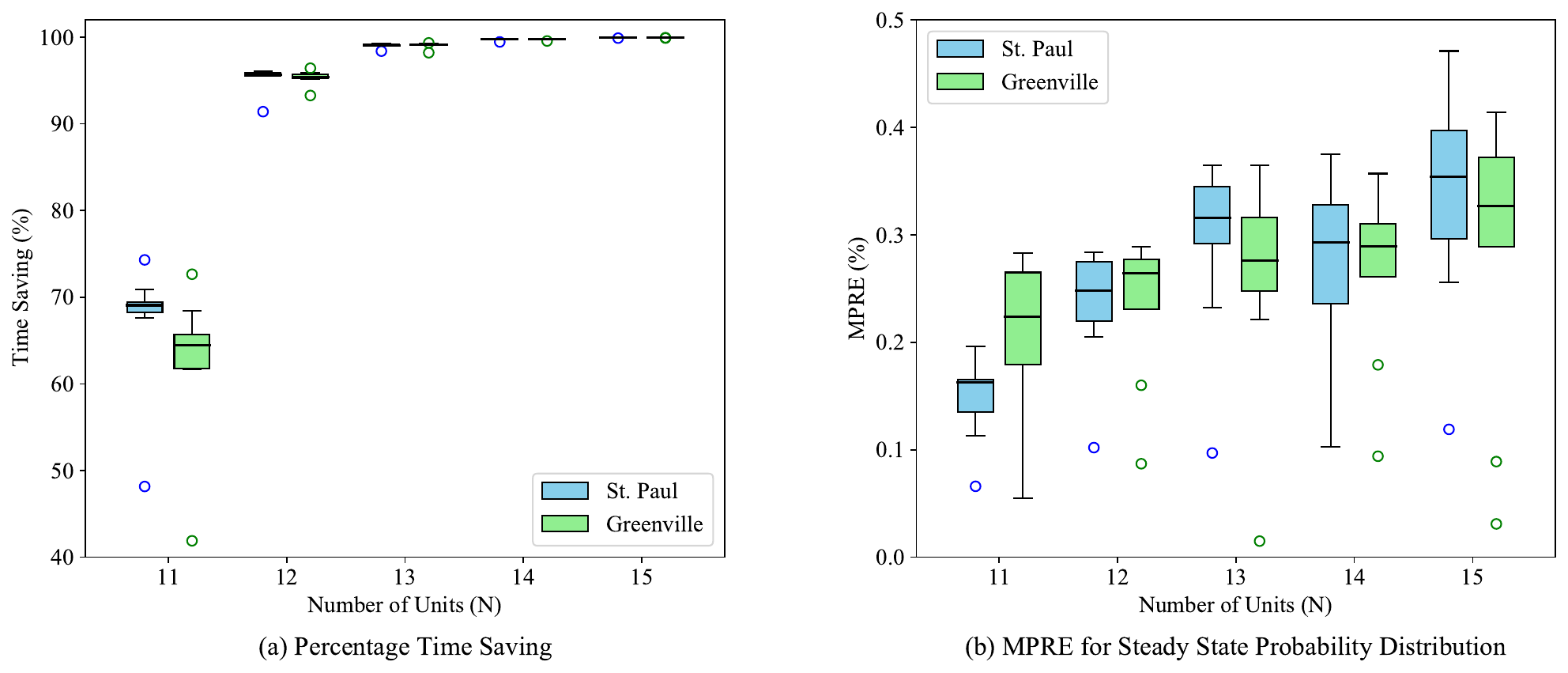}
    \end{center}
    \caption{Percentage Time Savings and Errors for Heterogeneous Cases.}
    {\footnotesize 
    \linespread{1.0}\selectfont 
    \noindent\textit{Notes.} The percentage time savings is calculated as $\frac{T-\hat{T}}{T} \times 100\%$, where $T$ and $\hat{T}$ represent the running times of the sparse solver and our method, respectively. The coefficient generation time is excluded. The MPRE $(\%)$ is computed using Equation \eqref{eqn-MPRE}, where $\mathcal{M}_m$ and $\hat{\mathcal{M}}_m$ correspond to $P\{B_m\}$ and $\hat{P}\{B_m\}$, the steady-state distributions from the sparse solver and our method, respectively. For each box, we compute 9 scenarios by varying the system utilization $\rho$ from 0.1 to 0.9.
    \par}
    \label{fig_heter_combined}
\end{figure}

\begin{figure}
    \begin{center}
    \includegraphics[width=1\linewidth]{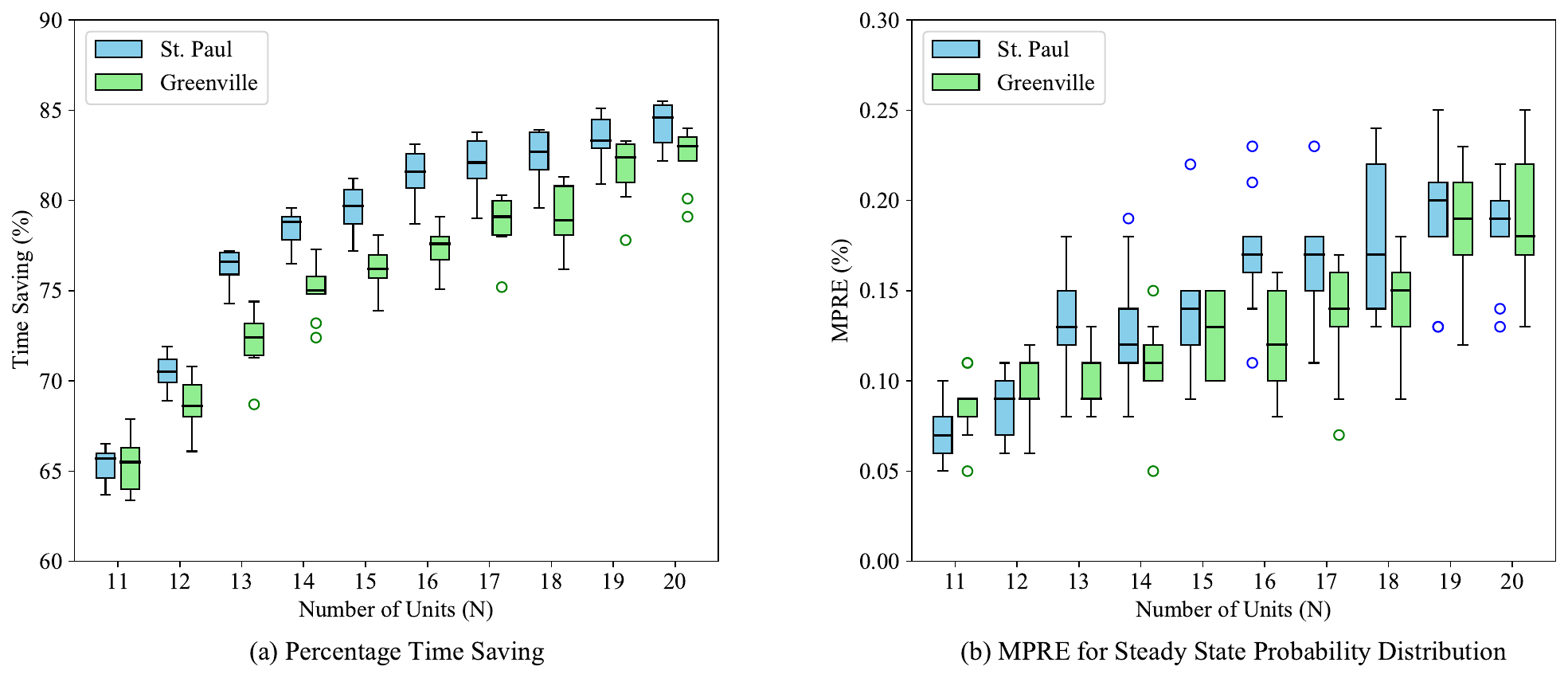}
    \end{center}
    \caption{Percentage Time Savings and Errors for Homogeneous Cases.}
    {\footnotesize 
    \linespread{1.0}\selectfont 
    \noindent\textit{Notes.} The percentage time savings is calculated as $\frac{T-\hat{T}}{T} \times 100\%$, where $T$ and $\hat{T}$ represent the running times of the alternative hyperplane method and our method, respectively. The coefficient generation time is excluded. The MPRE $(\%)$ is computed using Equation \eqref{eqn-MPRE}, where $\mathcal{M}_m$ and $\hat{\mathcal{M}}_m$ correspond to $P\{B_m\}$ and $\hat{P}\{B_m\}$, the steady-state distributions from the two methods, respectively. For each box, we compute 9 scenarios by varying the system utilization $\rho$ from 0.1 to 0.9. 
    \par}
    \label{fig_homo_combined}
\end{figure}

To compare our method with the alternating hyperplane approach in \cite{larson74hypercube}, which is the state-of-the-art method for solving the problem exactly and has been used in a RAND corporation document \citep{larson1975hypercube}, we compare both methods under the setting of uniform service rates for all units. The original version of Larson’s program was implemented in PL/I. For a fair comparison, we reimplemented and updated the code using modern programming techniques. The results are shown in Figure \ref{fig_homo_combined}, and the detailed computation times for both methods are provided in Section~\ref{ssec_extended_homo}. 
From the figure, we observe that as the number of units, $N$, increases, the percentage of total time savings achieved by our algorithm grows, while maintaining a maximum relative error of less than 0.25\% for all cases. 
In the case of 20 units, our algorithm is approximately 83\% faster in terms of total time. The time savings result from both coefficient generation and computation. 

\subsection{Comparison with Discrete-Event Simulation} \label{sec-num_ComparisonDES}

\begin{table}[b]
    \caption{Run Time and MPRE (\%) of the Discrete-Event Simulation and CPU Algorithm vs. Exact Solution.}
    \centering
    \renewcommand\arraystretch{1.2}
    \tabcolsep=6pt
    \footnotesize
    \begin{adjustbox}{center}
    \begin{tabular}{@{\extracolsep{2pt}}c|crrrr|cccc@{}}
        \hline
        & \multicolumn{5}{c|}{Discrete-Event Simulation} & \multicolumn{4}{c}{CPU Algorithm} \\
        \hline
        \# Units & \# Arrivals & \multicolumn{1}{c}{Time} & \multicolumn{3}{c|}{MPRE (\%)} & Time & \multicolumn{3}{c}{MPRE (\%)} \\
        \cline{4-6} \cline{8-10}
         ($N$) & ($\times 10^7$) & \multicolumn{1}{c}{(s)} & \multicolumn{1}{c}{$\rho=0.1$} & \multicolumn{1}{c}{$\rho=0.5$} & \multicolumn{1}{c|}{$\rho=0.9$} & \multicolumn{1}{c}{(s)} & \multicolumn{1}{c}{$\rho=0.1$} & \multicolumn{1}{c}{$\rho=0.5$} & \multicolumn{1}{c}{$\rho=0.9$} \\
        \hline
        8 & 0.5 & $118.5$ & $38.8 \; \; \; \; (\pm 5.9)$ & $3.1 \ (\pm  0.2)$ & $5.2 \ (\pm  0.4)$ & $0.3$ & $0.04$ & $0.08$ & $0.10$\\
        9 & 1 & $238.1$ & $59.2 \; \; \;\;  (\pm  4.3)$ & $3.5 \ (\pm  0.2)$ & $6.0 \ (\pm  0.2)$ & $0.5$ & $0.03$ & $0.15$ & $0.11$\\
        10 & 2 & $485.2$ & $93.1 \ (\pm  12.0)$ & $3.9 \ (\pm  0.3)$ & $7.0 \ (\pm  0.3)$ & $0.8$ & $0.07$ & $0.17$ & $0.16$\\
        11 & 4 & $977.8$ & $198.1 \ (\pm  22.8)$ & $4.4 \ (\pm  0.2)$ & $9.1 \ (\pm  0.3)$ & $1.0$ & $0.07$ & $0.18$ & $0.16$ \\
        12 & 8 & $1963.7$ & $474.0 \ (\pm  77.1)$ & $5.1 \ (\pm  0.2)$ & $10.9 \ (\pm  0.7)$ & $1.5$ & $0.10$ & $0.27$ & $0.22$ \\
        \hline
    \end{tabular}
    \end{adjustbox}
    \begin{tablenotes}
        \scriptsize
        \item \textit{Note.}  MPRE is computed from the sparse solver’s solution and the averaged probability distributions from the simulations and CPU algorithm. The reported run time is the mean time across different $\rho$ values.  
    \end{tablenotes} 
    \label{tab_simu_com}
\end{table}

One major advantage of the CPU algorithm is its exact solution accuracy and convergence guarantee. As mentioned earlier, approximation algorithms have been developed (e.g., \citealt{larson75approx}, \citealt{jarvis85approximating}, \citealt{budge09technical}, and \citealt{hua2022cross}) to solve these systems quickly. While these approximation methods are highly effective and practical due to their speed, they lack guarantees for convergence and accuracy, which means that exact methods or discrete-event simulation are required to validate the results of these algorithms. 
Therefore, we compare the CPU algorithm to a discrete-event simulation. For the simulation, we run 20 replications for each case. The steady-state probability distribution is estimated based on the proportion of time the system spends in each state. For comparison, we fix the total number of arrivals and evaluate the accuracy of the solution and simulation time against the CPU algorithm. The results are shown in Table \ref{tab_simu_com}. Compared with discrete-event simulation, our CPU algorithm runs more than 500 times faster while achieving higher accuracy. When system utilization $\rho = 0.1$, the MPRE from the discrete-event simulation is relatively large, as the low arrival rate limits the system’s ability to sufficiently explore the state space. As the system size increases, the simulation accuracy deteriorates, and its performance also becomes sensitive to the system load $\rho$. 
In contrast, the CPU algorithm maintains errors below 0.3\% consistently within a significantly shorter run time.

\subsection{Results on Parallel Computing}
\label{sec_distributed}

\textcolor{black}{We} evaluate the performance of our parallel algorithm. Solving Larson’s (1974) hypercube exactly is limited to small problems because of memory limitations for storing coefficients and the computational time required. As a result, approximation methods or simulation are used instead. In contrast, the Parallel CPU Algorithm (Algorithm \ref{algo-parallel}) effectively solves large-scale problems, as shown in Table \ref{tab_para}.

Table \ref{tab_para} presents the run time of Algorithm \ref{algo-parallel} for 21 to 30 units.
For each $N$, we run the parallel algorithm with a different number of workers to assess its degree of parallelization. In each case, we use Algorithm~\ref{alg:transition_generation} to generate the transition rates during the first iteration. These data are then stored in an external cache with a large capacity. In subsequent iterations, the cached data are directly loaded rather than regenerating them by re-executing Algorithm~\ref{alg:transition_generation}. Table~\ref{tab_para} reports both the computation time for the first iteration and the total execution time of Algorithm~\ref{algo-parallel}.

\begin{table}[bp]
    \centering
    \renewcommand\arraystretch{1.2}
    \tabcolsep=3.5pt 
    \caption{Run Time (mins) for Parallel CPU.}
    \footnotesize
    \label{tab_para}
    \begin{tabular}{@{}c | rrcrrcrrcrrcrrcrr | cc@{}}
        \hline
        \multirow{2}{*}{\# Units} & \multicolumn{17}{c|}{\# Workers } &  \multicolumn{2}{c}{\multirow{2}{*}{$P$}} \\
        \cline{2-18}
        & \multicolumn{2}{c}{2} & & \multicolumn{2}{c}{4} & & \multicolumn{2}{c}{6} & & \multicolumn{2}{c}{8} & & \multicolumn{2}{c}{10} & & \multicolumn{2}{c|}{12} &  &  \\
        \cline{2-3} \cline{5-6} \cline{8-9} \cline{11-12} \cline{14-15} \cline{17-20}
        ($N$) & First & Total & & First & Total & & First & Total & & First & Total & & First & Total & & First & Total &  First & Total \\
        \hline
        21 & 2.6 & 4.2 & & 1.5 & 2.4 & & 1.1 & 1.7 & & 0.9 & 1.5 & & 0.7 & 1.3 & & 0.7 & 1.2 & 94.2\% & 92.8\% \\
        22 & 5.4 & 8.2 & & 3.0 & 4.7 & & 1.9 & 3.2 & & 1.8 & 3.0 & & 1.5 & 2.6 & & 1.4 & 2.4 & 94.9\% & 92.5\% \\
        23 & 8.1 & 14.9 & & 4.4 & 8.3 & & 3.0 & 5.8 & & 2.6 & 5.4 & & 2.3 & 4.7 & & 2.1 & 4.3 & 95.1\% & 92.8\% \\
        24 & 16.0 & 30.5 & & 8.5 & 17.0 & & 6.4 & 13.4 & & 5.1 & 10.6 & & 4.6 & 9.6 & & 4.0 & 9.0 & 94.9\% & 92.1\% \\
        25 & 31.4 & 58.8 & & 17.1 & 33.6 & & 11.1 & 23.2 & & 9.4 & 20.3 & & 8.4 & 18.8 & & 6.7 & 16.5 & 96.7\% & 92.9\% \\
        26 & 63.2 & 121.3 & & 34.2 & 68.9 & & 24.4 & 51.5 & & 18.5 & 41.6 & & 15.2 & 36.5 & & 14.4 & 35.5 & 96.7\% & 92.7\% \\
        27 & 131.8 & 261.9 & & 69.5 & 143.7 & & 45.9 & 101.6 & & 39.0 & 90.2 & & 34.8 & 84.6 & & 27.9 & 75.0 & 97.0\% & 92.9\% \\
        28 & 279.3 & 579.7 & & 147.0 & 333.0 & & 101.3 & 244.4 & & 83.1 & 209.0 & & 67.1 & 184.3 & & 62.7 & 174.5 & 97.0\% & 91.9\% \\
        29 & 555.5 & 1212.7 & & 290.0 & 647.0 & & 206.6 & 479.8 & & 159.6 & 407.6 & & 136.6 & 374.7 & & 120.8 & 358.1 & 97.1\% & 93.0\% \\
        30 & 1056.3 & 2402.8 & & 564.4 & 1421.3 & & 416.4 & 1105.2 & & 356.1 & 931.1 & & 245.2 & 778.1 & & 215.1 & 717.0 & 96.5\% & 91.0\% \\
        \hline
    \end{tabular}
    \begin{tablenotes}
        \scriptsize
        \item \textit{Note.} First indicates the run time of the first iteration, and Total denotes the overall execution time. The column labeled $P$ shows the proportion of the algorithm that is parallelizable, as determined by Lemma~\ref{lemma-Amdahl}.
    \end{tablenotes}
\end{table}

As shown in Table~\ref{tab_para}, our results demonstrate strong parallelization efficiency in both coefficient generation and probability iteration. For all experiments, the coefficient of determination $R^2$ from the linear regression fitting of Amdahl’s law equals one for every problem size $N$, indicating an excellent fit. According to Amdahl's law (Lemma~\ref{lemma-Amdahl}), over 91\% of Algorithm~\ref{algo-parallel} is parallelizable, as reported in the column labeled $P$. In contrast, previous methods could only handle up to 20 units due to memory limitations. We note that all results were obtained on a personal computer. With additional computational resources, more workers can be deployed to handle larger problems under the same parallel framework and convergence guarantees.

\color{black}

\subsection{Results for General Service Time Distributions}
\label{ssec_general_rate}


\begin{table}[t]
    \centering
    \renewcommand\arraystretch{1.2}
    \tabcolsep=4pt
    \caption{MPRE (\%) of MRT and Utilization with Simulation for Various Service Time Distributions.}
    \scriptsize
    \begin{tabular}{@{\extracolsep{2pt}}cccccccc} 
        \hline
        \multirow{2}{*}{Distribution} & \multirow{2}{*}{$\rho$} & \multicolumn{6}{c}{MPRE (\%) of Mean Response Time (MRT)}  \\
        \cline{3-8}
        & & $N=8$ & $N=9$ & $N=10$ & $N=11$ & $N=12$ & Average\\
        \hline
        \multirow{3}{*}{\shortstack{Exponential \\ ($\nu_i$)}}
        & 0.1 & $0.008 \ (\pm 0.002)$ & $0.004 \ (\pm 0.001)$ & $0.003 \ (\pm 0.001)$ & $0.003 \ (\pm 0.001)$ & $0.004 \ (\pm 0.001)$ & \\
        & 0.5 & $0.024 \ (\pm 0.010)$ & $0.018 \ (\pm 0.006)$ & $0.010 \ (\pm 0.004)$ & $0.007 \ (\pm 0.003)$ & $0.007 \ (\pm 0.003)$ & $0.012$ \\
        & 0.9 & $0.022 \ (\pm 0.006)$ & $0.019 \ (\pm 0.005)$ & $0.019 \ (\pm 0.006)$ & $0.009 \ (\pm 0.002)$ & $0.008 \ (\pm 0.003)$ & \\ 
        \hline

        \multirow{3}{*}{\shortstack{Uniform \\ $\left(\left[\frac{0.75}{\nu_i}, \frac{1.25}{\nu_i}\right]\right)$}}
        & 0.1 & $0.006 \ (\pm 0.002)$ & $0.003 \ (\pm 0.001)$ & $0.003 \ (\pm 0.001)$ & $0.004 \ (\pm 0.001)$ & $0.003 \ (\pm 0.001)$ & \\
        & 0.5 & $0.022 \ (\pm 0.007)$ & $0.012 \ (\pm 0.005)$ & $0.022 \ (\pm 0.005)$ & $0.008 \ (\pm 0.003)$ & $0.015 \ (\pm 0.003)$ & $0.016$ \\
        & 0.9 & $0.027 \ (\pm 0.009)$ & $0.019 \ (\pm 0.006)$ & $0.043 \ (\pm 0.007)$ & $0.078 \ (\pm 0.004)$ & $0.024 \ (\pm 0.004)$ & \\
        \hline

        \multirow{3}{*}{\shortstack{Log-normal \\$\left(\sigma=1, \mu = -\ln{\nu_i} - \frac{1}{2}\right)$}}
        & 0.1 & $0.009 \ (\pm 0.003)$ & $0.004 \ (\pm 0.001)$ & $0.005 \ (\pm 0.001)$ & $0.004 \ (\pm 0.001)$ & $0.004 \ (\pm 0.002)$ & \\
        & 0.5 & $0.031 \ (\pm 0.011)$ & $0.021 \ (\pm 0.006)$ & $0.009 \ (\pm 0.004)$ & $0.012 \ (\pm 0.004)$ & $0.008 \ (\pm 0.003)$ & $0.014$\\
        & 0.9 & $0.027 \ (\pm 0.008)$ & $0.023 \ (\pm 0.006)$ & $0.020 \ (\pm 0.006)$ & $0.012 \ (\pm 0.004)$ & $0.012 \ (\pm 0.003)$ & \\
        \hline

        \multirow{3}{*}{\shortstack{Gamma \\ $\left(\alpha =0.5, \beta = \frac{1}{0.5\nu_i}\right)$}}
        & 0.1 & $0.010 \ (\pm 0.004)$ & $0.004 \ (\pm 0.001)$ & $0.004 \ (\pm 0.001)$ & $0.004 \ (\pm 0.001)$ & $0.004 \ (\pm 0.001)$ & \\
        & 0.5 & $0.020 \ (\pm 0.006)$ & $0.017 \ (\pm 0.006)$ & $0.012 \ (\pm 0.005)$ & $0.014 \ (\pm 0.004)$ & $0.006 \ (\pm 0.002)$ & $0.013$ \\
        & 0.9 & $0.026 \ (\pm 0.010)$ & $0.030 \ (\pm 0.009)$ & $0.024 \ (\pm 0.007)$ & $0.010 \ (\pm 0.003)$ & $0.012 \ (\pm 0.004)$ & \\
        \hline

        \multirow{3}{*}{\shortstack{Gamma \\ $\left(\alpha =5, \beta = \frac{1}{5\nu_i}\right)$}}
        & 0.1 & $0.009 \ (\pm 0.002)$ & $0.004 \ (\pm 0.001)$ & $0.003 \ (\pm 0.001)$ & $0.003 \ (\pm 0.001)$ & $0.004 \ (\pm 0.001)$ & \\
        & 0.5 & $0.019 \ (\pm 0.006)$ & $0.019 \ (\pm 0.005)$ & $0.012 \ (\pm 0.004)$ & $0.011 \ (\pm 0.003)$ & $0.014 \ (\pm 0.003)$ & $0.013$ \\
        & 0.9 & $0.020 \ (\pm 0.008)$ & $0.023 \ (\pm 0.006)$ & $0.017 \ (\pm 0.005)$ & $0.031 \ (\pm 0.003)$ & $0.015 \ (\pm 0.003)$ & \\
        \toprule
        \multirow{2}{*}{Distribution} & \multirow{2}{*}{$\rho$} & \multicolumn{6}{c}{MPRE (\%) of Unit Utilization}  \\
        \cline{3-8}
        & & $N=8$ & $N=9$ & $N=10$ & $N=11$ & $N=12$ & Average \\ 
        \hline
        \multirow{3}{*}{\shortstack{Exponential \\ ($\nu_i$)}}
        & 0.1 & $1.06 \ (\pm 0.14)$ & $0.85 \ (\pm 0.12)$ & $0.59 \ (\pm 0.08)$ & $0.54 \ (\pm 0.09)$ & $0.41 \ (\pm 0.07)$ & \\
        & 0.5 & $0.56 \ (\pm 0.08)$ & $0.45 \ (\pm 0.08)$ & $0.34 \ (\pm 0.05)$ & $0.26 \ (\pm 0.03)$ & $0.20 \ (\pm 0.03)$ & $0.41$ \\
        & 0.9 & $0.36 \ (\pm 0.05)$ & $0.26 \ (\pm 0.04)$ & $0.18 \ (\pm 0.02)$ & $0.13 \ (\pm 0.02)$ & $0.10 \ (\pm 0.02)$ & \\
        \hline

        \multirow{3}{*}{\shortstack{Uniform \\ $\left(\left[\frac{0.75}{\nu_i}, \frac{1.25}{\nu_i}\right]\right)$}}
        & 0.1 & $0.69 \ (\pm 0.08)$ & $0.60 \ (\pm 0.08)$ & $0.43 \ (\pm 0.06)$ & $0.34 \ (\pm 0.05)$ & $0.31 \ (\pm 0.03)$ & \\
        & 0.5 & $0.48 \ (\pm 0.05)$ & $0.42 \ (\pm 0.04)$ & $0.45 \ (\pm 0.04)$ & $0.71 \ (\pm 0.05)$ & $0.63 \ (\pm 0.04)$ & $0.47$ \\
        & 0.9 & $0.34 \ (\pm 0.03)$ & $0.38 \ (\pm 0.03)$ & $0.35 \ (\pm 0.02)$ & $0.47 \ (\pm 0.02)$ & $0.57 \ (\pm 0.02)$ & \\
        \hline

        \multirow{3}{*}{\shortstack{Log-normal \\$\left(\sigma=1, \mu = -\ln{\nu_i} - \frac{1}{2}\right)$}}
        & 0.1 & $1.13 \ (\pm 0.17)$ & $0.89 \ (\pm 0.11)$ & $0.76 \ (\pm 0.08)$ & $0.64 \ (\pm 0.10)$ & $0.49 \ (\pm 0.09)$ & \\
        & 0.5 & $0.68 \ (\pm 0.10)$ & $0.54 \ (\pm 0.10)$ & $0.35 \ (\pm 0.04)$ & $0.30 \ (\pm 0.06)$ & $0.25 \ (\pm 0.03)$ & $0.47$ \\
        & 0.9 & $0.40 \ (\pm 0.06)$ & $0.30 \ (\pm 0.04)$ & $0.21 \ (\pm 0.03)$ & $0.14 \ (\pm 0.03)$ & $0.12 \ (\pm 0.02)$ & \\
        \hline

        \multirow{3}{*}{\shortstack{Gamma \\ $\left(\alpha =0.5, \beta = \frac{1}{0.5\nu_i}\right)$}}
        & 0.1 & $1.07 \ (\pm 0.13)$ & $1.00 \ (\pm 0.09)$ & $0.77 \ (\pm 0.09)$ & $0.55 \ (\pm 0.08)$ & $0.52 \ (\pm 0.09)$ & \\
        & 0.5 & $0.69 \ (\pm 0.08)$ & $0.48 \ (\pm 0.07)$ & $0.43 \ (\pm 0.07)$ & $0.34 \ (\pm 0.05)$ & $0.23 \ (\pm 0.03)$ & $0.47$ \\
        & 0.9 & $0.42 \ (\pm 0.07)$ & $0.33 \ (\pm 0.04)$ & $0.26 \ (\pm 0.03)$ & $0.19 \ (\pm 0.02)$ & $0.13 \ (\pm 0.02)$ & \\
        \hline

        \multirow{3}{*}{\shortstack{Gamma \\ $\left(\alpha =5, \beta = \frac{1}{5\nu_i}\right)$}}
        & 0.1 & $0.84 \ (\pm 0.10)$ & $0.58 \ (\pm 0.07)$ & $0.50 \ (\pm 0.07)$ & $0.40 \ (\pm 0.05)$ & $0.30 \ (\pm 0.05)$ & \\
        & 0.5 & $0.45 \ (\pm 0.07)$ & $0.44 \ (\pm 0.06)$ & $0.29 \ (\pm 0.03)$ & $0.38 \ (\pm 0.06)$ & $0.31 \ (\pm 0.03)$ & $0.37$ \\
        & 0.9 & $0.27 \ (\pm 0.04)$ & $0.23 \ (\pm 0.02)$ & $0.18 \ (\pm 0.02)$ & $0.21 \ (\pm 0.02)$ & $0.19 \ (\pm 0.02)$ & \\
        \hline
    \end{tabular}
    \label{tab_simulation}
    \begin{tablenotes}
        \scriptsize
        \item \textit{Note.} In the simulation, we run 20 replications for each case, with each replication simulating $10^6$ arrivals. The MPRE (\%) is computed using the mean response time  or unit utilization obtained from the simulation ($\mathcal{M}_m$) and our CPU algorithm ($\hat{\mathcal{M}}_m$). The last column of each distribution reports the average MPRE (\%) across all 15 cases. All errors remain below 1.4\%.
    \end{tablenotes}
\end{table}

While the model is based on a Markovian system for analytical tractability, it is well established that the steady-state behavior of zero-queue systems is insensitive to the specific service time distribution (\citealt{ross2014introduction}, page 543). In this section, we evaluate the model's performance across various service time distributions, including log-normal, gamma with different parameters, and uniform distributions. This greatly increases the applicability of our solution to real-world systems. We note that the insensitivity property does not hold for finite-buffer systems \citep{kimura1996transform, keilson1966ergodic}. Therefore, in this section, we restrict our attention to the zero-queue system.

For non-Markovian systems, steady-state probability distributions are not suitable as evaluation metrics since they are only valid for Markovian systems. In non-Markovian systems, the concept of states as defined in Markovian models becomes meaningless due to the absence of the memoryless property. Therefore, we use mean response time (MRT) and utilization for each unit as performance metrics. These are important metrics in emergency service systems and are commonly used to evaluate system performance \citep{erkut2008ambulance, budge09technical}. 

For comparison, we set the means of each service distribution to be identical. \textcolor{black}{The lognormal distribution has the same variance as the exponential distribution, while both the uniform and gamma distributions with $\beta = \tfrac{1}{5\nu_i}$ have smaller variances. In contrast, the gamma distribution with $\beta = \tfrac{1}{0.5\nu_i}$ has a larger variance.} 
The results, presented in Table \ref{tab_simulation}, cover scenarios with $N=8$ to 12 units under varying system loads. We compare our solution to simulation results obtained by running the simulation over an extended period. Across all tested distributions, our method remains highly accurate, yielding percentage errors below 0.08\% for mean response time and 1.4\% for unit utilization, with consistently low standard deviations.

\section{Conclusion}
\color{black}
\label{sec_conclusion}
\ins{In this paper, we propose new and faster sequential and parallel CPU algorithms for exactly solving Larson’s (1974) spatial hypercube queueing model. Both algorithms achieve geometric convergence to the exact solution. Our paper is the first to improve upon the performance of the classic  \cite{larson74hypercube} hypercube queueing model since its introduction 50 years ago. The model assumes identical service rates for all servers, whereas we consider both homogeneous and heterogeneous service rate settings. Under the homogeneous setting where all servers have identical service rates, our CPU algorithm reduces computation time by more than 80\% compared with the \cite{larson74hypercube} algorithm. For example, it saves about 15 minutes for a problem with 20 units, while larger problems with more than 20 units are intractable for the original approach. Our method can now handle larger instances.}

\ins{When service rates vary across servers, the Larson approach cannot handle the model. For comparison, we benchmark our algorithm against a sparse solver and a discrete-event simulation. Our method runs nearly 1,000 times faster than the sparse solver and more than 500 times faster than the discrete-event simulation, while achieving higher accuracy. We also develop a parallel version of our algorithm that achieves over 91\% parallelization efficiency. With 12 workers, the parallel algorithm provides an additional eightfold speedup over the sequential version and enables the solution of substantially larger problems. The performance gain would further increase with additional computational resources.}
We have open-sourced both our algorithm and Larson’s original approach, which is reprogrammed in Python code to provide accessibility to other researchers.


\ACKNOWLEDGMENT{\ins{This research was supported by the National Natural Science Foundation of China [72031006].}}







\bibliographystyle{informs2014}
\bibliography{references.bib}

@article{erkut2008ambulance,
  title={Ambulance location for maximum survival},
  author={Erkut, Erhan and Ingolfsson, Armann and Erdo{\u{g}}an, G{\"u}ne{\c{s}}},
  journal={Naval Research Logistics (NRL)},
  volume={55},
  number={1},
  pages={42--58},
  year={2008},
  publisher={Wiley Online Library}
}

@article{rastpour2020modeling,
  title={Modeling yellow and red alert durations for ambulance systems},
  author={Rastpour, Amir and Ingolfsson, Armann and Kolfal, Bora},
  journal={Production and Operations Management},
  volume={29},
  number={8},
  pages={1972--1991},
  year={2020},
  publisher={Wiley Online Library}
}

@article{alanis2013markov,
  title={A Markov chain model for an EMS system with repositioning},
  author={Alanis, Ramon and Ingolfsson, Armann and Kolfal, Bora},
  journal={Production and operations management},
  volume={22},
  number={1},
  pages={216--231},
  year={2013},
  publisher={Wiley Online Library}
}

@book{larson1975hypercube,
  title={Hypercube queuing model: User's Manual},
  author={Larson, Richard C},
  year={1975},
  publisher={Rand}
}

@article{burwell1993modeling,
  title={Modeling co-located servers and dispatch ties in the hypercube model},
  author={Burwell, Timothy H and Jarvis, James P and McKnew, Mark A},
  journal={Computers \& Operations Research},
  volume={20},
  number={2},
  pages={113--119},
  year={1993},
  publisher={Elsevier}
}

@article{iannoni2023review,
  title={A review on hypercube queuing model's extensions for practical applications},
  author={Iannoni, Ana P and Morabito, Reinaldo},
  journal={Socio-Economic Planning Sciences},
  pages={101677},
  year={2023},
  publisher={Elsevier}
}

@article{hua2022cross,
  title={Cross-Trained Fire-Medics Respond to Medical Calls and Fire Incidents: A Fast Algorithm for a Three-State Spatial Queuing Problem},
  author={Hua, Cheng and Swersey, Arthur J},
  journal={Manufacturing \& Service Operations Management},
  volume={24},
  number={6},
  pages={3177--3192},
  year={2022},
  publisher={INFORMS}
}

@article{iannoni2007multiple,
  title={A multiple dispatch and partial backup hypercube queuing model to analyze emergency medical systems on highways},
  author={Iannoni, Ana Paula and Morabito, Reinaldo},
  journal={Transportation research part E: logistics and transportation review},
  volume={43},
  number={6},
  pages={755--771},
  year={2007},
  publisher={Elsevier}
}

@article{iannoni2009optimization,
  title={An optimization approach for ambulance location and the districting of the response segments on highways},
  author={Iannoni, Ana Paula and Morabito, Reinaldo and Saydam, Cem},
  journal={European Journal of Operational Research},
  volume={195},
  number={2},
  pages={528--542},
  year={2009},
  publisher={Elsevier}
}

@article{mcewen1974patrol,
  title={Patrol planning in the Rotterdam police department},
  author={McEwen, Tom and Larson, Richard C},
  journal={Journal of Criminal Justice},
  volume={2},
  number={3},
  pages={235--238},
  year={1974},
  publisher={Elsevier}
}

@phdthesis{chelst1975quantitative,
  title={Quantitative models for police patrol deployment.},
  author={Chelst, Kenneth Richard},
  year={1975},
  school={Massachusetts Institute of Technology}
}

@incollection{larson2013hypercube,
  title={Hypercube queueing model},
  author={Larson, Richard C},
  booktitle={Encyclopedia of operations research and management science},
  publisher={Springer New York, NY},
  editor = {Saul I. Gass and Michael C. Fu},
  pages={733--739},
  year={2013},
  edition = 3
}

@article{ghobadi2019hypercube,
  title={Hypercube queuing models in emergency service systems: A state-of-the-art review},
  author={Ghobadi, Maryam and Arkat, Jamal and Tavakkoli-Moghaddam, Reza},
  journal={Scientia Iranica},
  volume={26},
  number={2},
  pages={909--931},
  year={2019},
  publisher={Sharif University of Technology}
}

@article{berman1982median,
  title={The median problem with congestion},
  author={Berman, Oded and Larson, Richard C},
  journal={Computers \& Operations Research},
  volume={9},
  number={2},
  pages={119--126},
  year={1982},
  publisher={Elsevier}
}

@article{larson75approx,
 author = {Richard C. Larson},
 journal = {Operations Research},
 number = {5},
 pages = {845-868},
 publisher = {INFORMS},
 title = {Approximating the Performance of Urban Emergency Service Systems},
 volume = {23},
 year = {1975}
}

@article{larson74hypercube,
  title={A hypercube queuing model for facility location and redistricting in urban emergency services},
  author={Larson, Richard C},
  journal={Computers \& Operations Research},
  volume={1},
  number={1},
  pages={67--95},
  year={1974},
  publisher={Elsevier}
}

@article{budge09technical,
  title={Technical Note-Approximating Vehicle Dispatch Probabilities for Emergency Service Systems with Location-Specific Service Times and Multiple Units per Location},
  author={Budge, Susan and Ingolfsson, Armann and Erkut, Erhan},
  journal={Operations Research},
  volume={57},
  number={1},
  pages={251--255},
  year={2009},
  publisher={Informs}
}

@article{jarvis85approximating,
  title={Approximating the equilibrium behavior of multi-server loss systems},
  author={Jarvis, James P},
  journal={Management Science},
  volume={31},
  number={2},
  pages={235--239},
  year={1985},
  publisher={INFORMS}
}

@article{chelst81multiple,
  title={Multiple unit dispatches in emergency services: models to estimate system performance},
  author={Chelst, Kenneth and Barlach, Ziv},
  journal={Management Science},
  volume={27},
  number={12},
  pages={1390--1409},
  year={1981},
  publisher={INFORMS}
}

@article{mendoncca2001analysing,
  title={Analysing emergency medical service ambulance deployment on a {B}razilian highway using the hypercube model},
  author={Mendon{\c{c}}a, FC and Morabito, R},
  journal={Journal of the Operational Research Society},
  volume={52},
  number={3},
  pages={261--270},
  year={2001},
  publisher={Springer}
}

@article{budge2010empirical,
  title={Empirical analysis of ambulance travel times: the case of Calgary emergency medical services},
  author={Budge, Susan and Ingolfsson, Armann and Zerom, Dawit},
  journal={Management Science},
  volume={56},
  number={4},
  pages={716--723},
  year={2010},
  publisher={INFORMS}
}

@inproceedings{amdahl1967validity,
  title={Validity of the single processor approach to achieving large scale computing capabilities},
  author={Amdahl, GM},
  booktitle={AFIPS Spring Joint Computer Conference Proceeding},
  volume={30},
  pages={483--485},
  year={1967}
}

@article{luo2015fully,
  title={Fully sequential procedures for large-scale ranking-and-selection problems in parallel computing environments},
  author={Luo, Jun and Hong, L Jeff and Nelson, Barry L and Wu, Yang},
  journal={Operations Research},
  volume={63},
  number={5},
  pages={1177--1194},
  year={2015},
  publisher={INFORMS}
}

@book{cormen2022introduction,
  title={Introduction to algorithms},
  author={Cormen, Thomas H and Leiserson, Charles E and Rivest, Ronald L and Stein, Clifford},
  year={2022},
  publisher={MIT press}
}

@book{ross2014introduction,
  title={Introduction to probability models},
  author={Ross, Sheldon M},
  year={2014},
  publisher={Academic press}
}

@article{beojone2021efficient,
  title={An efficient exact hypercube model with fully dedicated servers},
  author={Beojone, Caio Vitor and M{\'a}ximo de Souza, Regiane and Iannoni, Ana Paula},
  journal={Transportation Science},
  volume={55},
  number={1},
  pages={222--237},
  year={2021},
  publisher={INFORMS}
}

@article{zhu2022data,
  title={Data-driven optimization for Atlanta police-zone design},
  author={Zhu, Shixiang and Wang, He and Xie, Yao},
  journal={INFORMS Journal on Applied Analytics},
  volume={52},
  number={5},
  pages={412--432},
  year={2022},
  publisher={INFORMS}
}

@article{maxwell2010approximate,
  title={Approximate dynamic programming for ambulance redeployment},
  author={Maxwell, Matthew S and Restrepo, Mateo and Henderson, Shane G and Topaloglu, Huseyin},
  journal={INFORMS Journal on Computing},
  volume={22},
  number={2},
  pages={266--281},
  year={2010},
  publisher={INFORMS}
}

@article{jenkins2021approximate,
  title={Approximate dynamic programming for military medical evacuation dispatching policies},
  author={Jenkins, Phillip R and Robbins, Matthew J and Lunday, Brian J},
  journal={INFORMS Journal on Computing},
  volume={33},
  number={1},
  pages={2--26},
  year={2021},
  publisher={INFORMS}
}

@article{kimura1996transform,
  title={A transform-free approximation for the finite capacity M/G/s queue},
  author={Kimura, Toshikazu},
  journal={Operations Research},
  volume={44},
  number={6},
  pages={984--988},
  year={1996},
  publisher={INFORMS}
}

@article{keilson1966ergodic,
  title={The ergodic queue length distribution for queueing systems with finite capacity},
  author={Keilson, Julian},
  journal={Journal of the Royal Statistical Society Series B: Statistical Methodology},
  volume={28},
  number={1},
  pages={190--201},
  year={1966},
  publisher={Oxford University Press}
}

@article{blelloch1996programming,
  title={Programming parallel algorithms},
  author={Blelloch, Guy E},
  journal={Communications of the ACM},
  volume={39},
  number={3},
  pages={85--97},
  year={1996},
  publisher={ACM New York, NY, USA}
}

@book{padua2011encyclopedia,
  title={Encyclopedia of parallel computing},
  author={Padua, David},
  year={2011},
  publisher={Springer Science \& Business Media}
}

%
%
%
\clearpage
\renewcommand{\thepage}{ec\arabic{page}}  
\setcounter{page}{1}

\begin{appendices}
\renewcommand{\thesection}{EC.\arabic{section}}   
\renewcommand{\thetable}{EC.\arabic{table}}   
\renewcommand{\thefigure}{EC.\arabic{figure}}
\renewcommand{\theequation}{EC.\arabic{equation}}
\renewcommand{\thetheorem}{EC.\arabic{theorem}}
\renewcommand{\theproposition}{EC.\arabic{proposition}}
\renewcommand{\thelemma}{EC.\arabic{lemma}}
\renewcommand{\thealgorithm}{EC.\arabic{algorithm}}
\counterwithin{equation}{section}
\counterwithin{table}{section}
\counterwithin{figure}{section}
\setcounter{section}{0}
\setcounter{equation}{0}
\setcounter{table}{0}
\setcounter{figure}{0}
\setcounter{lemma}{0}
\setcounter{theorem}{0}
\setcounter{algorithm}{0}

\vspace*{0.1cm}
\begin{center}
  \large\textbf{E-Companion --- ``A Geometrically Convergent Solution to Spatial Hypercube Queueing Models''}\\
\end{center}
\vspace*{0.3cm}

\section{Miscellaneous Proofs}
\label{app_proof}

\subsection{Proof of Proposition \ref{lemma_bnd}}
\label{app_thm_equal_2_3}
\proof{Proof.}
We show that the conditional probability formulation is a transformation from the two-dimensional formulation. First, we express the joint probability $P\{n, B_m\}$ as 
\begin{eqnarray}
P\{n, B_m\} = p(n)\cdot p_n(B_m). \label{eq:conditional}
\end{eqnarray}
Next, the balance equations \eqref{eq:transformer_bal_eq_heter} become, for $m\in \mathscr{C}_{n}$ and $n=0,1\ldots,N$, 
\begin{eqnarray}
    p(n)p_n(B_m)\left(\lambda_{m}+ \mu_m \right)=\sum_{l \in \mathscr{C}_{n-1}} p(n-1)p_{n-1}(B_l) \lambda_{lm}+\sum_{l \in \mathscr{C}^m_{n+1}} p(n+1)p_{n+1}(B_l) \mu_{lm}. \label{eq:new_bal}
\end{eqnarray}
We note that $p(n)$ represents the steady-state probability in the birth-death process, and is determined by
\begin{equation*}
p(n)= \frac{1}{G} \prod_{s=1}^{n} \frac{\lambda(s-1)}{\mu(s)}, \quad \mbox{where} \quad G = 1+\sum_{n=1}^{N} \prod_{s=1}^{n} \frac{\lambda(s-1)}{\mu(s)}.
\end{equation*}
Here, $G$ serves as a normalization factor for the birth-death process. 
Applying the formula of $p(n)$ to the above equation \eqref{eq:new_bal}, we obtain
\begin{eqnarray*}
   \frac{1}{G} \prod_{s=1}^{n} \frac{\lambda(s-1)}{\mu(s)}  p_n(B_m)\left(\lambda_{m}+\mu_m \right) = \sum_{l \in \mathscr{C}_{n-1}}\frac{1}{G} \prod_{s=1}^{n-1} \frac{\lambda(s-1)}{\mu(s)}  p_{n-1}(B_l) \lambda_{lm} \nonumber\\
     +\sum_{l \in \mathscr{C}^m_{n+1}}\frac{1}{G} \prod_{s=1}^{n+1} \frac{\lambda(s-1)}{\mu(s)}  p_{n+1}(B_l) \mu_{lm}. 
\end{eqnarray*}
After canceling $G$ and terms in the product, the equation simplifies to
\begin{eqnarray*}
    p_n(B_m)\left(\lambda_{m}+\mu_m \right)=\sum_{l \in \mathscr{C}_{n-1}} p_{n-1}(B_l) \frac{\mu(n)}{\lambda(n-1)}\lambda_{lm}  + \sum_{l \in \mathscr{C}^m_{n+1}} p_{n+1}(B_l) \frac{\lambda(n)}{\mu(n+1)}\mu_{lm}. 
\end{eqnarray*}
By dividing both sides of this equation by $\lambda_m+\mu_m$, we obtain 
\begin{equation*}
    p_n(B_m)=\sum_{l \in \mathscr{C}_{n-1}} p_{n-1}(B_l) \frac{\mu(n)}{\lambda(n-1)}\frac{\lambda_{l m}}{\lambda_{m}+\mu_m} + \frac{\mu_{lm}}{\lambda_{m}+\mu_m}\sum_{l \in \mathscr{C}^m_{n+1}} p_{n+1}(B_l) \frac{\lambda(n)}{\mu(n+1)}.
\end{equation*}
We also require that the conditional probabilities sum to 1 for each layer $\mathscr{C}_n$, for $n=0,1,\ldots, N$, 
\begin{equation*}
    \sum_{m \in \mathscr{C}_n} p_n(B_m) = 1.
\end{equation*}
In the concluding step, we determine the probability distribution of steady states using the conditional probabilities $p_n(B_m)$. The relationship is expressed as
\begin{equation*}
    P\{B_m\} = P\{w(B_m), B_m\} = p(w(B_m))\cdot p_{w(B_m)}(B_m).
\end{equation*}
where we replace $n=w(B_m)$ in Equation \eqref{eq:conditional} in the above equation. 
\Halmos
\endproof

\subsection{Proof of Lemma \ref{lemma:conv&sum}}
\label{ssec:proof_lemma_hetero}
\proof{Proof.} 
We prove this lemma by induction. For $k=1$, we initialize $p_n^1(B_m) = p_n^{1,1}(B_m) =\frac{1}{{N \choose n}}$. This ensures $\sum_{m \in \mathscr{C}_n} p_n^{1,1}(B_m) = 1$ for all $n=0,1,\ldots, N$.
    
Assuming the statement holds for $k-1$, we proceed with the inductive hypothesis to verify the properties for $k$.
    Note that $\lambda_m = \sum_{l} \lambda_{ml}=\lambda$, which is valid for every state except when all units are busy. This is attributed to the sum of transition rates out of state $B_m$ upon arrival being equal to the total system arrival rate $\lambda$. Subsequently, we derive the following results:
    \begin{enumerate}[(i)]
        \item For $n=0$, we have only one state, $B_0$. Using Equation \eqref{eq:updateHeter}, we can simplify this as:
        \begin{eqnarray*} 
            p_0^{k,r}(B_0) \lambda = \sum_{l \in \mathscr{C}_{1}} p_{1}^{k-1}(B_l) \frac{\lambda^{k-1}(0)}{\mu^{k-1}(1)} \mu_{l0} = \frac{\lambda^{k-1}(0)}{\mu^{k-1}(1)} \sum_{l \in \mathscr{C}_{1}} p_{1}^{k-1}(B_l) \mu_{l0}. 
        \end{eqnarray*}
        By utilizing Equations \eqref{eq:lambdaHeter} and  \eqref{eq:muHeter}, we have:
        \begin{eqnarray*} 
            p_0^{k,r}(B_0) \lambda = \frac{\lambda^{k-1}(0)}{\mu^{k-1}(1)} \mu^{k-1}(1) = \lambda^{k-1}(0) = p_{0}^{k-1}\left(B_0\right) \lambda. 
        \end{eqnarray*}
        Thus, $p_0^{k,r}(B_0) = p_0^{k-1}(B_0)$ for all $r = 1,2,\ldots$. Given our assumption that the two statements are true for $k-1$, they are also valid for $k$ when $n=0$.
        \item Assume the statements hold for $n - 1$, with $p_{n-1}^{k,r}(B_m)$, $\lambda^{k,r}(n)$ converging to $p_{n-1}^{k}(B_m)$ and $\lambda^{k}(n)$ respectively. Then, using Equation \eqref{eq:updateHeter} (i.e., subtracting it by replacing $r$ with $r-1$) in Equation \eqref{eq:updateHeter}, we have: 
            \begin{eqnarray*}
            && \left(p_n^{k,r}(B_m) - p_n^{k,r-1}(B_m)\right)\left(\lambda+\mu_m \right) \nonumber = \sum_{l \in \mathscr{C}_{n-1}} p_{n-1}^{k}(B_l) \frac{\mu^{k,r-1}(n)}{\lambda^{k}(n-1)}\lambda_{lm} - \sum_{l \in \mathscr{C}_{n-1}} p_{n-1}^{k}(B_l) \frac{\mu^{k,r-2}(n)}{\lambda^{k}(n-1)}\lambda_{lm}.  
        \end{eqnarray*}
        By the induction hypothesis, $\lambda^{k}(n-1) = \lambda$ because of the normalization property. We define
        \begin{equation*} 
            \Delta_{n,m}^{k,r} = p_n^{k,r}(B_m) - p_n^{k,r-1}(B_m),
        \end{equation*} 
        and rewrite the above equation as:
        \begin{eqnarray*}
            \Delta_{n,m}^{k,r}\left(\lambda+\mu_m \right) 
            &=& \left(\frac{1}{\lambda} \sum_{l \in \mathscr{C}_{n-1}} p_{n-1}^{k}(B_l) \lambda_{lm} \right) \left(\mu^{k,r-1}(n) - \mu^{k,r-2}(n)\right) \nonumber \\
            &=& \left(\frac{1}{\lambda} \sum_{l \in \mathscr{C}_{n-1}} p_{n-1}^{k}(B_l) \lambda_{lm} \right) \left(\sum_{m \in \mathscr{C}_n}\mu_{m} \Delta_{n,m}^{k,r-1} \right), \label{eq:eq-Delta_m}
        \end{eqnarray*}
        where the second equality arises from Equation \eqref{eq:muHeter}. By applying the \textit{Triangle Inequality}, and taking absolute values on both sides, we obtain
        \begin{eqnarray*}\label{eq:ineq-triangle}
            \left(\lambda+\mu_m \right) \left| \Delta_{n,m}^{k,r} \right| \leq \left(\frac{1}{\lambda} \sum_{l \in \mathscr{C}_{n-1}} p_{n-1}^{k}(B_l) \lambda_{lm} \right) \left(\sum_{m \in \mathscr{C}_n}\mu_{m}  \left|\Delta_{n,m}^{k,r-1} \right| \right).
        \end{eqnarray*}

         Summing both sides over all $m \in \mathscr{C}_n$, we obtain
        \begin{eqnarray*} 
            \sum_{m \in \mathscr{C}_n} \left(\lambda+\mu_m \right) \left| \Delta_{n,m}^{k,r} \right| 
            & \leq & \left(\sum_{m \in \mathscr{C}_n} \mu_m\left|\Delta_{n,m}^{k,r-1} \right| \right) \sum_{m \in \mathscr{C}_n} \left(\frac{1}{\lambda} \sum_{l \in \mathscr{C}_{n-1}} p_{n-1}^{k}(B_l) \lambda_{lm} \right) \nonumber \\
            &=& \left(\sum_{m \in \mathscr{C}_n} \mu_m\left|\Delta_{n,m}^{k,r-1} \right| \right) \left(\frac{1}{\lambda} \sum_{l \in \mathscr{C}_{n-1}} p_{n-1}^{k}(B_l) \underbrace{\sum_{m \in \mathscr{C}_n} \lambda_{lm}}_{=\lambda} \right) \nonumber\\
            &=& \left(\sum_{m \in \mathscr{C}_n} \mu_m \left|\Delta_{n,m}^{k,r-1} \right| \right) \left(\frac{1}{\lambda} \lambda^{k}(n-1) \right) \nonumber\\
            &=& \sum_{m \in \mathscr{C}_n} \left(\lambda+\mu_m \right)\left|\Delta_{n,m}^{k,r-1} \right| - \lambda \sum_{m \in \mathscr{C}_n} \left|\Delta_{n,m}^{k,r-1} \right| .
        \end{eqnarray*}
        Define $A_r = \sum_{m \in \mathscr{C}_n} \left| \Delta_{n,m}^{k,r} \right|$. Then, for all $r$, we have
        \begin{eqnarray} 
            A_r &<& \sum_{m \in \mathscr{C}_n} \left(\lambda+\mu_m \right)\left| \Delta_{n,m}^{k,r} \right| \nonumber \\
            &\leq& \sum_{m \in \mathscr{C}_n} \left(\lambda+\mu_m \right)\left| \Delta_{n,m}^{k,r-1} \right| - \lambda A_{r-1} \nonumber\\
            &\leq& \ldots  \nonumber \\
            &\leq& \sum_{m \in \mathscr{C}_n}\left(\lambda+\mu_m \right) \left| \Delta_{n,m}^{k,1} \right| - \lambda \sum_{s=1}^{r-1}A_s. \nonumber
        \end{eqnarray}
        Using the induction hypothesis $\sum_{m \in \mathscr{C}_n} p_n^{k - 1}(B_m) = 1$, we have
        \begin{eqnarray*}
            \sum_{m \in \mathscr{C}_n} \left(\lambda+\mu_m \right)\left| \Delta_{n,m}^{k,1} \right| \leq \sum_{m \in \mathscr{C}_n} \left(\lambda+\mu_m \right)\left( p_n^{k,1}(B_m) + p_n^{k - 1}(B_m) \right) = D^{k,1}_n.
        \end{eqnarray*}
        At iteration $k$, $p_n^{k-1}(B_m)$ is fixed, and given all probabilities and transition rates before layer $n$ at step $k$, $p_n^{k,1}(B_m)$ is known, so $D^{k,1}_n$ is fixed. Because $A_s \ge 0$ for any $s \ge 1$, the sum $\sum_{s=1}^{r-1}A_s$ is nondecreasing and bounded above by $D^{k,1}_n / \lambda$. Therefore,
        \begin{eqnarray*} \label{eq:absCompare}
            \lim_{r \rightarrow \infty} \left| \Delta_{n,m}^{k,r} \right| \leq \lim_{r \rightarrow \infty} A_r = 0,
        \end{eqnarray*} 
        and $\left|p_n^{k,r}(B_m) - p_n^{k,r-1}(B_m) \right|$ converges to 0 as $r$ approaches infinity.

        To prove the normalization property, we revisit Equation \eqref{eq:updateHeter} and take the sum of both sides:
        \begin{eqnarray*}
            && \sum_{m \in \mathscr{C}_n} \left(\lambda+\mu_{m}\right) p_n^{k,r}(B_m) \nonumber \\
            &=& \sum_{m \in \mathscr{C}_n} \sum_{l \in \mathscr{C}_{n-1}} p_{n-1}^{k}(B_l) \frac{\mu^{k,r-1}(n)}{\lambda^{k}(n-1)}\lambda_{l m} + \sum_{m \in \mathscr{C}_n} \sum_{l \in \mathscr{C}^m_{n+1}} p_{n+1}^{k-1}(B_l) \frac{\lambda^{k-1}(n)}{\mu^{k-1}(n+1)}\mu_{lm} \nonumber \\
            &=& \left(\frac{\mu^{k,r-1}(n)}{\lambda^{k}(n-1)} \underbrace{\sum_{l \in \mathscr{C}_{n-1}} p_{n-1}^{k}(B_l) \sum_{m \in \mathscr{C}_n} \lambda_{lm}}_{=\lambda^{k}(n-1)} \right) + \left(\frac{\lambda^{k-1}(n)}{\mu^{k-1}(n+1)} \underbrace{\sum_{l \in \mathscr{C}_{n+1}} p_{n+1}^{k-1}(B_l) \sum_{m \in \mathscr{C}_n} \mu_{lm}}_{=\mu^{k-1}(n+1)} \right) \nonumber \\
            &=& \mu^{k,r-1}(n) + \lambda \nonumber \\
            &=& \sum_{m \in \mathscr{C}_n} \mu_{m} p_{n}^{k,r-1}\left(B_m\right) + \lambda.
        \end{eqnarray*} 
        Taking the limit of both sides of the equation, we have:
        \begin{eqnarray*}
             \lim_{r \rightarrow \infty} \sum_{m \in \mathscr{C}_n} \mu_{m}  p_n^{k,r}(B_m) + \lim_{r \rightarrow \infty} \sum_{m \in \mathscr{C}_n} \lambda p_n^{k,r}(B_m) = \lim_{r \rightarrow \infty} \sum_{m \in \mathscr{C}_n} \mu_m p_n^{k,r-1}(B_m) + \lambda. 
        \end{eqnarray*}
        Given our earlier proof that $p_n^{k,r}(B_m)$ converges, the first terms on both sides of the above equation are equal. Therefore, 
        \begin{eqnarray*}
            \lambda \lim_{r \rightarrow \infty} \sum_{m \in \mathscr{C}_n} p_n^{k,r}(B_m) = \lambda. 
       \end{eqnarray*}
       This implies
       \begin{eqnarray*}
            \lim_{r \rightarrow \infty} \sum_{m \in \mathscr{C}_n} p_n^{k,r}(B_m) = 1.
       \end{eqnarray*}
    \end{enumerate}

    We have thus demonstrated that these properties are valid for layer \( n \) at step \( k \), which concludes our proof.\Halmos
\endproof

\subsection{Proof of Theorem \ref{thm_convergence_heter}}
\label{ssec:proof_conv_heter}
\proof{Proof.}
    By Lemma \ref{lemma:conv&sum}, we can write Equation \eqref{eq:updateHeter} as 
    \begin{equation} 
    p_n^{k}(B_m)=\sum_{l \in \mathscr{C}_{n-1}} p_{n-1}^{k}(B_l) \frac{\mu^{k}(n)}{\lambda^{k}(n-1)}\frac{\lambda_{l m}}{\lambda_{m}+\mu_{m}} + \sum_{l \in \mathscr{C}^m_{n+1}} p_{n+1}^{k-1}(B_l) \frac{\lambda^{k-1}(n)}{\mu^{k-1}(n+1)}\frac{\mu_{lm}}{\lambda_{m}+\mu_{m}}.  \label{eq:updateHeter-k} 
\end{equation}
    By subtracting the value at the $k$-th iteration from the value at the $(k+1)$-th iteration, we obtain 
    \begin{eqnarray*}
        p_n^{k+1}(B_m)-p_n^k(B_m) &=& \sum_{l \in \mathscr{C}_{n-1}} \left( \frac{\mu^{k+1}(n)}{\lambda^{k+1}(n-1)} p_{n-1}^{k+1}(B_l) - \frac{\mu^{k}(n)}{\lambda^{k}(n-1)} p_{n-1}^{k}(B_l)\right)  \frac{\lambda_{lm}}{\lambda + \mu_m }  \nonumber\\
        &&+ \sum_{l \in \mathscr{C}^m_{n+1}} \left(\frac{\lambda^{k}(n)}{\mu^{k}(n+1)} p_{n+1}^{k}(B_l) - \frac{\lambda^{k-1}(n)}{\mu^{k-1}(n+1)}p_{n+1}^{k-1}(B_l)\right) \frac{\mu_{lm} }{\lambda + \mu_m }. 
    \end{eqnarray*}
    By taking the absolute values on both sides and applying the \textit{Triangle Inequality}, we obtain
    \begin{eqnarray*}
        \left| p_n^{k+1}(B_m)-p_n^k(B_m)\right| & \leq & \left| \sum_{l \in \mathscr{C}_{n-1}} \left( \frac{\mu^{k+1}(n)}{\lambda^{k+1}(n-1)} p_{n-1}^{k+1}(B_l) - \frac{\mu^{k}(n)}{\lambda^{k}(n-1)} p_{n-1}^{k}(B_l)\right)  \frac{\lambda_{lm}}{\lambda + \mu_m } \right|  \nonumber\\
        &&+ \left| \sum_{l \in \mathscr{C}^m_{n+1}} \left(\frac{\lambda^{k}(n)}{\mu^{k}(n+1)} p_{n+1}^{k}(B_l) - \frac{\lambda^{k-1}(n)}{\mu^{k-1}(n+1)}p_{n+1}^{k-1}(B_l)\right) \frac{\mu_{lm} }{\lambda + \mu_m }\right| \\
        & \leq &  \frac{1}{\lambda + \underline{\mu}_n } \left| \sum_{l \in \mathscr{C}_{n-1}} \left( \frac{\mu^{k+1}(n)}{\lambda^{k+1}(n-1)} p_{n-1}^{k+1}(B_l) - \frac{\mu^{k}(n)}{\lambda^{k}(n-1)} p_{n-1}^{k}(B_l)\right)\lambda_{lm} \right|  \nonumber\\
        &&+ \frac{1}{\lambda + \underline{\mu}_n } \left| \sum_{l \in \mathscr{C}^m_{n+1}} \left(\frac{\lambda^{k}(n)}{\mu^{k}(n+1)} p_{n+1}^{k}(B_l) - \frac{\lambda^{k-1}(n)}{\mu^{k-1}(n+1)}p_{n+1}^{k-1}(B_l)\right) \mu_{lm} \right|.
    \end{eqnarray*}
    Taking the sum on both sides, we obtain
    \begin{eqnarray*}
        \sum_{m \in \mathscr{C}_n} \left|p_n^{k+1}(B_m)-p_n^k(B_m)\right| \nonumber & \leq &  \frac{1}{\lambda + \underline{\mu}_n } \sum_{m \in \mathscr{C}_n} \left| \sum_{l \in \mathscr{C}_{n-1}} \left(\frac{\mu^{k+1}(n)}{\lambda^{k+1}(n-1)} p_{n-1}^{k+1}(B_l) - \frac{\mu^{k}(n)}{\lambda^{k}(n-1)} p_{n-1}^{k}(B_l)\right) \lambda_{lm} \right| \nonumber \\
        && + \frac{1}{\lambda + \underline{\mu}_n } \sum_{m \in \mathscr{C}_n} \left| \sum_{l \in \mathscr{C}^m_{n+1}} \left(\frac{\lambda^{k}(n)}{\mu^{k}(n+1)} p_{n+1}^{k}(B_l) - \frac{\lambda^{k-1}(n)}{\mu^{k-1}(n+1)}p_{n+1}^{k-1}(B_l)\right) \mu_{lm}\right| .
    \end{eqnarray*}
    According to Lemma \ref{lemma:conv&sum}, we have:
    \begin{eqnarray*}
        && \lambda^{k+1}(n-1) = \lambda^{k}(n-1) = \lambda^{k-1}(n) = \lambda^{k}(n) = \lambda; \\
        && \mu^{k+1}(n) \leq \overline{\mu}_n, \ \mu^{k}(n) \leq \overline{\mu}_n, \ \mu^{k}(n+1) \ge \underline{\mu}_{n+1}, \ \mu^{k-1}(n+1) \ge\underline{\mu}_{n+1}.
    \end{eqnarray*}
    We also define: 
    \begin{eqnarray*}
        \Delta^{k}_{n,m} = p_n^{k}(B_m)-p_n^{k-1}(B_m).
    \end{eqnarray*}

    Therefore, we have
    \begin{eqnarray*}
        \sum_{m \in \mathscr{C}_n} \left|\Delta^{k+1}_{n,m}\right|
        & \leq &  \frac{1}{\lambda + \underline{\mu}_n } \left(\sum_{m \in \mathscr{C}_n} \sum_{l \in \mathscr{C}_{n-1}} \frac{ \overline{\mu}_n}{\lambda} \lambda_{lm}  \left|\Delta^{k+1}_{n-1,l} \right| \nonumber +  \sum_{m \in \mathscr{C}_n} \sum_{l \in \mathscr{C}^m_{n+1}} \frac{\lambda}{\underline{\mu}_{n+1}} \mu_{lm}\left| \Delta^{k}_{n+1,l} \right| \right)\\
        &=& \frac{1}{\lambda + \underline{\mu}_n } \left(\sum_{l \in \mathscr{C}_{n-1}} \frac{ \overline{\mu}_n}{\lambda} \left|\Delta^{k+1}_{n-1,l} \right| \underbrace{\sum_{m \in \mathscr{C}_n} \lambda_{lm}}_{=\lambda} +  \sum_{l \in \mathscr{C}_{n+1}} \frac{\lambda}{\underline{\mu}_{n+1}} \left| \Delta^{k}_{n+1,l} \right| \underbrace{\sum_{m \in \mathscr{C}_n}\mu_{lm}}_{=\mu_l}\right)\\
        &=& \frac{1}{\lambda + \underline{\mu}_n } \left(\sum_{l \in \mathscr{C}_{n-1}} \overline{\mu}_n \left|\Delta^{k+1}_{n-1,l} \right| +  \sum_{l \in \mathscr{C}_{n+1}} \frac{\lambda \mu_l}{\underline{\mu}_{n+1}} \left| \Delta^{k}_{n+1,l} \right| \right) \\
        & \leq & \frac{1}{\lambda + \underline{\mu}_n } \left( \overline{\mu}_n \sum_{l \in \mathscr{C}_{n-1}} \left|\Delta^{k+1}_{n-1,l} \right| +  \frac{\lambda \overline{\mu}_{n+1}}{\underline{\mu}_{n+1}} \sum_{l \in \mathscr{C}_{n+1}} \left| \Delta^{k}_{n+1,l} \right| \right).
    \end{eqnarray*}
    Let $M_{k,n}=\sum_{l \in \mathscr{C}_n}\left|\Delta^{k}_{n,l}\right|$ and $\gamma_n  = \frac{ \overline{\mu}_{n}}{\underline{\mu}_{n}}$. Then, we have 
    \begin{eqnarray*}
        M_{k+1,n} \leq \frac{1}{\lambda + \underline{\mu}_n } \left( \overline{\mu}_n M_{k+1,n-1} +  \gamma_{n+1} \lambda M_{k,n+1} \right).
    \end{eqnarray*}
    For every step $k$, we have:
    \begin{eqnarray*}
        && M_{k,0} = \sum_{l \in \mathscr{C}_0}\left|p_0^{k}(B_l)-p_0^{k-1}(B_l)\right| = \left|p_0^{k}(B_0)-p_0^{k-1}(B_0)\right| = |1-1| = 0, \\
        && M_{k,N} = \sum_{l \in \mathscr{C}_N}\left|p_N^{k}(B_l)-p_N^{k-1}(B_l)\right| = \left|p_N^{k}(B_{2^N-1})-p_N^{k-1}(B_{2^N-1})\right| = |1-1| = 0 .
    \end{eqnarray*}
    Let $M_k = \sup_n M_{k,n}$,
    \begin{eqnarray*}
        M_{k+1,1} &\leq & \frac{1}{\lambda + \underline{\mu}_1 } \left(\overline{\mu}_1 M_{k+1,0} + \gamma_2\lambda M_{k,2} \right) = \frac{1}{\lambda + \underline{\mu}_1 } \gamma_2\lambda M_{k,2} \leq \frac{\lambda}{\lambda + \underline{\mu}_1 } \gamma_2  M_{k} \\
        M_{k+1,2} &\leq & \frac{1}{\lambda + \underline{\mu}_2 } \left(\overline{\mu}_2 M_{k+1,1} +  \gamma_3\lambda M_{k,3} \right) \leq \frac{\lambda}{\lambda + \underline{\mu}_2 } \left(\frac{\overline{\mu}_2}{\lambda + \underline{\mu}_1 } \gamma_2 +  \gamma_3 \right)M_{k} \\
        M_{k+1,3} &\leq & \frac{1}{\lambda + \underline{\mu}_3 } \left(\overline{\mu}_3 M_{k+1,2} +  \gamma_4 \lambda M_{k,4} \right) 
        \leq  \frac{\lambda}{\lambda + \underline{\mu}_3 } \left( \frac{\overline{\mu}_3}{\lambda + \underline{\mu}_2} \left(\frac{\overline{\mu}_2}{\lambda + \underline{\mu}_1} \gamma_2 +  \gamma_3\right)  +  \gamma_4 \right)M_{k} \\
        & = &  \frac{\lambda}{\lambda + \underline{\mu}_3 } \left(\frac{\overline{\mu}_3}{\lambda + \underline{\mu}_2} \frac{\overline{\mu}_2}{\lambda + \underline{\mu}_1} \gamma_2 + \frac{ \overline{\mu}_3}{\lambda + \underline{\mu}_2} \gamma_3 + \gamma_4\right)M_{k}\\
        M_{k+1,4} &\leq & \frac{1}{\lambda + \underline{\mu}_4} \left(\overline{\mu}_4 M_{k+1,3} +  \gamma_5 \lambda M_{k,5} \right) \\
        &\leq & \frac{\lambda}{\lambda + \underline{\mu}_4} \left( \frac{\overline{\mu}_4}{\lambda + \underline{\mu}_3 } \left(\frac{\overline{\mu}_3}{\lambda + \underline{\mu}_2} \frac{ \overline{\mu}_2}{\lambda + \underline{\mu}_1} \gamma_2 + \frac{\overline{\mu}_3}{\lambda + \underline{\mu}_2} \gamma_3 + \gamma_4 \right) +  \gamma_5 \right)M_{k} \\
        & = &  \frac{\lambda}{\lambda + \underline{\mu}_4} \left(\frac{\overline{\mu}_4}{\lambda + \underline{\mu}_3 } \frac{\overline{\mu}_3}{\lambda + \underline{\mu}_2} \frac{ \overline{\mu}_2}{\lambda + \underline{\mu}_1} \gamma_2 + \frac{\overline{\mu}_4}{\lambda + \underline{\mu}_3 }\frac{\overline{\mu}_3}{\lambda + \underline{\mu}_2} \gamma_3 + \frac{ \overline{\mu}_4}{\lambda + \underline{\mu}_3 }\gamma_4 + \gamma_5 \right)M_{k}.
    \end{eqnarray*}
    For notational convenience, we define:
    \begin{eqnarray*}
        && \phi_{n,q} =  \gamma_q \prod_{j=q}^{n} \frac{\overline{\mu}_j}{\lambda + \underline{\mu}_{j-1}} \ (2 \leq q \leq n)  \quad \mbox{and} \quad \phi_{n,n+1} = \gamma_{n+1}, \\
        && \Phi_n =  \frac{\lambda}{\lambda + \underline{\mu}_n} \sum_{q=2}^{n+1} \phi_{n,q}, \ (n \ge 2).
    \end{eqnarray*}
    We proceed by induction to show that for any $n \ge 2$, 
    \begin{eqnarray} \label{eq:MkInduct}
        M_{k+1,n} \leq \Phi_n M_k.
    \end{eqnarray}
    We have previously shown the case for $n=2$. Assuming the formula holds for $n=s$, for the case $n=s+1$, we have
    \begin{eqnarray*}
        M_{k+1,s+1} & \leq & \frac{1}{\lambda + \underline{\mu}_{s+1}} \left( \overline{\mu}_{s+1} \underbrace{M_{k+1,s}}_{\leq \Phi_{s} M_k} + \gamma_{s+2}\lambda \underbrace{M_{k,s+2}}_{\leq M_k} \right) \nonumber\\
        & \leq & \frac{\lambda}{\lambda + \underline{\mu}_{s+1}} \left( \overline{\mu}_{s+1} \left(\frac{1}{\lambda + \underline{\mu}_s} \sum_{q=2}^{s} \phi_{s,q} + \frac{1}{\lambda + \underline{\mu}_{s}} \gamma_{s+1} \right) + \gamma_{s+2}\right) M_k \label{eq:PhiInduct-1} \nonumber\\ 
        & = & \frac{\lambda}{\lambda + \underline{\mu}_{s+1}} \left( \frac{\overline{\mu}_{s+1}}{\lambda + \underline{\mu}_s} \sum_{q=2}^{s} \gamma_q \prod_{j=q}^{s} \frac{ \overline{\mu}_j}{\lambda + \underline{\mu}_{j-1}} + \frac{ \overline{\mu}_{s+1}}{\lambda + \underline{\mu}_{s}} \gamma_{s+1}+ \gamma_{s+2}\right) M_k \nonumber\\
        & = & \frac{\lambda}{\lambda + \underline{\mu}_{s+1}} \left( \underbrace{\sum_{q=2}^{s} \gamma_q \prod_{j=q}^{s+1} \frac{ \overline{\mu}_j}{\lambda + \underline{\mu}_{j-1}} + \frac{ \overline{\mu}_{s+1}}{\lambda + \underline{\mu}_{s}} \gamma_{s+1}}_{ = \sum_{q=2}^{s+1}\phi_{s+1,q}} + \gamma_{s+2}\right) M_k \nonumber\\
        & = & \Phi_{s+1} M_k .\label{eq:PhiInduct-2}
    \end{eqnarray*}
    Thus, we confirm that \eqref{eq:MkInduct} holds and is independent of $k$.
    Notice that $\Phi_N = \max_n \Phi_n < 1$, where this inequality follows from Assumption \ref{assu-phi}.
    Therefore, $M_{k+1}\leq \Phi_N M_k$ for each $k$ because
    \begin{equation}
        M_{k+1} = \sup_{n} M_{k+1,n} \leq \Phi_N M_k. \label{eq:heter M_k+1}
    \end{equation}
    Moreover, based on \eqref{eq:heter M_k+1}, 
    \begin{equation}
        M_{k+1} \leq \Phi_N M_k \leq \Phi^2_N M_{k-1}\leq \Phi_N ^k M_1.
    \end{equation}
    Given that $\Phi_N<1$ and is independent of $k$,
    \begin{equation*}
       \lim_{k\rightarrow \infty} M_{k+1} \leq \lim_{k\rightarrow \infty} \Phi_N ^k M_1 = 0.
    \end{equation*}
    Hence, for all $n$ and $k$, 
    \begin{eqnarray*}
       \left|p_n^{k+1}(B_l)-p_n^{k}(B_l)\right| \leq \sum_{l \in \mathscr{C}_n}\left|p_n^{k+1}(B_l)-p_n^{k}(B_l)\right| = M_{k+1,n}\leq \sup_n M_{k+1,n} = M_{k+1}.
    \end{eqnarray*}
    In conclusion, 
    \begin{equation*}
        \lim_{k\rightarrow \infty}\left|p_n^{k+1}(B_l)-p_n^{k}(B_l)\right| < \lim_{k\rightarrow \infty} M_{k+1} < \lim_{k\rightarrow \infty} \Phi_N ^k M_1 = 0.
    \end{equation*}
    Thus, for each $n$, $\left|p_n^{k+1}(B_l)-p_n^{k}(B_l)\right|$ converges to $0$ as $k$ approaches infinity at a geometric rate. 

    Next, we show that it converges to the original hypercube solution. Suppose the algorithm converges to $\bar{p}_n(B_m)$, with $\lim_{k\rightarrow \infty} p_n^k(B_m) = \bar{p}_n(B_m) \not= p_n(B_m)$. If the algorithm converges, there exists an $M>0$ and an infinitesimally small $\epsilon>0$ such that for all $m>M$,
\begin{eqnarray*}
   &&|p_n^m(B_m) - \bar{p}_n(B_m)| \leq \epsilon, \quad \mbox{and} \quad |p_n^{m-1}(B_m) - \bar{p}_n(B_m)| \leq \epsilon,
\end{eqnarray*}
for all $n=0,1,\ldots, N$ and all $m$. By proving this and noting that $p_n^m(B_m)$ and $p_n^{m-1}(B_m)$ follow Equation \eqref{eq:updateHeter-k}, we can infer that $\bar{p}_n(B_m)$ also satisfies Equation \eqref{eq:updateHeter-k} when $\epsilon \rightarrow 0$. We then have both $\bar{p}_n(B_m)$ and $p_n(B_m)$ as solutions. 

Given that all states in our model are positive recurrent and the continuous Markov process for the spatial queueing system is irreducible, there is a guarantee of a unique positive stationary distribution. Thus, the steady state distribution $P\{n, B_m\}$ corresponding to \eqref{eq:transformer_bal_eq_heter} and \eqref{eq:sumto1} is unique. We have $p_n(B_m) = P\{n, B_m\}/ p(n)$, which is a one-to-one mapping from state joint probability $P\{n, B_m\}$ to conditional probability $p_n(B_m)$. Consequently, we have $\bar{p}_n(B_m) = p_n(B_m)$. 
\Halmos
\endproof

\section{Extended Numerical Analysis}
\label{sec_extended_numerical}

This section provides detailed times, including both coefficient generation and computation, for St. Paul, MN, and Greenville County, SC. 

\subsection{Heterogeneous Case}
\label{ssec_extended_hetero}
In this section, we present the time for heterogeneous service rates for different cases, detailed in Table~\ref{tab_EC_heter_detail}. In these instances, our CPU algorithm (i.e., Algorithm~\ref{algo-BnDHeter}) is compared with a method that solves the problem using a sparse solver, given that both the original hypercube model and the modified model employing alternating hyperplane methods are infeasible for heterogeneous service rates.

For each scenario, we detail the computation times for both the sparse solver and our algorithm, with the latter emphasized in boldface. Note that NA indicates that results are omitted because the original hypercube solution becomes computationally prohibitive for systems with more than 20 units. From the tables, we note that the computation time for the sparse solver increases significantly as $N$ grows. For instance, the sparse solver requires more than 9,000 seconds to tackle a 15-unit problem and becomes impractical for scenarios with over 15 units. Conversely, our method is able to solve considerably larger problems efficiently. Specifically, for the 15-unit case, our method is 2,000 times faster compared to the sparse solver. 

  \begin{table}[htbp]
  \tiny
  \centering
  \renewcommand\arraystretch{1.3}
  \tabcolsep=3pt
  \caption{Detailed Time (s) for the Heterogeneous Case, St. Paul.}
  \begin{adjustbox}{center}
  \begin{threeparttable}
    \begin{tabular}{@{\extracolsep{2pt}}cccccccccccccccccccccc@{}}
      \hline
      & \multirow{2}{*}{$\rho$} & \multicolumn{20}{c}{\# Units ($N$)} \\
      \cline{3-22}
      & & \multicolumn{2}{c}{11} & \multicolumn{2}{c}{12} & \multicolumn{2}{c}{13} & \multicolumn{2}{c}{14} & \multicolumn{2}{c}{15} & \multicolumn{2}{c}{16} & \multicolumn{2}{c}{17} & \multicolumn{2}{c}{18} & \multicolumn{2}{c}{19} & \multicolumn{2}{c}{20} \\
      \hline
      \multirow{9}{*}{St. Paul} & 0.1 & 2.5 & $\boldsymbol{1.3}$ & 23.3 & $\boldsymbol{2.0}$ & 198.5 & $\boldsymbol{3.2}$ & 1120.1 & $\boldsymbol{6.1}$ & 9042.9 & $\boldsymbol{9.7}$ & NA & $\boldsymbol{35.2}$ & NA & $\boldsymbol{108.8}$ & NA & $\boldsymbol{265.1}$ & NA & $\boldsymbol{606.7}$ & NA & $\boldsymbol{1764.4}$ \\
      & 0.2 & 2.5 & $\boldsymbol{0.7}$ & 22.9 & $\boldsymbol{0.9}$ & 194.8 & $\boldsymbol{1.6}$ & 1086.0 & $\boldsymbol{3.2}$ & 9226.3 & $\boldsymbol{3.5}$ & NA & $\boldsymbol{10.6}$ & NA & $\boldsymbol{53.0}$ & NA & $\boldsymbol{123.9}$ & NA & $\boldsymbol{291.9}$ & NA & $\boldsymbol{645.2}$ \\
      & 0.3 & 2.4 & $\boldsymbol{0.7}$ & 22.9 & $\boldsymbol{1.0}$ & 194.4 & $\boldsymbol{1.9}$ & 1097.1 & $\boldsymbol{2.6}$ & 9012.0 & $\boldsymbol{4.1}$ & NA & $\boldsymbol{13.0}$ & NA & $\boldsymbol{63.8}$ & NA & $\boldsymbol{148.9}$ & NA & $\boldsymbol{342.1}$ & NA & $\boldsymbol{715.9}$ \\
      & 0.4 & 2.4 & $\boldsymbol{0.7}$ & 23.0 & $\boldsymbol{1.0}$ & 193.4 & $\boldsymbol{2.0}$ & 1088.6 & $\boldsymbol{2.5}$ & 9227.1 & $\boldsymbol{4.3}$ & NA & $\boldsymbol{14.0}$ & NA & $\boldsymbol{68.2}$ & NA & $\boldsymbol{155.5}$ & NA & $\boldsymbol{356.5}$ & NA & $\boldsymbol{718.0}$ \\
      & 0.5 & 2.5 & $\boldsymbol{0.8}$ & 23.0 & $\boldsymbol{1.0}$ & 193.9 & $\boldsymbol{1.9}$ & 1075.5 & $\boldsymbol{2.5}$ & 9310.8 & $\boldsymbol{4.2}$ & NA & $\boldsymbol{13.8}$ & NA & $\boldsymbol{69.7}$ & NA & $\boldsymbol{156.6}$ & NA & $\boldsymbol{360.4}$ & NA & $\boldsymbol{702.6}$ \\
      & 0.6 & 2.5 & $\boldsymbol{0.8}$ & 22.9 & $\boldsymbol{1.0}$ & 195.0 & $\boldsymbol{1.8}$ & 1086.7 & $\boldsymbol{2.4}$ & 9276.2 & $\boldsymbol{4.1}$ & NA & $\boldsymbol{13.4}$ & NA & $\boldsymbol{65.8}$ & NA & $\boldsymbol{150.1}$ & NA & $\boldsymbol{344.5}$ & NA & $\boldsymbol{674.6}$ \\
      & 0.7 & 2.5 & $\boldsymbol{0.8}$ & 22.8 & $\boldsymbol{0.9}$ & 194.1 & $\boldsymbol{1.7}$ & 1113.1 & $\boldsymbol{2.2}$ & 9185.2 & $\boldsymbol{3.9}$ & NA & $\boldsymbol{12.8}$ & NA & $\boldsymbol{63.7}$ & NA & $\boldsymbol{139.6}$ & NA & $\boldsymbol{327.2}$ & NA & $\boldsymbol{645.6}$ \\
      & 0.8 & 2.5 & $\boldsymbol{0.8}$ & 22.8 & $\boldsymbol{0.9}$ & 194.1 & $\boldsymbol{1.6}$ & 1087.1 & $\boldsymbol{2.1}$ & 9226.3 & $\boldsymbol{3.6}$ & NA & $\boldsymbol{11.9}$ & NA & $\boldsymbol{59.4}$ & NA & $\boldsymbol{133.1}$ & NA & $\boldsymbol{310.2}$ & NA & $\boldsymbol{601.8}$ \\
      & 0.9 & 2.5 & $\boldsymbol{0.8}$ & 22.8 & $\boldsymbol{0.9}$ & 193.6 & $\boldsymbol{1.4}$ & 1097.0 & $\boldsymbol{2.0}$ & 9246.6 & $\boldsymbol{3.4}$ & NA & $\boldsymbol{11.0}$ & NA & $\boldsymbol{53.9}$ & NA & $\boldsymbol{123.5}$ & NA & $\boldsymbol{283.4}$ & NA & $\boldsymbol{575.3}$ \\
      \hline
      \multirow{9}{*}{Greenville} & 0.1 & 2.5 & $\boldsymbol{1.5}$ & 24.5 & $\boldsymbol{1.7}$ & 177.3 & $\boldsymbol{3.2}$ & 1006.2 & $\boldsymbol{4.4}$ & 8713.4 & $\boldsymbol{8.7}$ & NA & $\boldsymbol{25.4}$ & NA & $\boldsymbol{76.4}$ & NA & $\boldsymbol{188.7}$ & NA & $\boldsymbol{575.8}$ & NA & $\boldsymbol{1152.4}$ \\
      & 0.2 & 2.5 & $\boldsymbol{0.7}$ & 23.5 & $\boldsymbol{0.8}$ & 172.4 & $\boldsymbol{1.1}$ & 1035.4 & $\boldsymbol{1.7}$ & 9165.5 & $\boldsymbol{5.3}$ & NA & $\boldsymbol{15.4}$ & NA & $\boldsymbol{40.4}$ & NA & $\boldsymbol{100.9}$ & NA & $\boldsymbol{253.0}$ & NA & $\boldsymbol{475.4}$ \\
      & 0.3 & 2.4 & $\boldsymbol{0.8}$ & 23.6 & $\boldsymbol{1.0}$ & 172.3 & $\boldsymbol{1.4}$ & 1024.3 & $\boldsymbol{2.0}$ & 8581.8 & $\boldsymbol{6.4}$ & NA & $\boldsymbol{18.8}$ & NA & $\boldsymbol{48.1}$ & NA & $\boldsymbol{121.5}$ & NA & $\boldsymbol{301.8}$ & NA & $\boldsymbol{622.0}$ \\
      & 0.4 & 2.4 & $\boldsymbol{0.9}$ & 23.9 & $\boldsymbol{1.1}$ & 174.3 & $\boldsymbol{1.5}$ & 1192.7 & $\boldsymbol{2.1}$ & 8883.9 & $\boldsymbol{6.9}$ & NA & $\boldsymbol{20.4}$ & NA & $\boldsymbol{52.6}$ & NA & $\boldsymbol{127.7}$ & NA & $\boldsymbol{338.0}$ & NA & $\boldsymbol{674.6}$ \\
      & 0.5 & 2.4 & $\boldsymbol{0.9}$ & 24.2 & $\boldsymbol{1.2}$ & 175.6 & $\boldsymbol{1.6}$ & 1019.8 & $\boldsymbol{2.2}$ & 8813.1 & $\boldsymbol{7.1}$ & NA & $\boldsymbol{21.2}$ & NA & $\boldsymbol{54.2}$ & NA & $\boldsymbol{130.8}$ & NA & $\boldsymbol{345.4}$ & NA & $\boldsymbol{690.6}$ \\
      & 0.6 & 2.5 & $\boldsymbol{0.9}$ & 23.5 & $\boldsymbol{1.1}$ & 175.9 & $\boldsymbol{1.6}$ & 1022.8 & $\boldsymbol{2.2}$ & 9184.3 & $\boldsymbol{6.9}$ & NA & $\boldsymbol{21.1}$ & NA & $\boldsymbol{54.3}$ & NA & $\boldsymbol{127.7}$ & NA & $\boldsymbol{327.9}$ & NA & $\boldsymbol{676.0}$ \\
      & 0.7 & 2.4 & $\boldsymbol{0.9}$ & 23.5 & $\boldsymbol{1.1}$ & 175.7 & $\boldsymbol{1.5}$ & 1021.3 & $\boldsymbol{2.1}$ & 8521.4 & $\boldsymbol{6.9}$ & NA & $\boldsymbol{20.8}$ & NA & $\boldsymbol{52.7}$ & NA & $\boldsymbol{124.7}$ & NA & $\boldsymbol{308.4}$ & NA & $\boldsymbol{632.3}$ \\
      & 0.8 & 2.4 & $\boldsymbol{0.8}$ & 24.7 & $\boldsymbol{1.1}$ & 172.3 & $\boldsymbol{1.4}$ & 1132.7 & $\boldsymbol{2.0}$ & 8577.4 & $\boldsymbol{6.7}$ & NA & $\boldsymbol{20.0}$ & NA & $\boldsymbol{51.1}$ & NA & $\boldsymbol{121.4}$ & NA & $\boldsymbol{293.0}$ & NA & $\boldsymbol{603.9}$ \\
      & 0.9 & 2.4 & $\boldsymbol{0.8}$ & 23.6 & $\boldsymbol{1.0}$ & 181.7 & $\boldsymbol{1.4}$ & 1341.3 & $\boldsymbol{1.9}$ & 8619.7 & $\boldsymbol{6.2}$ & NA & $\boldsymbol{19.5}$ & NA & $\boldsymbol{49.6}$ & NA & $\boldsymbol{116.2}$ & NA & $\boldsymbol{278.0}$ & NA & $\boldsymbol{560.1}$ \\
      \hline
  \end{tabular}
  \begin{tablenotes}
        \footnotesize
        \vspace{-0.3cm}
        \item \textit{Note.} Regular: computation time of the sparse solver. Bold: computation time of CPU algorithm.  NA: results not included. 
    \end{tablenotes}
    \end{threeparttable}
  \end{adjustbox}
  \label{tab_EC_heter_detail}
  \vspace{-0.3cm}
\end{table}

\subsection{Homogeneous Case}
\label{ssec_extended_homo}

This section evaluates the total time for cases with homogeneous service rates, using the CPU algorithm and the alternating hyperplane method by \cite{larson74hypercube}. 
We find that our coefficient generation operates almost twice as fast as the original TOUR method, and the computation time using our algorithm is more than 97\% faster than the computation time using the alternating hyperplane method. To make a fairer comparison, we improved the alternating hyperplane method by using current computer methodologies. We implemented sparse matrix manipulation while adhering to Larson's method in the other portions of his procedure. We then compared our algorithm with the enhanced version of Larson's method.
Our method reduces the total time by about 50\% compared to the enhanced Larson method. 

In the  Tables \ref{tab_homo_detail} and \ref{tab_greenville_homo_detail}, CPU represents our method, and AH denotes the original alternating hyperplane solution. The original version used point updating with for-loops. We also modified the original code to enhance its speed, labeling it as MAH, which stands for the modified solution of the original hypercube method. In this modified solution, we first convert the storage vector into a sparse matrix and then replace the for-loops with sparse matrix multiplications. For each case, we report the times for generating the transition rates (highlighted in boldface) and for computation. The original alternating hyperplane solution requires a significant amount of time for systems larger than 20 units, \ins{taking over 20 minutes for the 21-unit case and more than four hours for the 25-unit case}. Therefore, we did not include these results in the table.
    
\begin{table}[htbp]
\centering
\rotatebox{270}{
\begin{minipage}{\textheight}
\renewcommand\arraystretch{1.8}
\tabcolsep=1pt
\caption{Detailed Time (s) for the Homogeneous Case, St. Paul.}
\tiny
\begin{adjustbox}{center}
\begin{threeparttable}
\begin{tabular}{@{\extracolsep{2pt}}cccccccccccccccccccccccccccccccc@{}}
    \hline
    \multirow{2}{*}{} & \multirow{2}{*}{$\rho$} & \multicolumn{30}{c}{\# Units ($N$)} \\
    \cline{3-32}
    & & \multicolumn{2}{c}{11} & \multicolumn{2}{c}{12} & \multicolumn{2}{c}{13} & \multicolumn{2}{c}{14} & \multicolumn{2}{c}{15} & \multicolumn{2}{c}{16} & \multicolumn{2}{c}{17} & \multicolumn{2}{c}{18} & \multicolumn{2}{c}{19} & \multicolumn{2}{c}{20} & \multicolumn{2}{c}{21} & \multicolumn{2}{c}{22} & \multicolumn{2}{c}{23} & \multicolumn{2}{c}{24} & \multicolumn{2}{c}{25} \\
    \hline
    \multirow{9}{*}{CPU} & 0.1 & $\boldsymbol{0.2}$ & 0.0 & $\boldsymbol{0.4}$ & 0.1 & $\boldsymbol{0.9}$ & 0.1 & $\boldsymbol{1.8}$ & 0.2 & $\boldsymbol{3.6}$ & 0.3 & $\boldsymbol{7.4}$ & 0.5 & $\boldsymbol{15.1}$ & 0.9 & $\boldsymbol{30.4}$ & 2.2 & $\boldsymbol{62.1}$ & 4.6 & $\boldsymbol{126.6}$ & 11.8 & $\boldsymbol{257.0}$ & 29.5 & $\boldsymbol{522.2}$ & 63.7 & $\boldsymbol{1064.0}$ & 143.4 & $\boldsymbol{2111.5}$ & 324.3 & $\boldsymbol{4374.1}$ & 692.6 \\
    & 0.2 & $\boldsymbol{0.2}$ & 0.1 & $\boldsymbol{0.4}$ & 0.1 & $\boldsymbol{0.9}$ & 0.2 & $\boldsymbol{1.8}$ & 0.2 & $\boldsymbol{3.7}$ & 0.4 & $\boldsymbol{7.4}$ & 0.7 & $\boldsymbol{15.1}$ & 1.4 & $\boldsymbol{31.0}$ & 3.1 & $\boldsymbol{62.5}$ & 6.7 & $\boldsymbol{126.3}$ & 16.2 & $\boldsymbol{254.3}$ & 40.8 & $\boldsymbol{519.0}$ & 87.4 & $\boldsymbol{1060.8}$ & 198.9 & $\boldsymbol{2155.2}$ & 456.2 & $\boldsymbol{4326.9}$ & 997.2 \\
    & 0.3 & $\boldsymbol{0.2}$ & 0.1 & $\boldsymbol{0.4}$ & 0.1 & $\boldsymbol{0.9}$ & 0.2 & $\boldsymbol{1.8}$ & 0.3 & $\boldsymbol{3.6}$ & 0.4 & $\boldsymbol{7.4}$ & 0.8 & $\boldsymbol{15.3}$ & 1.7 & $\boldsymbol{30.4}$ & 3.6 & $\boldsymbol{62.1}$ & 7.9 & $\boldsymbol{126.3}$ & 18.1 & $\boldsymbol{255.6}$ & 45.1 & $\boldsymbol{520.4}$ & 98.2 & $\boldsymbol{1052.7}$ & 228.9 & $\boldsymbol{2117.4}$ & 525.4 & $\boldsymbol{4383.1}$ & 1166.9 \\
    & 0.4 & $\boldsymbol{0.2}$ & 0.1 & $\boldsymbol{0.4}$ & 0.1 & $\boldsymbol{0.9}$ & 0.2 & $\boldsymbol{1.7}$ & 0.3 & $\boldsymbol{3.6}$ & 0.5 & $\boldsymbol{7.4}$ & 0.9 & $\boldsymbol{15.3}$ & 1.8 & $\boldsymbol{30.8}$ & 3.9 & $\boldsymbol{62.1}$ & 8.3 & $\boldsymbol{126.4}$ & 18.4 & $\boldsymbol{255.9}$ & 45.1 & $\boldsymbol{519.7}$ & 98.3 & $\boldsymbol{1060.7}$ & 233.1 & $\boldsymbol{2143.8}$ & 539.3 & $\boldsymbol{4366.4}$ & 1215.6 \\
    & 0.5 & $\boldsymbol{0.2}$ & 0.1 & $\boldsymbol{0.4}$ & 0.1 & $\boldsymbol{0.9}$ & 0.2 & $\boldsymbol{1.7}$ & 0.3 & $\boldsymbol{3.6}$ & 0.5 & $\boldsymbol{7.5}$ & 0.9 & $\boldsymbol{15.2}$ & 1.9 & $\boldsymbol{30.5}$ & 4.0 & $\boldsymbol{62.1}$ & 8.4 & $\boldsymbol{126.4}$ & 18.1 & $\boldsymbol{255.4}$ & 44.2 & $\boldsymbol{525.1}$ & 98.4 & $\boldsymbol{1061.1}$ & 226.1 & $\boldsymbol{2130.7}$ & 523.8 & $\boldsymbol{4354.4}$ & 1187.6 \\
    & 0.6 & $\boldsymbol{0.2}$ & 0.1 & $\boldsymbol{0.4}$ & 0.1 & $\boldsymbol{0.9}$ & 0.2 & $\boldsymbol{1.7}$ & 0.3 & $\boldsymbol{3.6}$ & 0.4 & $\boldsymbol{7.4}$ & 0.9 & $\boldsymbol{15.1}$ & 1.8 & $\boldsymbol{30.6}$ & 3.8 & $\boldsymbol{61.9}$ & 8.1 & $\boldsymbol{126.7}$ & 17.3 & $\boldsymbol{255.7}$ & 42.5 & $\boldsymbol{521.9}$ & 94.7 & $\boldsymbol{1079.2}$ & 214.9 & $\boldsymbol{2133.5}$ & 499.4 & $\boldsymbol{4370.3}$ & 1113.7 \\
    & 0.7 & $\boldsymbol{0.2}$ & 0.1 & $\boldsymbol{0.4}$ & 0.1 & $\boldsymbol{0.9}$ & 0.2 & $\boldsymbol{1.8}$ & 0.2 & $\boldsymbol{3.6}$ & 0.4 & $\boldsymbol{7.3}$ & 0.8 & $\boldsymbol{15.2}$ & 1.7 & $\boldsymbol{30.4}$ & 3.6 & $\boldsymbol{62.9}$ & 7.7 & $\boldsymbol{126.4}$ & 16.6 & $\boldsymbol{254.9}$ & 40.8 & $\boldsymbol{513.9}$ & 88.1 & $\boldsymbol{1065.7}$ & 203.0 & $\boldsymbol{2132.3}$ & 465.5 & $\boldsymbol{4260.4}$ & 1040.3 \\
    & 0.8 & $\boldsymbol{0.2}$ & 0.1 & $\boldsymbol{0.4}$ & 0.1 & $\boldsymbol{0.9}$ & 0.2 & $\boldsymbol{1.7}$ & 0.2 & $\boldsymbol{3.6}$ & 0.4 & $\boldsymbol{7.4}$ & 0.7 & $\boldsymbol{15.3}$ & 1.6 & $\boldsymbol{31.0}$ & 3.3 & $\boldsymbol{62.0}$ & 7.2 & $\boldsymbol{125.5}$ & 15.4 & $\boldsymbol{257.3}$ & 38.2 & $\boldsymbol{503.4}$ & 83.4 & $\boldsymbol{1065.8}$ & 191.5 & $\boldsymbol{2086.9}$ & 431.5 & $\boldsymbol{4386.2}$ & 941.7 \\
    & 0.9 & $\boldsymbol{0.2}$ & 0.1 & $\boldsymbol{0.4}$ & 0.1 & $\boldsymbol{0.9}$ & 0.2 & $\boldsymbol{1.7}$ & 0.2 & $\boldsymbol{3.7}$ & 0.4 & $\boldsymbol{7.4}$ & 0.7 & $\boldsymbol{15.1}$ & 1.5 & $\boldsymbol{30.6}$ & 3.2 & $\boldsymbol{62.0}$ & 6.8 & $\boldsymbol{126.4}$ & 14.7 & $\boldsymbol{256.8}$ & 35.3 & $\boldsymbol{521.8}$ & 78.3 & $\boldsymbol{1067.9}$ & 179.4 & $\boldsymbol{2112.3}$ & 406.9 & $\boldsymbol{4251.2}$ & 905.5 \\
    \hline
    \multirow{9}{*}{AH} & 0.1 & $\boldsymbol{0.5}$ & 0.2 & $\boldsymbol{1.0}$ & 0.7 & $\boldsymbol{2.1}$ & 1.7 & $\boldsymbol{4.2}$ & 3.9 & $\boldsymbol{8.4}$ & 8.6 & $\boldsymbol{16.6}$ & 20.0 & $\boldsymbol{32.5}$ & 44.0 & $\boldsymbol{65.7}$ & 93.6 & $\boldsymbol{133.6}$ & 215.2 & $\boldsymbol{267.6}$ & 510.9 & NA & NA & NA & NA & NA & NA & NA & NA & NA & NA \\
    & 0.2 & $\boldsymbol{0.5}$ & 0.3 & $\boldsymbol{1.0}$ & 0.8 & $\boldsymbol{2.1}$ & 2.3 & $\boldsymbol{4.2}$ & 5.1 & $\boldsymbol{8.2}$ & 11.3 & $\boldsymbol{16.6}$ & 26.9 & $\boldsymbol{32.6}$ & 59.4 & $\boldsymbol{65.8}$ & 123.4 & $\boldsymbol{133.7}$ & 281.7 & $\boldsymbol{269.2}$ & 666.1 & NA & NA & NA & NA & NA & NA & NA & NA & NA & NA \\
    & 0.3 & $\boldsymbol{0.5}$ & 0.3 & $\boldsymbol{1.0}$ & 0.9 & $\boldsymbol{2.1}$ & 2.5 & $\boldsymbol{4.2}$ & 5.6 & $\boldsymbol{8.3}$ & 13.0 & $\boldsymbol{16.2}$ & 30.7 & $\boldsymbol{32.5}$ & 69.1 & $\boldsymbol{65.9}$ & 143.6 & $\boldsymbol{133.1}$ & 319.1 & $\boldsymbol{269.2}$ & 726.3 & NA & NA & NA & NA & NA & NA & NA & NA & NA & NA \\
    & 0.4 & $\boldsymbol{0.5}$ & 0.3 & $\boldsymbol{1.0}$ & 0.9 & $\boldsymbol{2.1}$ & 2.6 & $\boldsymbol{4.2}$ & 5.6 & $\boldsymbol{8.4}$ & 12.9 & $\boldsymbol{16.3}$ & 32.4 & $\boldsymbol{32.6}$ & 72.8 & $\boldsymbol{65.8}$ & 148.1 & $\boldsymbol{134.6}$ & 337.6 & $\boldsymbol{268.5}$ & 727.0 & NA & NA & NA & NA & NA & NA & NA & NA & NA & NA \\
    & 0.5 & $\boldsymbol{0.5}$ & 0.4 & $\boldsymbol{1.0}$ & 0.9 & $\boldsymbol{2.1}$ & 2.6 & $\boldsymbol{4.2}$ & 5.4 & $\boldsymbol{8.4}$ & 12.8 & $\boldsymbol{16.3}$ & 31.7 & $\boldsymbol{32.6}$ & 70.8 & $\boldsymbol{65.7}$ & 148.1 & $\boldsymbol{132.8}$ & 327.0 & $\boldsymbol{269.4}$ & 712.1 & NA & NA & NA & NA & NA & NA & NA & NA & NA & NA \\
    & 0.6 & $\boldsymbol{0.5}$ & 0.3 & $\boldsymbol{1.0}$ & 0.8 & $\boldsymbol{2.1}$ & 2.5 & $\boldsymbol{4.3}$ & 5.3 & $\boldsymbol{8.4}$ & 12.2 & $\boldsymbol{16.4}$ & 29.4 & $\boldsymbol{32.5}$ & 66.4 & $\boldsymbol{65.9}$ & 137.2 & $\boldsymbol{133.3}$ & 316.1 & $\boldsymbol{268.1}$ & 663.7 & NA & NA & NA & NA & NA & NA & NA & NA & NA & NA \\
    & 0.7 & $\boldsymbol{0.5}$ & 0.3 & $\boldsymbol{1.0}$ & 0.8 & $\boldsymbol{2.1}$ & 2.3 & $\boldsymbol{4.2}$ & 5.0 & $\boldsymbol{8.4}$ & 11.4 & $\boldsymbol{16.3}$ & 27.9 & $\boldsymbol{32.1}$ & 62.3 & $\boldsymbol{66.0}$ & 130.6 & $\boldsymbol{133.7}$ & 288.9 & $\boldsymbol{269.2}$ & 609.4 & NA & NA & NA & NA & NA & NA & NA & NA & NA & NA \\
    & 0.8 & $\boldsymbol{0.5}$ & 0.3 & $\boldsymbol{1.0}$ & 0.7 & $\boldsymbol{2.1}$ & 2.1 & $\boldsymbol{4.2}$ & 4.6 & $\boldsymbol{8.5}$ & 10.4 & $\boldsymbol{16.5}$ & 25.4 & $\boldsymbol{32.5}$ & 56.9 & $\boldsymbol{66.1}$ & 121.4 & $\boldsymbol{133.7}$ & 269.6 & $\boldsymbol{268.8}$ & 572.6 & NA & NA & NA & NA & NA & NA & NA & NA & NA & NA \\
    & 0.9 & $\boldsymbol{0.5}$ & 0.3 & $\boldsymbol{1.0}$ & 0.7 & $\boldsymbol{2.1}$ & 2.0 & $\boldsymbol{4.3}$ & 4.2 & $\boldsymbol{8.4}$ & 9.7 & $\boldsymbol{16.5}$ & 23.1 & $\boldsymbol{32.4}$ & 52.0 & $\boldsymbol{65.9}$ & 110.6 & $\boldsymbol{132.5}$ & 245.5 & $\boldsymbol{263.7}$ & 536.2 & NA & NA & NA & NA & NA & NA & NA & NA & NA & NA \\
    \hline
    \multirow{9}{*}{MAH} & 0.1 & $\boldsymbol{0.5}$ & 0.0 & $\boldsymbol{1.0}$ & 0.0 & $\boldsymbol{2.1}$ & 0.1 & $\boldsymbol{4.3}$ & 0.1 & $\boldsymbol{8.4}$ & 0.3 & $\boldsymbol{16.6}$ & 0.7 & $\boldsymbol{32.4}$ & 1.4 & $\boldsymbol{66.9}$ & 2.9 & $\boldsymbol{133.8}$ & 6.3 & $\boldsymbol{262.3}$ & 14.0 & $\boldsymbol{535.5}$  & 30.2   & $\boldsymbol{1075.0}$  & 60.2   & $\boldsymbol{2107.0}$  & 146.5   & $\boldsymbol{4271.5}$  & 300.8   & $\boldsymbol{8456.6}$  & 725.9 \\
    & 0.2 & $\boldsymbol{0.5}$ & 0.0 & $\boldsymbol{1.0}$ & 0.0 & $\boldsymbol{2.1}$ & 0.1 & $\boldsymbol{4.2}$ & 0.2 & $\boldsymbol{8.4}$ & 0.3 & $\boldsymbol{16.5}$ & 0.7 & $\boldsymbol{32.6}$ & 1.5 & $\boldsymbol{66.6}$ & 3.6 & $\boldsymbol{133.6}$ & 7.5 & $\boldsymbol{263.3}$ & 17.0 & $\boldsymbol{533.0}$  & 32.7   & $\boldsymbol{1070.6}$  & 74.7   & $\boldsymbol{2105.0}$  & 177.2   & $\boldsymbol{4246.7}$  & 373.6   & $\boldsymbol{8554.4}$  & 857.0 \\
    & 0.3 & $\boldsymbol{0.5}$ & 0.0 & $\boldsymbol{1.0}$ & 0.0 & $\boldsymbol{2.1}$ & 0.1 & $\boldsymbol{4.2}$ & 0.2 & $\boldsymbol{8.3}$ & 0.3 & $\boldsymbol{16.5}$ & 0.7 & $\boldsymbol{32.2}$ & 1.6 & $\boldsymbol{66.5}$ & 4.0 & $\boldsymbol{133.7}$ & 8.1 & $\boldsymbol{262.0}$ & 18.0 & $\boldsymbol{531.8}$  & 36.5   & $\boldsymbol{1072.3}$  & 83.1   & $\boldsymbol{2111.5}$  & 188.4   & $\boldsymbol{4233.8}$  & 415.8   & $\boldsymbol{8460.6}$  & 931.1 \\
    & 0.4 & $\boldsymbol{0.5}$ & 0.0 & $\boldsymbol{1.0}$ & 0.0 & $\boldsymbol{2.1}$ & 0.1 & $\boldsymbol{4.2}$ & 0.2 & $\boldsymbol{8.4}$ & 0.3 & $\boldsymbol{16.5}$ & 0.8 & $\boldsymbol{32.5}$ & 1.7 & $\boldsymbol{66.7}$ & 4.2 & $\boldsymbol{133.8}$ & 8.4 & $\boldsymbol{264.2}$ & 18.1 & $\boldsymbol{534.0}$  & 37.5   & $\boldsymbol{1070.5}$  & 85.3   & $\boldsymbol{2116.7}$  & 189.6   & $\boldsymbol{4245.9}$  & 426.4   & $\boldsymbol{8477.9}$  & 967.6 \\
    & 0.5 & $\boldsymbol{0.5}$ & 0.0 & $\boldsymbol{1.0}$ & 0.0 & $\boldsymbol{2.1}$ & 0.1 & $\boldsymbol{4.2}$ & 0.2 & $\boldsymbol{8.4}$ & 0.3 & $\boldsymbol{16.5}$ & 0.8 & $\boldsymbol{32.7}$ & 1.7 & $\boldsymbol{66.3}$ & 4.1 & $\boldsymbol{134.7}$ & 8.3 & $\boldsymbol{264.9}$ & 17.9 & $\boldsymbol{530.6}$  & 37.4   & $\boldsymbol{1080.4}$  & 84.6   & $\boldsymbol{2121.2}$  & 187.1   & $\boldsymbol{4242.7}$  & 423.2   & $\boldsymbol{8451.6}$  & 957.9 \\
    & 0.6 & $\boldsymbol{0.5}$ & 0.0 & $\boldsymbol{1.0}$ & 0.0 & $\boldsymbol{2.1}$ & 0.1 & $\boldsymbol{4.2}$ & 0.2 & $\boldsymbol{8.3}$ & 0.3 & $\boldsymbol{16.4}$ & 0.7 & $\boldsymbol{32.9}$ & 1.6 & $\boldsymbol{66.4}$ & 3.9 & $\boldsymbol{134.3}$ & 8.1 & $\boldsymbol{263.1}$ & 16.9 & $\boldsymbol{527.1}$  & 36.1   & $\boldsymbol{1077.2}$  & 79.8   & $\boldsymbol{2101.1}$  & 181.1   & $\boldsymbol{4200.3}$  & 403.1   & $\boldsymbol{8520.9}$  & 903.6 \\
    & 0.7 & $\boldsymbol{0.5}$ & 0.0 & $\boldsymbol{1.0}$ & 0.0 & $\boldsymbol{2.1}$ & 0.1 & $\boldsymbol{4.3}$ & 0.2 & $\boldsymbol{8.3}$ & 0.3 & $\boldsymbol{16.4}$ & 0.7 & $\boldsymbol{32.8}$ & 1.5 & $\boldsymbol{67.1}$ & 3.8 & $\boldsymbol{134.0}$ & 7.6 & $\boldsymbol{262.9}$ & 16.1 & $\boldsymbol{526.1}$  & 33.9   & $\boldsymbol{1075.1}$  & 76.3   & $\boldsymbol{2118.9}$  & 172.6   & $\boldsymbol{4216.5}$  & 389.7   & $\boldsymbol{8497.2}$  & 863.9 \\
    & 0.8 & $\boldsymbol{0.5}$ & 0.0 & $\boldsymbol{1.0}$ & 0.0 & $\boldsymbol{2.1}$ & 0.1 & $\boldsymbol{4.3}$ & 0.1 & $\boldsymbol{8.3}$ & 0.3 & $\boldsymbol{16.4}$ & 0.7 & $\boldsymbol{32.8}$ & 1.4 & $\boldsymbol{66.4}$ & 3.6 & $\boldsymbol{134.4}$ & 7.2 & $\boldsymbol{263.7}$ & 15.3 & $\boldsymbol{526.7}$  & 32.4   & $\boldsymbol{1077.9}$  & 72.3   & $\boldsymbol{2092.2}$  & 161.6   & $\boldsymbol{4196.4}$  & 367.1   & $\boldsymbol{8490.5}$  & 819.0 \\
    & 0.9 & $\boldsymbol{0.5}$ & 0.0 & $\boldsymbol{1.0}$ & 0.0 & $\boldsymbol{2.1}$ & 0.1 & $\boldsymbol{4.2}$ & 0.1 & $\boldsymbol{8.4}$ & 0.3 & $\boldsymbol{16.4}$ & 0.6 & $\boldsymbol{32.8}$ & 1.4 & $\boldsymbol{66.5}$ & 3.4 & $\boldsymbol{134.7}$ & 7.0 & $\boldsymbol{264.3}$ & 14.7 & $\boldsymbol{527.1}$  & 31.2   & $\boldsymbol{1080.7}$  & 69.6   & $\boldsymbol{2112.5}$  & 153.1   & $\boldsymbol{4230.6}$  & 344.3   & $\boldsymbol{8481.0}$  & 788.0 \\
    \hline
\end{tabular} 
\begin{tablenotes}
    \footnotesize
    \vspace{-0.3cm}
    \item \textit{Note.} Bold: coefficient generation time. \ins{Regular}: computation time. NA: results not included. CPU: our algorithm. AH: alternating hyperplane algorithm. MAH: modified alternating hyperplane algorithm. 
\end{tablenotes}
\end{threeparttable}
\end{adjustbox}
\label{tab_homo_detail}
\end{minipage}
}
\end{table}


\begin{table}[htbp]
    \centering 
    \rotatebox{270}{
    \begin{minipage}{\textheight}
    \caption{Detailed Time (s) for the Homogeneous Case, Greenville County.}
    \tiny
    \renewcommand\arraystretch{1.8}
    \tabcolsep=1pt
    \begin{adjustbox}{center}
    \begin{threeparttable}
        \begin{tabular}{@{\extracolsep{1.6pt}}cccccccccccccccccccccccccccccccc@{}}
        \hline
        \multirow{2}{*}{ } & \multirow{2}{*}{$\rho$} & \multicolumn{30}{c}{\# Units ($N$)} \\
        \cline{3-32}
        & & \multicolumn{2}{c}{11} & \multicolumn{2}{c}{12} & \multicolumn{2}{c}{13} & \multicolumn{2}{c}{14} & \multicolumn{2}{c}{15} & \multicolumn{2}{c}{16} & \multicolumn{2}{c}{17} & \multicolumn{2}{c}{18} & \multicolumn{2}{c}{19} & \multicolumn{2}{c}{20} & \multicolumn{2}{c}{21} & \multicolumn{2}{c}{22} & \multicolumn{2}{c}{23} & \multicolumn{2}{c}{24} & \multicolumn{2}{c}{25}\\
        \hline
        \multirow{9}{*}{CPU} & 0.1 & $\boldsymbol{0.3}$ & 0.1 & $\boldsymbol{0.6}$ & 0.1 & $\boldsymbol{1.2}$ & 0.1 & $\boldsymbol{2.4}$ & 0.1 & $\boldsymbol{4.8}$ & 0.2 & $\boldsymbol{9.7}$ & 0.3 & $\boldsymbol{20.1}$ & 0.7 & $\boldsymbol{40.5}$ & 1.5 & $\boldsymbol{82.7}$ & 3.6 & $\boldsymbol{166.4}$ & 8.9 & $\boldsymbol{339.4}$ & 19.9 & $\boldsymbol{687.5}$ & 43.7 & $\boldsymbol{1405.0}$ & 108.7 & $\boldsymbol{2859.1}$ & 246.1 & $\boldsymbol{5690.8}$ & 526.1\\
        & 0.2 & $\boldsymbol{0.3}$ & 0.1 & $\boldsymbol{0.6}$ & 0.1 & $\boldsymbol{1.2}$ & 0.1 & $\boldsymbol{2.4}$ & 0.2 & $\boldsymbol{4.9}$ & 0.3 & $\boldsymbol{9.9}$ & 0.5 & $\boldsymbol{20.0}$ & 1.0 & $\boldsymbol{40.5}$ & 2.2 & $\boldsymbol{82.4}$ & 5.4 & $\boldsymbol{166.8}$ & 13.4 & $\boldsymbol{336.4}$ & 28.4 & $\boldsymbol{689.3}$ & 63.5 & $\boldsymbol{1382.9}$ & 156.3 & $\boldsymbol{2802.2}$ & 370.9 & $\boldsymbol{5687.1}$ & 792.2 \\
        & 0.3 & $\boldsymbol{0.3}$ & 0.1 & $\boldsymbol{0.6}$ & 0.1 & $\boldsymbol{1.2}$ & 0.1 & $\boldsymbol{2.4}$ & 0.2 & $\boldsymbol{4.8}$ & 0.3 & $\boldsymbol{9.9}$ & 0.6 & $\boldsymbol{20.1}$ & 1.2 & $\boldsymbol{40.7}$ & 2.7 & $\boldsymbol{82.0}$ & 6.6 & $\boldsymbol{166.3}$ & 16.2 & $\boldsymbol{335.1}$ & 35.3 & $\boldsymbol{688.0}$ & 78.0 & $\boldsymbol{1406.6}$ & 183.0 & $\boldsymbol{2823.4}$ & 463.5 & $\boldsymbol{5725.9}$ & 982.8 \\
        & 0.4 & $\boldsymbol{0.3}$ & 0.1 & $\boldsymbol{0.6}$ & 0.1 & $\boldsymbol{1.2}$ & 0.2 & $\boldsymbol{2.4}$ & 0.2 & $\boldsymbol{4.9}$ & 0.4 & $\boldsymbol{9.9}$ & 0.6 & $\boldsymbol{20.1}$ & 1.3 & $\boldsymbol{40.8}$ & 3.0 & $\boldsymbol{82.5}$ & 7.2 & $\boldsymbol{164.8}$ & 18.3 & $\boldsymbol{336.4}$ & 39.5 & $\boldsymbol{683.2}$ & 85.3 & $\boldsymbol{1400.9}$ & 202.8 & $\boldsymbol{2789.1}$ & 522.5 & $\boldsymbol{5723.8}$ & 1105.4 \\
        & 0.5 & $\boldsymbol{0.3}$ & 0.1 & $\boldsymbol{0.6}$ & 0.1 & $\boldsymbol{1.2}$ & 0.2 & $\boldsymbol{2.4}$ & 0.2 & $\boldsymbol{4.9}$ & 0.4 & $\boldsymbol{9.9}$ & 0.7 & $\boldsymbol{20.1}$ & 1.4 & $\boldsymbol{40.6}$ & 3.1 & $\boldsymbol{82.6}$ & 7.6 & $\boldsymbol{166.7}$ & 19.1 & $\boldsymbol{336.3}$ & 41.2 & $\boldsymbol{686.3}$ & 90.7 & $\boldsymbol{1397.4}$ & 206.8 & $\boldsymbol{2825.7}$ & 541.9 & $\boldsymbol{5733.1}$ & 1165.2 \\
        & 0.6 & $\boldsymbol{0.3}$ & 0.1 & $\boldsymbol{0.6}$ & 0.1 & $\boldsymbol{1.2}$ & 0.2 & $\boldsymbol{2.4}$ & 0.2 & $\boldsymbol{4.9}$ & 0.4 & $\boldsymbol{9.9}$ & 0.7 & $\boldsymbol{20.0}$ & 1.4 & $\boldsymbol{40.5}$ & 3.1 & $\boldsymbol{82.2}$ & 7.4 & $\boldsymbol{166.2}$ & 18.7 & $\boldsymbol{338.0}$ & 41.2 & $\boldsymbol{683.4}$ & 89.0 & $\boldsymbol{1412.1}$ & 211.2 & $\boldsymbol{2812.2}$ & 533.9 & $\boldsymbol{5673.4}$ & 1149.3 \\
        & 0.7 & $\boldsymbol{0.3}$ & 0.1 & $\boldsymbol{0.6}$ & 0.1 & $\boldsymbol{1.2}$ & 0.2 & $\boldsymbol{2.4}$ & 0.2 & $\boldsymbol{4.8}$ & 0.4 & $\boldsymbol{10.0}$ & 0.6 & $\boldsymbol{20.0}$ & 1.3 & $\boldsymbol{40.4}$ & 3.0 & $\boldsymbol{82.1}$ & 7.1 & $\boldsymbol{166.6}$ & 17.9 & $\boldsymbol{334.7}$ & 39.5 & $\boldsymbol{686.8}$ & 87.2 & $\boldsymbol{1407.4}$ & 202.8 & $\boldsymbol{2840.0}$ & 499.8 & $\boldsymbol{5650.6}$ & 1080.4 \\
        & 0.8 & $\boldsymbol{0.3}$ & 0.1 & $\boldsymbol{0.6}$ & 0.1 & $\boldsymbol{1.2}$ & 0.1 & $\boldsymbol{2.4}$ & 0.2 & $\boldsymbol{4.9}$ & 0.4 & $\boldsymbol{9.9}$ & 0.6 & $\boldsymbol{20.1}$ & 1.3 & $\boldsymbol{40.3}$ & 2.9 & $\boldsymbol{82.0}$ & 6.9 & $\boldsymbol{166.5}$ & 17.1 & $\boldsymbol{336.9}$ & 36.9 & $\boldsymbol{692.2}$ & 81.7 & $\boldsymbol{1408.4}$ & 191.1 & $\boldsymbol{2784.5}$ & 463.6 & $\boldsymbol{5712.9}$ & 1009.8 \\
        & 0.9 & $\boldsymbol{0.3}$ & 0.1 & $\boldsymbol{0.6}$ & 0.1 & $\boldsymbol{1.2}$ & 0.1 & $\boldsymbol{2.4}$ & 0.2 & $\boldsymbol{4.8}$ & 0.3 & $\boldsymbol{9.8}$ & 0.6 & $\boldsymbol{20.0}$ & 1.2 & $\boldsymbol{40.7}$ & 2.8 & $\boldsymbol{82.6}$ & 6.6 & $\boldsymbol{166.8}$ & 16.3 & $\boldsymbol{335.5}$ & 35.2 & $\boldsymbol{679.1}$ & 78.1 & $\boldsymbol{1403.4}$ & 180.5 & $\boldsymbol{2814.8}$ & 430.5 & $\boldsymbol{5831.0}$ & 937.5 \\
        \hline

        \multirow{9}{*}{AH} & 0.1 & $\boldsymbol{0.6}$   & 0.3   & $\boldsymbol{1.3}$   & 0.7   & $\boldsymbol{2.5}$   & 1.6   & $\boldsymbol{5.3}$   & 3.8   & $\boldsymbol{10.7}$   & 8.5   & $\boldsymbol{21.5}$   & 18.9   & $\boldsymbol{42.9}$   & 40.8   & $\boldsymbol{86.5}$   & 90.0   & $\boldsymbol{171.1}$   & 217.2  & $\boldsymbol{354.5}$ & 483.2 & NA & NA & NA & NA & NA & NA & NA & NA & NA & NA \\
        & 0.2 & $\boldsymbol{0.6}$   & 0.4   & $\boldsymbol{1.3}$   & 0.8   & $\boldsymbol{2.6}$   & 2.1   & $\boldsymbol{5.4}$   & 4.9   & $\boldsymbol{10.3}$   & 10.9   & $\boldsymbol{21.4}$   & 23.0   & $\boldsymbol{43.3}$   & 52.7   & $\boldsymbol{86.6}$   & 119.1   & $\boldsymbol{173.0}$   & 291.8 & $\boldsymbol{353.4}$ & 658.1 & NA & NA & NA & NA & NA & NA & NA & NA & NA & NA \\
        & 0.3 & $\boldsymbol{0.7}$  & 0.5   & $\boldsymbol{1.3}$   & 1.0   & $\boldsymbol{2.5}$   & 2.4   & $\boldsymbol{5.1}$   & 5.1   & $\boldsymbol{10.5}$   & 11.6   & $\boldsymbol{21.2}$   & 25.8   & $\boldsymbol{42.8}$   & 60.9   & $\boldsymbol{88.6}$  & 138.0   & $\boldsymbol{172.6}$   & 337.7  & $\boldsymbol{356.0}$ & 748.1 & NA & NA & NA & NA & NA & NA & NA & NA & NA & NA \\
        & 0.4 & $\boldsymbol{0.6}$   & 0.5   & $\boldsymbol{1.3}$   & 1.0   & $\boldsymbol{2.6}$   & 2.4   & $\boldsymbol{5.5}$   & 5.9   & $\boldsymbol{10.6}$   & 13.2   & $\boldsymbol{21.7}$   & 28.5   & $\boldsymbol{43.1}$   & 65.3   & $\boldsymbol{87.6}$   & 146.8   & $\boldsymbol{172.8}$   & 357.3 & $\boldsymbol{350.8}$ & 788.3 & NA & NA & NA & NA & NA & NA & NA & NA & NA & NA \\
        & 0.5 & $\boldsymbol{0.6}$   & 0.5   & $\boldsymbol{1.3}$   & 1.0   & $\boldsymbol{2.6}$   & 2.6   & $\boldsymbol{5.2}$  & 5.6   & $\boldsymbol{10.2}$   & 12.5   & $\boldsymbol{21.5}$   & 28.3   & $\boldsymbol{43.3}$   & 65.5   & $\boldsymbol{86.9}$   & 141.0   & $\boldsymbol{172.7}$   & 361.8 & $\boldsymbol{350.3}$ & 814.0 & NA & NA & NA & NA & NA & NA & NA & NA & NA & NA \\
        & 0.6 & $\boldsymbol{0.6}$   & 0.5   & $\boldsymbol{1.2}$   & 1.0   & $\boldsymbol{2.6}$   & 2.5   & $\boldsymbol{5.3}$   & 5.4   & $\boldsymbol{10.5}$   & 11.6   & $\boldsymbol{21.0}$   & 26.9   & $\boldsymbol{43.8}$   & 62.9   & $\boldsymbol{85.5}$   & 113.2   & $\boldsymbol{178.7}$   & 356.8 & $\boldsymbol{350.4}$ & 767.2 & NA & NA & NA & NA & NA & NA & NA & NA & NA & NA \\
        & 0.7 & $\boldsymbol{0.6}$   & 0.4   & $\boldsymbol{1.3}$  & 0.9   & $\boldsymbol{2.6}$   & 2.0   & $\boldsymbol{5.2}$   & 5.2   & $\boldsymbol{10.8}$   & 12.0   & $\boldsymbol{21.5}$   & 25.7   & $\boldsymbol{42.9}$   & 59.1   & $\boldsymbol{75.5}$   & 114.9   & $\boldsymbol{175.3}$   & 332.2 & $\boldsymbol{351.6}$ & 730.6 & NA & NA & NA & NA & NA & NA & NA & NA & NA & NA \\
        & 0.8 & $\boldsymbol{0.7}$   & 0.4   & $\boldsymbol{1.2}$   & 0.8   & $\boldsymbol{2.5}$   & 2.2   & $\boldsymbol{5.4}$   & 5.0   & $\boldsymbol{10.1}$   & 10.8   & $\boldsymbol{21.6}$   & 25.5   & $\boldsymbol{43.2}$   & 55.5   & $\boldsymbol{83.1}$   & 121.7   & $\boldsymbol{172.5}$  & 296.1  & $\boldsymbol{355.7}$ & 686.8 & NA & NA & NA & NA & NA & NA & NA & NA & NA & NA \\
        & 0.9 & $\boldsymbol{0.6}$   & 0.4   & $\boldsymbol{1.1}$  & 0.9   & $\boldsymbol{2.6}$   & 2.1   & $\boldsymbol{4.9}$   & 4.7   & $\boldsymbol{10.6}$   & 10.9   & $\boldsymbol{21.0}$   & 23.7   & $\boldsymbol{42.7}$   & 53.7   & $\boldsymbol{85.4}$   & 114.1   & $\boldsymbol{172.9}$   & 278.3 & $\boldsymbol{350.9}$ & 570.4 & NA & NA & NA & NA & NA & NA & NA & NA & NA & NA \\
        \hline
        \multirow{9}{*}{MAH} & 0.1 & $\boldsymbol{0.6}$  & 0.0   & $\boldsymbol{1.3}$  & 0.0   & $\boldsymbol{2.5}$  & 0.1   & $\boldsymbol{5.1}$ & 0.1   & $\boldsymbol{10.2}$  & 0.3   & $\boldsymbol{20.6}$  & 0.6   & $\boldsymbol{41.7}$  & 1.2   & $\boldsymbol{84.1}$  & 2.6   & $\boldsymbol{168.4}$  & 5.8 & $\boldsymbol{341.7}$  & 12.2   & $\boldsymbol{684.7}$  & 27.3   & $\boldsymbol{1409.3}$  & 58.4   & $\boldsymbol{2648.4}$  & 126.6   & $\boldsymbol{5243.0}$  & 317.3   & $\boldsymbol{12649.3}$  & 778.0 \\
        & 0.2 & $\boldsymbol{0.6}$  & 0.0   & $\boldsymbol{1.2}$  & 0.0   & $\boldsymbol{2.5}$  & 0.1   & $\boldsymbol{5.1}$  & 0.1   & $\boldsymbol{10.3}$  & 0.3   & $\boldsymbol{20.7}$  & 0.6   & $\boldsymbol{41.8}$  & 1.2   & $\boldsymbol{84.5}$  & 3.1   & $\boldsymbol{168.1}$  & 6.7  & $\boldsymbol{341.0}$  & 14.8   & $\boldsymbol{688.3}$  & 32.5   & $\boldsymbol{1411.9}$  & 71.2   & $\boldsymbol{2651.7}$  & 153.9   & $\boldsymbol{5264.2}$  & 386.7   & $\boldsymbol{12561.3}$  & 903.0 \\
        & 0.3 & $\boldsymbol{0.6}$ & 0.0   & $\boldsymbol{1.2}$  & 0.0   & $\boldsymbol{2.5}$  & 0.1   & $\boldsymbol{5.1}$  & 0.1   & $\boldsymbol{10.2}$  & 0.3   & $\boldsymbol{20.5}$  & 0.6   & $\boldsymbol{41.8}$  & 1.4   & $\boldsymbol{84.3}$  & 3.4   & $\boldsymbol{168.1}$  & 7.4 & $\boldsymbol{340.8}$  & 16.3   & $\boldsymbol{686.4}$  & 35.5   & $\boldsymbol{1418.6}$  & 77.7   & $\boldsymbol{2649.1}$  & 169.2   & $\boldsymbol{5225.6}$  & 421.6   & $\boldsymbol{12571.4}$  & 993.3  \\
        & 0.4 & $\boldsymbol{0.6}$  & 0.0   & $\boldsymbol{1.3}$  & 0.0   & $\boldsymbol{2.5}$  & 0.1   & $\boldsymbol{5.1}$  & 0.1   & $\boldsymbol{10.2}$  & 0.3   & $\boldsymbol{21.0}$  & 0.6   & $\boldsymbol{41.7}$  & 1.4   & $\boldsymbol{84.1}$  & 3.5   & $\boldsymbol{167.9}$  & 7.8 & $\boldsymbol{340.8}$  & 17.1   & $\boldsymbol{685.1}$  & 36.7   & $\boldsymbol{1418.5}$  & 80.8   & $\boldsymbol{2656.6}$  & 178.2   & $\boldsymbol{5339.2}$  & 438.6   & $\boldsymbol{12494.7}$  & 1047.8  \\
        & 0.5 & $\boldsymbol{0.6}$  & 0.0   & $\boldsymbol{1.3}$  & 0.0   & $\boldsymbol{2.5}$  & 0.1   & $\boldsymbol{5.1}$  & 0.1   & $\boldsymbol{10.2}$  & 0.3   & $\boldsymbol{20.7}$  & 0.6   & $\boldsymbol{41.6}$  & 1.4   & $\boldsymbol{84.3}$  & 3.5   & $\boldsymbol{168.2}$  & 7.8  & $\boldsymbol{340.9}$  & 17.4   & $\boldsymbol{686.8}$  & 36.7   & $\boldsymbol{1417.5}$  & 80.8   & $\boldsymbol{2636.6}$  & 181.1   & $\boldsymbol{5260.8}$  & 438.5   & $\boldsymbol{12527.6}$  & 1047.6 \\
        & 0.6 & $\boldsymbol{0.6}$  & 0.0   & $\boldsymbol{1.3}$  & 0.0   & $\boldsymbol{2.5}$  & 0.1   & $\boldsymbol{5.1}$  & 0.1   & $\boldsymbol{10.2}$  & 0.3   & $\boldsymbol{20.5}$  & 0.6   & $\boldsymbol{42.0}$  & 1.4   & $\boldsymbol{84.5}$  & 3.4   & $\boldsymbol{168.0}$  & 7.6 & $\boldsymbol{342.5}$  & 16.9   & $\boldsymbol{687.4}$  & 35.5   & $\boldsymbol{1415.7}$  & 78.5   & $\boldsymbol{2643.4}$  & 177.9   & $\boldsymbol{5272.8}$  & 423.4   & $\boldsymbol{12547.9}$  & 1017.7  \\
        & 0.7 & $\boldsymbol{0.6}$  & 0.0   & $\boldsymbol{1.3}$  & 0.0   & $\boldsymbol{2.5}$  & 0.1   & $\boldsymbol{5.1}$  & 0.1   & $\boldsymbol{10.2}$  & 0.3   & $\boldsymbol{20.6}$  & 0.6   & $\boldsymbol{41.7}$  & 1.3   & $\boldsymbol{84.5}$  & 3.3   & $\boldsymbol{168.3}$  & 7.3 & $\boldsymbol{342.9}$  & 16.3   & $\boldsymbol{686.3}$  & 34.2   & $\boldsymbol{1418.4}$  & 75.7   & $\boldsymbol{2647.9}$  & 170.4   & $\boldsymbol{5262.0}$  & 408.2   & $\boldsymbol{12553.3}$  & 982.0  \\
        & 0.8 & $\boldsymbol{0.6}$  & 0.0   & $\boldsymbol{1.3}$  & 0.0   & $\boldsymbol{2.5}$  & 0.1   & $\boldsymbol{5.1}$  & 0.1   & $\boldsymbol{10.2}$  & 0.3   & $\boldsymbol{20.7}$  & 0.6   & $\boldsymbol{41.8}$  & 1.3   & $\boldsymbol{85.4}$  & 3.2   & $\boldsymbol{167.9}$  & 6.9 & $\boldsymbol{342.6}$  & 15.4   & $\boldsymbol{686.3}$  & 32.3   & $\boldsymbol{1414.9}$  & 71.7   & $\boldsymbol{2647.9}$  & 163.8   & $\boldsymbol{5258.7}$  & 385.0   & $\boldsymbol{12681.9}$  & 929.6  \\
        & 0.9 & $\boldsymbol{0.6}$  & 0.0   & $\boldsymbol{1.3}$  & 0.0   & $\boldsymbol{2.5}$  & 0.1   & $\boldsymbol{5.1}$  & 0.1   & $\boldsymbol{10.2}$  & 0.3   & $\boldsymbol{20.7}$  & 0.6   & $\boldsymbol{41.5}$  & 1.2   & $\boldsymbol{84.3}$  & 3.1   & $\boldsymbol{168.5}$  & 6.6 & $\boldsymbol{343.5}$  & 14.6   & $\boldsymbol{685.3}$  & 31.0   & $\boldsymbol{1419.6}$  & 67.6   & $\boldsymbol{2643.7}$  & 157.1   & $\boldsymbol{5258.4}$  & 369.5   & $\boldsymbol{12516.2}$  & 877.3  \\
        \hline
    \end{tabular}
    \begin{tablenotes}
    \footnotesize
    \vspace{-0.3cm}
    \item \textit{Note.} Bold: coefficient generation time. Regular: computation time.  NA: results not included. CPU: our algorithm. AH: alternating hyperplane algorithm. MAH: modified alternating hyperplane algorithm. 
    \end{tablenotes}
    \end{threeparttable}
    \end{adjustbox}
    \label{tab_greenville_homo_detail}
    \vspace{-0.3cm}
    \end{minipage}
    }
\end{table}

\end{appendices}

\end{document}